\documentclass[article,12pt]{elsarticle}

\usepackage[T1]{fontenc}
\usepackage{lmodern}
\usepackage{graphicx}
\usepackage{multirow}
\usepackage{amsmath,amssymb,amsfonts,amsthm,bm}
\usepackage{mathrsfs}
\usepackage[title]{appendix}
\usepackage{xcolor}
\usepackage{textcomp}
\usepackage{manyfoot}
\usepackage{booktabs,longtable,array}
\usepackage{algorithm}
\usepackage{algorithmicx}
\usepackage{algpseudocode}
\usepackage{capt-of}
\usepackage{listings}
\usepackage{microtype}
\usepackage{tikz}
\usetikzlibrary{arrows.meta,positioning}
\usepackage{hyperref}
\usepackage[utf8]{inputenc}

\theoremstyle{plain}
\newtheorem{theorem}{Theorem}
\newtheorem{proposition}{Proposition}

\theoremstyle{remark}

\newtheorem{remark}{Remark}
\theoremstyle{definition}
\newtheorem{definition}{Definition}
\newtheorem{assumption}{Assumption}

\biboptions{numbers,sort&compress}
\journal{Applied Mathematical Modelling}
\newenvironment{breakablealgorithm}[1]
{%
  \par\addvspace{\baselineskip}%
  \hrule height 0.8pt
  \nobreak\vskip 3pt
  \refstepcounter{algorithm}%
  \addcontentsline{loa}{algorithm}
    {\protect\numberline{\thealgorithm}#1}%
  \noindent\textbf{Algorithm~\thealgorithm}\enspace #1\par
  \nobreak\vskip 3pt
  \hrule height 0.4pt
  \nobreak\vskip 6pt
}
{%
  \par\nobreak\vskip 4pt
  \hrule height 0.8pt
  \addvspace{\baselineskip}%
}

\begin{document}
\begin{frontmatter}

\title{Hankel--Koopman Finite-Horizon Energy Decomposition of Coupled Experimental Data: A Three-Phase Data-Driven Twin Forecasting Framework}

\author[1]{Diana A. Bistrian\corref{cor1}}
\ead{diana.bistrian@upt.ro}
\author[1]{Marcel Topor}
\ead{marcel.topor@upt.ro}
\cortext[cor1]{Corresponding author}

\affiliation[1]{organization={Department of Electrical Engineering and Industrial Informatics, University Politehnica Timisoara},
                city={Timisoara},
                postcode={300006},
                country={Romania}}

\begin{abstract}
Finite-horizon prediction is challenging when observations are limited
and a quantity of interest depends on several coupled nonlinear
measurement channels. This paper develops a three-phase data-driven
twin framework that integrates finite-horizon Hankel--Koopman energy
decomposition, inverse-calibrated reconstruction using nonlinear
autoregressive models with exogenous inputs, and forecast
composition. Channel dynamics are identified and simulated in Hankel
coefficient space, with candidate models assessed after reconstruction
in physical measurement space. The resulting reduced-order model of the channel dynamics drives
recursive prediction of the quantity of interest.
Three theorems support the framework, establishing the energy
decomposition and modal orthogonality; existence, uniqueness, and
inverse stability of the auxiliary continuous inverse-calibration
solution; and finite-horizon stability of the composed forecast
for a fixed selected model under stated assumptions. A retrospective solar-power case study yields high correlations
between predicted and reference trajectories across the considered
forecast horizons, demonstrating reproduction of a constructed
channels-dependent quantity of interest.
\end{abstract}


\end{frontmatter}


\section*{Mathematical Notation}
 
\begingroup
\small
\renewcommand{\arraystretch}{1.10}
\setlength{\LTpre}{6pt}
\setlength{\LTpost}{12pt}

\begin{longtable}{@{}>{\raggedright\arraybackslash}p{0.28\textwidth}
                       >{\raggedright\arraybackslash}p{0.66\textwidth}@{}}
\toprule
\textbf{Symbol} & \textbf{Description}\\
\midrule
\endfirsthead

\multicolumn{2}{@{}l}{\textit{Mathematical notation (continued)}}\\[2pt]
\toprule
\textbf{Symbol} & \textbf{Description}\\
\midrule
\endhead

\midrule
\multicolumn{2}{r@{}}{\textit{Continued on next page}}\\
\endfoot

\bottomrule
\endlastfoot

$\mathbb F$, $m$, $N$ &
Scalar field ($\mathbb R$ or $\mathbb C$), number of measured
channels, and number of samples in the complete record.\\

$t_k$, $\Delta t$ &
Sampling instant and uniform sampling interval.\\

$\mathcal T_{\mathrm{fin}}$ &
Finite sampled-time set indexing the complete experimental record.\\

$\mathcal T_{\mathrm{obs}}$, $N_Q$ &
Observation horizon and number of paired quantity--channel
samples used for identification.\\

$\mathcal T_H$, $H$ &
Calibration window and number of calibration intervals;
the window contains $H+1$ samples and follows
$\mathcal T_{\mathrm{obs}}$.\\

$\mathcal T_{\mathrm{for}}$, $N_{\mathrm f}$ &
Forecast horizon and number of recursively forecast samples,
with $\mathcal T_{\mathrm{for}}\subseteq\mathcal T_H$.\\

$\bm y_k$, $\bm Y$, $\bm\varepsilon_k$ &
Measured channel vector, snapshot matrix
$\bm Y=[\bm y_1,\ldots,\bm y_N]$, and measurement noise
or unresolved effects.\\

$\bm y_H=\operatorname{vec}(\bm Y)$, $N_H=mN$ &
Columnwise serialized measurement record and its length.\\

$q$, $K=N_H-q+1$ &
Hankel delay depth and number of Hankel columns.\\

$\mathscr L_{q,K}$, $\mathscr H_{q,K}$ &
Hankel lifting operator and its range,
the finite-dimensional Hankel-lifted space.\\

$\bm h_\ell$, $\bm H$, $\bm H_0$, $\bm H_1$ &
Serialized delay vector, Hankel matrix, and its
one-serialized-index-shifted input and output matrices.\\

$\bm F_{\mathrm{ser}}$, $\mathcal K_{\mathrm{ser}}$ &
Channel-augmented serialized shift map and its
Koopman composition operator.\\

$g_{\mathrm{ser}}$, $\bm h_q$ &
Scalar serialized observable and associated
$q$-component delay observable.\\

$\bm K_H$, $\widetilde{\bm K}_p$ &
Exact matrix representation under delay-space invariance
and finite-data projected Koopman approximation, respectively.\\

$p$, $r$, $r^\star$ &
Data-supported rank of $\bm H_0$, retained modal dimension,
and selected retained dimension; reduced order corresponds to $r<p$.\\

$\bm U_p$, $\bm\Sigma_p$, $\bm V_p$ &
Rank-$p$ singular value decomposition factors of $\bm H_0$.\\

$\bm G_p=\widetilde{\bm K}_p^{\mathrm H}
          \widetilde{\bm K}_p$ &
Self-adjoint positive-semidefinite Koopman energy operator.\\

$\mu_j$, $\bm q_j$, $\bm Q$, $\bm\Lambda_G$ &
Energy-operator eigenvalue and orthonormal eigenvector,
eigenvector matrix, and diagonal eigenvalue matrix.\\

$\bm\Phi=\bm U_p\bm Q$, $\bm C=\bm\Phi^{\mathrm H}\bm H$ &
Orthogonal Hankel--Koopman energy-mode matrix
and associated coefficient matrix.\\

$\mathfrak t_j=(\mu_j,\bm\phi_j,\bm c_j^{\mathsf T})$ &
Modal triplet comprising an energy eigenvalue,
energy mode, and coefficient row.\\

$L$, $P_j^{(L)}$, $\alpha_j$, $E_j^{(L)}$ &
Operator horizon, persistence factor, coefficient energy,
and finite-horizon modal energy,
with $E_j^{(L)}=\alpha_jP_j^{(L)}$.\\

$\mathfrak T_{r^\star}$, $\widehat{\bm H}_{r^\star}$ &
Selected energy-ranked triplet family and its reduced
Hankel reconstruction
$\widehat{\bm H}_{r^\star}
=\bm\Phi_{r^\star}\bm C_{r^\star}$.\\

$\varepsilon_r$, $\gamma_r$, $\mathrm{CR}_r$ &
Relative Hankel reconstruction error, matrix cosine similarity,
and nominal factor-storage compression ratio.\\

$\widehat{\bm Y}$ &
Channel reconstruction obtained by anti-diagonal averaging,
inverse reshaping.\\

$\mathcal I_H$, $N_{\mathcal I}$, $K_{\mathcal I}$ &
Calibration-index set, number of physical calibration samples,
and number of local Hankel columns.\\

$\bm Y_{\mathcal I}$, $\bm H_{\mathcal I}$,
$\bm C_{\mathcal I}$ &
Measurement, Hankel, and selected coefficient matrices
associated with the calibration window.\\

$\bm\alpha=(n_a,n_b,n_k)$, $\bm B_{\bm\alpha}$ &
Phase~II NLARX order triple and multi-output parameter matrix.\\

$\mathscr S_{\mathcal I,\bm\alpha}$,
$\mathscr A_{\mathcal I}$,
$\mathscr R_{\mathcal I,\bm\alpha}$ &
Recursive coefficient simulator, anti-diagonal recovery
with inverse reshaping, and complete physical-space
reconstruction operator.\\

$\lambda_{\mathrm{reg}}$,
$\mathcal J_{\mathcal I,\lambda_{\mathrm{reg}}}$ &
Regularization parameter and auxiliary continuous
physical-space inverse-calibration functional.\\

$Q_{\mathrm I}$, $q_k$ &
Channel-dependent scalar quantity of interest
and its sampled value.\\

$\bm\beta=(n_a^Q,n_b^Q,n_k^Q)$, $\bm\theta_Q^\star$ &
Quantity-of-interest NLARX order triple
and selected parameter vector.\\

$\widehat{\bm Y}_{\mathrm{for}}$ &
Physical channel trajectories reconstructed in Phase~II
and supplied as exogenous forecast inputs.\\

$\mathcal F_{Q,h}^\star$, $\mathcal C_{Q,h}^\star$ &
Selected one-step quantity-of-interest evolution operator
and its ordered composition through forecast step $h$.\\

$\mathcal R_{Q/\mathrm{for}}=N_Q/N_{\mathrm f}$ &
Identification-to-forecast sample ratio.\\

\end{longtable}
\addtocounter{table}{-1}
\endgroup

\section{Introduction}\label{sec:introduction}

Coupled experimental measurements often reflect dynamics distributed
across channels, time scales, and delays. The challenge addressed in this paper
is to extract from a coupled finite record a compact, interpretable dynamical
representation that can be advanced in time, reconstructed in physical
measurement space, and used to forecast a quantity of interest.

\subsection{Foundational Developments in Data-Driven Dynamical Modeling}
\label{subsec:introduction_state_of_art}

Koopman's 1931 observation \cite{Koopman1931} established that a dynamical
system can be studied through functions of its state, called observables.
Koopman and von Neumann \cite{KoopmNeum1932} extended this viewpoint to
continuous spectra. Spectral computation remains challenging in the generally
infinite-dimensional observable space.

The modern spectral interpretation connected Koopman eigenvalues,
eigenfunctions, and modes with nonlinear dynamics and model reduction
\cite{MezicBanaszuk2004, Mezic2005}, including coherent-structure identification
\cite{Holmes1996,Rowley2009} and computational developments
\cite{MezicKoop2021,Mezic2022}.
Dynamic Mode Decomposition (DMD) \cite{Schmid2008,Schmid2010}
made modal analysis accessible from numerical and experimental data,
opening several directions
\cite{Rowley2010,Lee2013,Cueto2014,ProctorBruntonKutz2016,
colbrook2026,Ivagnes2026,Hajisharifi2026}.

\emph{Regression-, compression-, multiresolution-, and control-oriented modal methods.}
As DMD developed, attention turned to sampling flexibility and
computational cost.
Optimized DMD (optDMD) \cite{Chen2012} fits the full time record through
global exponential regression, later solved efficiently by variable
projection \cite{AskhamKutz2018}.
Generalized DMD \cite{Tu2014} accommodates broader sampling arrangements,
while Compressed DMD (cDMD) and Compressed-Sensing DMD (csDMD)
\cite{BruntonProctorTuKutz2016} reduce measurement requirements and,
under suitable assumptions, recover full-state modes.
Multiresolution DMD (mrDMD) \cite{KutzFuBrunton2016} separates localized
time scales, and DMD with Control (DMDc)
\cite{ProctorBruntonKutz2016} incorporates known inputs.
Randomized Low-Rank DMD (rDMD) \cite{ErichsonDonovan2016} and Adaptive
Randomized DMD (ArDMD) \cite{BistrianNavon2017} reduce basis-construction
costs for large snapshot matrices. Subsequent probabilistic formulations
\cite{ErichsonEtAl2019} control approximation quality through oversampling
and power iterations.

\emph{Galerkin- and Hankel-oriented modal methods.}
A related question concerns the coordinates in which nonlinear
dynamics are represented. Extended DMD (EDMD) \cite{Williams2015}
approximates the Koopman operator in an observable dictionary through
Galerkin projection, also central to projection-based model reduction
\cite{BennerGugercinWillcox2015}.
Delay coordinates construct observables from the available record.
Higher Order DMD (HODMD) \cite{LeClainche2017} resolves richer temporal
spectra from time-lagged snapshots when spatial complexity is limited.
Hankel-DMD \cite{ArbabiMezic2017} connects delays with Koopman Krylov
subspaces, while Hankel Alternative View of Koopman (HAVOK)
\cite{BruntonEtAl2017} represents chaotic dynamics through dominant
Hankel coordinates with intermittent forcing from a lower-energy coordinate.
Later work addressed finite-data prediction bounds \cite{Nueske2022}
and online delay-embedded identification \cite{Guo2023}.

\emph{Spectral verification and structure preservation in data-driven modal methods.}
Finite observable spaces also raise questions of spectral reliability
and structure preservation.
Residual DMD (ResDMD) \cite{ColbrookTownsend2023,ColbrookAytonSzoke2023}
uses operator residuals to assess computed spectra.
Physics-Informed DMD (piDMD) \cite{BaddooEtAl2023} imposes physical
constraints, while Measure-Preserving EDMD (mpEDMD) \cite{Colbrook2023}
retains the structure associated with an invariant measure.

\emph{Nonlinear input-output identification, reduced-order prediction, and data-driven twin modeling.}
Alongside these developments, learning methods began to construct
coordinates together with their evolution laws.
Neural networks learned observables for approximately linear dynamics
\cite{Takeishi2017}; autoencoders combined nonlinear coordinates with
linear latent evolution \cite{Lusch2018,Otto2019}; and consistent Koopman
autoencoders coupled forward and backward operators for multi-step
forecasting \cite{Azencot2020}.
Digital twins connect such models to physical processes through observations
and model updating
\cite{KapteynKnezevicWillcox2020AIAA,KapteynPretoriusWillcox2021,
KapteynEtAl2022}.
The digital twin data model in \cite{Bistrian2022} combined randomized DMD
with neural-network identification of temporal coefficients, recursive
simulation, and physical reconstruction.
Koopman Randomized Orthogonal Decomposition (KROD) \cite{Bistrian2025}
extended this line of work through orthogonal energy modes and explainable
deep learning for nonlinear PDEs.

\subsection{The Unresolved Methodological Gap}

Direct forecasting methods predict measured quantities from their
histories and available inputs
\cite{Hegazy2026,Sousa2026,Zito2025,Beletskaya2025,Eichholz2023}.
Although effective in many short-term applications, they may propagate
noise and redundancy while leaving cross-channel interactions implicit.
Selecting delays and inputs becomes harder as the number of channels grows,
and recursive errors may accumulate.

Deep-learning approaches learn spatio-temporal representations from historical
states and exogenous inputs
\cite{Napoles2026,Zhou2026,Lyu2025,Kong2025,Mojtahedi2025}.
These can support accurate prediction, but dynamics distributed across network
architectures, parameters, and training procedures can be difficult to
interpret physically and analyze mathematically.

These difficulties motivate a reduced dynamical representation
that retains the dominant coupled information and makes its
evolution explicit. Modal methods provide interpretable modes and
coefficients, but identifying closed nonlinear dynamics for the retained
coordinates remains a distinct task, even when linear or Koopman-based
evolution operators are available. The remaining gap concerns connecting
these nonlinear dynamics with physical reconstruction and finite-horizon
forecasting of a quantity of interest.

\subsection{Motivation for the Unified Three-Phase Framework}

To address this gap, the present work connects measurement
decomposition, nonlinear reduced dynamics, physical-space model
assessment, and forecasting through three computational phases.
Their integration is formulated on finite horizons and supported
by mathematical analysis under stated assumptions.

Phase~I introduces the \emph{Hankel--Koopman finite-horizon energy
decomposition (HKFED)}, constructing orthogonal energy modes and
ranking the associated modal triplets by finite-horizon energy
and persistence.
The resulting reduced coordinates provide the basis for Phase~II,
where a multi-output neural-network-based NLARX model is identified
and recursively simulated in Hankel coefficient space.
Candidate models are assessed after inverse Hankel reconstruction
in physical measurement space, giving the procedure its
\emph{inverse-calibrated} character.
The resulting reduced-order model of channel dynamics supplies
simulated channel inputs for Phase~III, which forecasts a quantity
of interest through a recursive NLARX model.

The principal contribution lies in coupling HKFED and
inverse-calibrated nonlinear modeling to channel-driven forecasting.
Three theorems support this integration: they establish the
finite-horizon energy decomposition; existence, uniqueness,
and inverse stability of the auxiliary continuous
inverse-calibration solution; and finite-horizon stability
of the composed forecast for a fixed selected model.

The remainder of the paper is organized as follows.
Section~\ref{sec:Mathback} presents the mathematical background.
Section~\ref{sec:computational_aspects} develops the three phases,
their algorithms, and supporting analysis.
Section~\ref{Num_exp} evaluates the framework through numerical
experiments, and Section~\ref{conclusion} discusses the main findings.

\section{Mathematical Background}\label{sec:Mathback}

\subsection{Hankel-Lifted Reduced-Order Modeling of Finite-Horizon Experimental Data}
\label{subsec:Data_ROM}

Correlations among experimental channels motivate a reduced-order
description that retains dominant measured information
\cite{Ruan2026,Tsai2026,Codega2026,Iliescu2025,Quarteroni2024,
Iliescu2022,Ahmed2020,Noack2011,Tissot2014}.
Here, Hankel lifting and modal decomposition construct a spectrally
truncated representation directly from finite experimental observations,
without a prescribed governing equation. The nonlinear evolution of
the retained coordinates is then identified for reconstruction
and forecasting.

\begin{definition}[Finite modeling horizon, calibration window, observation horizon, and forecast horizon]
\label{def:finite_observation_forecast_horizon}
Let the complete experimental record contain $N$ uniformly spaced sampling
instants
\begin{equation}
    t_k=t_0+(k-1)\Delta t,
    \qquad k=1,\ldots,N,
    \qquad \Delta t>0,
    \label{eq:finite_modeling_grid}
\end{equation}
and define the \emph{finite modeling horizon} by the finite sampled-time set
\begin{equation}
    \mathcal T_{\mathrm{fin}}
    =\{t_k:k=1,\ldots,N\},
    \qquad
    T_{\mathrm{fin}}=t_N-t_1=(N-1)\Delta t<\infty.
    \label{eq:finite_modeling_horizon}
\end{equation}
The complete experimental record is indexed by
$\mathcal T_{\mathrm{fin}}$.

Let $1\leq N_Q<N$ be the number of paired experimental
quantity--channel samples used for identification and define the
\emph{finite observation horizon}
\begin{equation}
    \mathcal T_{\mathrm{obs}}
    =\{t_k:k=1,\ldots,N_Q\},
    \qquad
    T_{\mathrm{obs}}=t_{N_Q}-t_1=(N_Q-1)\Delta t.
    \label{eq:finite_observation_horizon}
\end{equation}
The \emph{finite calibration window} begins at the next sampling
instant, $n_{\mathrm s}=N_Q+1$, and, for a prescribed number
$H\in\mathbb N_0$ of calibration intervals satisfying
$n_{\mathrm s}+H\leq N$, is
\begin{equation}
    \mathcal T_H
    =\{t_k:k=n_{\mathrm s},\ldots,n_{\mathrm s}+H\},
    \qquad
    T_{\mathrm{FH}}=H\Delta t.
    \label{eq:finite_calibration_window}
\end{equation}
For a prescribed $N_{\mathrm f}\in\mathbb N_0$ satisfying
$N_{\mathrm f}\leq H+1$, the \emph{finite forecast horizon} is
\begin{equation}
    \mathcal T_{\mathrm{for}}
    =\{t_{N_Q+h}:h=1,\ldots,N_{\mathrm f}\},
    \qquad
    T_{\mathrm f}=N_{\mathrm f}\Delta t,
    \label{eq:finite_forecast_horizon}
\end{equation}
with $\mathcal T_{\mathrm{for}}=\varnothing$ when $N_{\mathrm f}=0$.
Here, $T_{\mathrm f}$ measures the forecast lead time
from the last observation at $t_{N_Q}$.
Consequently, the four temporal sets satisfy
\begin{equation}
\begin{gathered}
    \mathcal T_{\mathrm{obs}}\cup\mathcal T_{\mathrm{for}}
    \subseteq\mathcal T_{\mathrm{fin}},
    \qquad
    \mathcal T_H\subseteq\mathcal T_{\mathrm{fin}},
    \qquad
    \mathcal T_{\mathrm{for}}\subseteq\mathcal T_H,\\
    \mathcal T_{\mathrm{obs}}\cap\mathcal T_H=\varnothing,
    \qquad
    \max\mathcal T_{\mathrm{obs}}<\min\mathcal T_H.
\end{gathered}
\label{eq:finite_temporal_set_relations}
\end{equation}
\end{definition}

Throughout this work, $\mathcal T_{\bullet}$ denotes a finite sampled-time
set, $T_{\bullet}$ its duration, and $N$, $N_Q$, $N_{\mathrm f}$, or $H$
the corresponding sample, forecast-step, or interval count.

Let
\begin{equation}
    \bm{y}_k
    = \bm{y}(t_k) + \bm{\varepsilon}_k
    \in \mathbb{F}^{m},
    \qquad
    t_k \in \mathcal T_{\mathrm{fin}},
    \qquad k=1,\ldots,N,
    \label{eq:experimental_measurements}
\end{equation}
denote the measurement vector recorded at time $t_k$, where
$\mathbb{F}=\mathbb{R}$ or $\mathbb{C}$, $m$ is the number of measured variables,
and $\bm{\varepsilon}_k$ represents measurement noise and unresolved
effects. The measurements are assembled columnwise in the snapshot matrix
\begin{equation}
    \bm{Y}
    = \begin{bmatrix}
        \bm{y}_1 & \bm{y}_2 & \cdots & \bm{y}_N
      \end{bmatrix}
    \in \mathbb{F}^{m\times N}.
    \label{eq:experimental_snapshot_matrix}
\end{equation}

We use the Euclidean inner product
$\langle\bm v,\bm w\rangle=\bm v^{\mathrm H}\bm w$
and its induced norm $\|\cdot\|_2$ for measurement vectors.
For a finite data matrix $\bm Z=[\bm z_1,\ldots,\bm z_N]$,
the Frobenius norm satisfies
$\|\bm Z\|_F^2=\sum_{k=1}^{N}\|\bm z_k\|_2^2$,
where $(\cdot)^{\mathrm H}$ denotes the conjugate transpose.

The multichannel measurement record is serialized directly as
\begin{equation}
    \bm y_H
    =\operatorname{vec}(\bm Y)
    =
    \begin{bmatrix}
       \bm y_1^{\mathsf T}&\cdots&\bm y_N^{\mathsf T}
    \end{bmatrix}^{\mathsf T}
    =
    \begin{bmatrix}
       y_{H,1}&\cdots&y_{H,N_H}
    \end{bmatrix}^{\mathsf T},
    \;
    N_H=mN.
    \label{eq:hkfed_number_hankel_columns}
\end{equation}
Thus, channel values retain a fixed order within each sample, and
samples are concatenated chronologically.

When $m\ll N$, direct decomposition of $\bm Y$ is limited by
$\operatorname{rank}(\bm Y)\leq m$.
Delayed-observable Hankel constructions \cite{ArbabiMezic2017}
allow a lifted dimension beyond the number of measured channels.
Here, scalar-entry serialization and delay lifting preserve every measurement
and provide richer temporal directions.

Let $q\in\{1,\ldots,N_H-1\}$ be the imposed \emph{Hankel delay depth} and set
$K=N_H-q+1\geq2$. The serialized delay vectors are
\begin{equation}
    \bm h_\ell
    =
    \begin{bmatrix}
       y_{H,\ell}&y_{H,\ell+1}&\cdots&y_{H,\ell+q-1}
    \end{bmatrix}^{\mathsf T}
    \in\mathbb F^q,
    \qquad \ell=1,\ldots,K,
    \label{eq:multichannel_delay_vector}
\end{equation}
and the corresponding lifted matrix is
\begin{equation}
    \bm H
    =
    \begin{bmatrix}
       \bm h_1&\bm h_2&\cdots&\bm h_K
    \end{bmatrix}
    =
    \begin{bmatrix}
       y_{H,1} & y_{H,2} & \cdots & y_{H,K}\\
       y_{H,2} & y_{H,3} & \cdots & y_{H,K+1}\\
       \vdots  & \vdots  & \ddots & \vdots\\
       y_{H,q} & y_{H,q+1} & \cdots & y_{H,N_H}
    \end{bmatrix}
    \in\mathbb F^{q\times K}.
    \label{eq:multichannel_block_hankel}
\end{equation}

\begin{definition}[Finite-dimensional Hankel-lifted space]
\label{def:finite_dimensional_hankel_lifted_space}
Let $q,K\in\mathbb N$ and $N_H=q+K-1$. The Hankel lifting
operator of delay depth $q$ is the linear map
\begin{equation}
    \mathscr L_{q,K}:\mathbb F^{N_H}\longrightarrow
    \mathbb F^{q\times K},
    \;
    [\mathscr L_{q,K}(\bm z)]_{ij}=z_{i+j-1},
    \;
    i=1,\ldots,q,\; j=1,\ldots,K.
    \label{eq:hankel_lifting_operator}
\end{equation}
Its range defines the \emph{finite-dimensional Hankel-lifted space}
\begin{equation}
    \mathscr H_{q,K}
    :=\operatorname{range}(\mathscr L_{q,K})
    =\left\{
      \mathscr L_{q,K}(\bm z):\bm z\in\mathbb F^{N_H}
     \right\},
    \label{eq:hankel_lifted_space}
\end{equation}
consisting of matrices constant along each anti-diagonal.
Every component of $\bm z$ occurs in its lift, so $\mathscr L_{q,K}$
is injective and $\dim\mathscr H_{q,K}=N_H$.
In particular, $\bm H=\mathscr L_{q,K}(\bm y_H)$.
\end{definition}

The two shifted matrices associated with the particular lifted matrix
$\bm H$ are defined by
\begin{equation}
    \bm H_0
    =\begin{bmatrix}\bm h_1&\cdots&\bm h_{K-1}\end{bmatrix},
    \;
    \bm H_1
    =\begin{bmatrix}\bm h_2&\cdots&\bm h_K\end{bmatrix},
    \;
    \bm H_0,\bm H_1\in\mathbb F^{q\times(K-1)}.
    \label{eq:shifted_hankel_pair}
\end{equation}
Thus, each column pair $(\bm h_\ell,\bm h_{\ell+1})$ represents one observed
shift in the serialized delay-coordinate space.

Because vectorization and $\mathscr L_{q,K}$ are injective, the passage from
$\bm Y$ to $\bm H$ loses no measured entry. 
The Hankel row space is endowed with the Euclidean inner product, and the
corresponding matrix space uses the Frobenius norm.

\begin{definition}[Hankel-lifted reduced-order model of finite-horizon experimental data]
\label{def:finite_horizon_experimental_rom}
Let $p$ be the dimension of the data-supported lifted subspace used for
modal construction. For $1\leq r\leq p$, an orthogonal modal representation
of the Hankel matrix consists of
\begin{equation}
    \widehat{\bm H}_r=\bm\Phi_r\bm C_r,
    \qquad
    \bm C_r=\bm\Phi_r^{\mathrm H}\bm H,
    \qquad
    \bm\Phi_r^{\mathrm H}\bm\Phi_r=\bm I_r,
    \label{eq:experimental_rom}
\end{equation}
where $\bm\Phi_r\in\mathbb F^{q\times r}$ and
$\bm C_r\in\mathbb F^{r\times K}$.

The $\ell$th column of $\bm C_r$, denoted by
$\bm c_{r,\ell}\in\mathbb F^r$, is the reduced-coordinate vector
associated with the $\ell$th Hankel column. For an orthogonal basis,
\begin{equation}
    \bm c_{r,\ell}=\bm\Phi_r^{\mathrm H}\bm h_\ell,
    \qquad \ell=1,\ldots,K.
    \label{eq:reduced_coordinates}
\end{equation}
Its evolution along the serialized Hankel-column shifts may
subsequently be represented by a data-driven map $\mathcal M_r$:
\begin{equation}
    \bm c_{r,\ell+1}
    =\mathcal M_r\!\left(
        \bm c_{r,\ell},\ldots,\bm c_{r,1};\bm\theta
    \right)+\bm\eta_\ell,
    \label{eq:reduced_discrete_dynamics}
\end{equation}
for $\ell=1,\ldots,K-1$, where $\bm\theta$ contains the identified
parameters and $\bm\eta_\ell\in\mathbb F^r$ represents model discrepancy.

Together with anti-diagonal recovery,
this map defines the channel model. The representation is
\emph{reduced order} when $r<p$, relative to the data-supported lifted
subspace. Phase~I constructs the modes and coefficients; Phase~II specifies the nonlinear realization of $\mathcal M_r$.
\end{definition}

\subsection{Koopman Operator Framework in the Hankel-Lifted Space}
\label{subsec:koopman_framework}

The reduced representation in
Definition~\ref{def:finite_horizon_experimental_rom} separates lifted
structures from their coefficients. Koopman theory provides the operator
setting for relating this representation to the serialized shift
\cite{Koopman1931,KoopmNeum1932,Koopman1936,LasotaMackey1985,
LasotaMackey94,MezicBanaszuk2004,Mezic2005,Rowley2009,Chen2012,Brunton2016}.

In the present setting, no governing differential equation is assumed to be
known. Let $\bm s_k\in\mathcal S$ denote the unknown underlying state at $t_k$ and
represent its one-sample evolution abstractly by
\begin{equation}
    \bm s_{k+1}=\bm F_{\Delta t}(\bm s_k),
    \qquad k=1,\ldots,N-1,
    \label{eq:experimental_evolution_map}
\end{equation}
and let the measurement observable be
\begin{equation}
    \bm g:\mathcal S\rightarrow\mathbb F^m,
    \qquad
    \bm y_k=\bm g(\bm s_k).
    \label{eq:experimental_observable}
\end{equation}
For this operator interpretation, the underlying state is assumed
to include any noise or unresolved variables needed to make
$\bm y_k=\bm g(\bm s_k)$ and the deterministic evolution map
well defined.

The physical map advances time by $\Delta t$, whereas the serialized
shift in \eqref{eq:hkfed_number_hankel_columns} advances through the channels
before moving to the next sample. To represent this distinction, define the
augmented serialized state space
\begin{equation}
    \mathcal S_{\mathrm{ser}}=\mathcal S\times\{1,\ldots,m\}
    \label{eq:serialized_augmented_state_space}
\end{equation}
and the serialized shift map $\bm F_{\mathrm{ser}}:\mathcal S_{\mathrm{ser}}
\rightarrow\mathcal S_{\mathrm{ser}}$ by
\begin{equation}
    \bm F_{\mathrm{ser}}(\bm s,i)
    =
    \begin{cases}
       (\bm s,i+1), & 1\leq i<m,\\
       (\bm F_{\Delta t}(\bm s),1), & i=m.
    \end{cases}
    \label{eq:serialized_shift_map}
\end{equation}
For $\ell=m(k-1)+i$, set $\bm\sigma_\ell=(\bm s_k,i)$ and introduce the
scalar serialized observable
\begin{equation}
    g_{\mathrm{ser}}(\bm s,i)
    =\bm e_i^{\mathsf T}\bm g(\bm s),
    \qquad
    y_{H,\ell}=g_{\mathrm{ser}}(\bm\sigma_\ell),
    \qquad
    \bm\sigma_{\ell+1}=\bm F_{\mathrm{ser}}(\bm\sigma_\ell),
    \label{eq:serialized_scalar_observable}
\end{equation}
where $\bm e_i$ is the $i$th canonical basis vector of
$\mathbb F^m$, $k=1,\ldots,N$, and $i=1,\ldots,m$.
The state-shift identity applies for
$\ell=1,\ldots,N_H-1$.

The $q$-component delay observable associated with the Hankel lift is therefore
\begin{equation}
    \bm h_q(\bm\sigma)
    =
    \begin{bmatrix}
       g_{\mathrm{ser}}(\bm\sigma) &
       g_{\mathrm{ser}}(\bm F_{\mathrm{ser}}(\bm\sigma)) &
       \cdots &
       g_{\mathrm{ser}}(\bm F_{\mathrm{ser}}^{q-1}(\bm\sigma))
    \end{bmatrix}^{\mathsf T}\in\mathbb F^q.
    \label{eq:delay_observable}
\end{equation}
Equations~\eqref{eq:multichannel_delay_vector} and
\eqref{eq:delay_observable} give $\bm h_\ell=\bm h_q(\bm\sigma_\ell)$:
the Hankel columns sample one vector-valued observable along the augmented
trajectory, distinguishing channel shifts from physical-time advances.

A Koopman composition operator requires a well-defined evolution map
and an observable space invariant under composition
\cite{Koopman1931,LasotaMackey1985,Mezic2005,Budisic2012,Williams2015}.
Connections with Hankel data include Hankel--DMD
\cite{ArbabiMezic2017}, spectral equivalence under delay embedding
\cite{SusukiSakoHikihara2017}, HAVOK \cite{BruntonEtAl2017},
delay-based spectral representations \cite{DasGiannakis2019,KambEtAl2020},
Krylov-space prediction \cite{KoltaiKunde2024}, and multichannel HAVOK
\cite{ColcheroEtAl2026}.

The following proposition specializes this established theory to the
channel-augmented serialized shift. It makes explicit how the scalar-entry
Hankel construction generates the sampled Koopman action used below.

\begin{proposition}[Koopman compatibility of the channel-augmented
serialized Hankel lift]
\label{prop:koopman_existence}
Let $\mathcal G_{\mathrm{ser}}$ be a linear space of observables containing
$g_{\mathrm{ser}}$ and invariant under composition with the serialized
shift $\bm F_{\mathrm{ser}}$ in \eqref{eq:serialized_shift_map}.
Then the Koopman operator
\begin{equation}
    \mathcal K_{\mathrm{ser}}:\mathcal G_{\mathrm{ser}}
    \longrightarrow\mathcal G_{\mathrm{ser}},
    \qquad
    \mathcal K_{\mathrm{ser}}g=g\circ\bm F_{\mathrm{ser}},
    \label{eq:koopman_operator}
\end{equation}
is well defined and linear. With the componentwise action denoted by
$\mathcal K_{\mathrm{ser}}^{(q)}$, the delay observable in
\eqref{eq:delay_observable} satisfies
\begin{equation}
    \bm h_\ell=\bm h_q(\bm\sigma_\ell),
    \qquad
    \bm h_{\ell+1}
    =(\mathcal K_{\mathrm{ser}}^{(q)}\bm h_q)(\bm\sigma_\ell),
    \quad \ell=1,\ldots,K-1.
    \label{eq:serialized_hankel_sampled_action}
\end{equation}
If the finite delay-observable space
\begin{equation}
    \mathcal G_q^H
    =\operatorname{span}\{g_{\mathrm{ser}},
      \mathcal K_{\mathrm{ser}}g_{\mathrm{ser}},\ldots,
      \mathcal K_{\mathrm{ser}}^{q-1}g_{\mathrm{ser}}\}
    \label{eq:finite_delay_observable_space}
\end{equation}
is also Koopman invariant, there exists $\bm K_H\in\mathbb F^{q\times q}$
such that
\begin{equation}
    \mathcal K_{\mathrm{ser}}^{(q)}\bm h_q=\bm K_H\bm h_q,
    \qquad
    \bm h_{\ell+1}=\bm K_H\bm h_\ell,
    \qquad
    \bm H_1=\bm K_H\bm H_0.
    \label{eq:exact_hankel_koopman_relation}
\end{equation}
The representing matrix is unique when the component observables of
$\bm h_q$ are linearly independent. Otherwise, its action on the sampled
delay vectors is still fixed by the observable evolution.
\end{proposition}

\begin{proof}
Invariance of $\mathcal G_{\mathrm{ser}}$ makes composition an operator
on that space, and linearity follows from the distributive property of
composition over linear combinations. The components of $\bm h_q$ are
$\mathcal K_{\mathrm{ser}}^j g_{\mathrm{ser}}$, $j=0,\ldots,q-1$.
Evaluating them along
$\bm\sigma_{\ell+1}=\bm F_{\mathrm{ser}}(\bm\sigma_\ell)$ gives
\eqref{eq:serialized_hankel_sampled_action}.
If $\mathcal G_q^H$ is invariant, the Koopman image of every component
is a linear combination of those components. The coefficients form the
rows of $\bm K_H$; evaluation along the trajectory and column stacking
give \eqref{eq:exact_hankel_koopman_relation}. Linear independence makes
these coefficients unique.
\end{proof}

The serialized Koopman shift is therefore compatible with the Hankel-lifted
reduced representation. The selected subspace is determined by finite-horizon
energy relevance and need not be Koopman invariant: orthogonal energy
directions define the reduced coordinates, whose nonlinear memory-dependent
evolution is identified separately. The next section constructs these modes
and coefficients from a finite-data Koopman approximation.

\section{Computational Aspects}\label{sec:computational_aspects}

The three phases connect the lifted measurement representation to
nonlinear channel reconstruction and then to a quantity-of-interest
forecast. The following subsections specify each construction and the
interfaces through which its outputs pass to the next phase.

\subsection{Phase I: Hankel--Koopman Finite-Horizon Energy Decomposition}
\label{subsec:phase_hkfed}

Phase~I constructs orthogonal energy modes from the serialized Hankel
representation and ranks their associated eigenvalue--mode--coefficient
triplets over a prescribed operator horizon. The retained coefficients
provide the states evolved in Phase~II.

\subsubsection{Hankel Representation of the Multichannel Data}

The construction uses $\bm H$, $\bm H_0$, and $\bm H_1$ from
\eqref{eq:multichannel_block_hankel} and \eqref{eq:shifted_hankel_pair},
with $q\approx\lfloor N/2\rfloor$ rows.
The fixed channel--time ordering places within-channel evolution and
cross-channel dependence in the same lifted representation.

\subsubsection{Finite-Dimensional Approximation of the Koopman Operator}

Proposition~\ref{prop:koopman_existence} identifies the shifted Hankel
pair as sampled Koopman action. We approximate that action on the
data-supported subspace by projected least squares.

Let
\begin{equation}
    p=\operatorname{rank}(\bm{H}_0)
    \leq \min\{q,K-1\}
    \label{eq:hkfed_economy_rank}
\end{equation}
denote the rank determined from the economy-size singular value decomposition
of the input matrix. %
Restricting the economy-size factors to the
positive singular spectrum gives
\begin{equation}
    \bm{H}_0
    =
    \bm{U}_{p}\bm{\Sigma}_{p}\bm{V}_{p}^{\mathrm H},
    \qquad
    \bm{U}_{p}^{\mathrm H}\bm{U}_{p}=\bm{I}_{p},
    \qquad
    \bm{V}_{p}^{\mathrm H}\bm{V}_{p}=\bm{I}_{p},
    \label{eq:hkfed_economy_svd}
\end{equation}
where $\bm{\Sigma}_{p}$ contains the positive singular values.

\begin{proposition}[Finite-dimensional approximation of the Koopman operator]
\label{prop:finite_koopman_approximation}
Let $\bm{H}_0,\bm{H}_1\in\mathbb{F}^{q\times(K-1)}$ be the
serialized-shift Hankel matrices defined by
\eqref{eq:shifted_hankel_pair}, and let
$\bm{H}_0=\bm{U}_p\bm{\Sigma}_p\bm{V}_p^{\mathrm H}$ be the
rank-$p$ economy-size SVD in \eqref{eq:hkfed_economy_svd}, with $p\geq1$.
The finite-dimensional approximation of the Koopman operator on the
data-supported subspace $\operatorname{range}(\bm{U}_p)$ is defined as the
solution of the finite-data least-squares problem
\begin{equation}
    \widetilde{\bm{K}}_p
    =
    \underset{\widetilde{\bm{K}}\in\mathbb{F}^{p\times p}}
    {\operatorname{arg\,min}}
    \left\|
        \bm{H}_1
        -\bm{U}_p\widetilde{\bm{K}}
         \bm{U}_p^{\mathrm H}\bm{H}_0
    \right\|_F^2.
    \label{eq:projected_koopman_least_squares}
\end{equation}
This problem has a unique solution, explicitly given by
\begin{equation}
    \widetilde{\bm{K}}_{p}
    =
    \bm{U}_{p}^{\mathrm H}
    \bm{H}_{1}
    \bm{V}_{p}\bm{\Sigma}_{p}^{-1}
    \in\mathbb{F}^{p\times p}.
    \label{eq:finite_koopman_approximation}
\end{equation}
\end{proposition}

\begin{proof}
The normal equation is
$\widetilde{\bm K}\bm\Sigma_p^2
=\bm U_p^{\mathrm H}\bm H_1\bm V_p\bm\Sigma_p$.
Since $\bm\Sigma_p$ has positive diagonal entries, the objective is
strictly convex and its unique minimizer is
\eqref{eq:finite_koopman_approximation}.
\end{proof}

When the finite delay space is Koopman invariant,
substitution of $\bm H_1=\bm K_H\bm H_0$ into
\eqref{eq:finite_koopman_approximation} gives
\begin{equation}
    \widetilde{\bm K}_p
    =\bm U_p^{\mathrm H}\bm K_H\bm U_p
    \in\mathbb F^{p\times p}.
    \label{eq:exact_to_projected_koopman_relation}
\end{equation}
Under invariance, $\widetilde{\bm K}_p$ is the orthogonal projection
of the exact representation. For general experimental data it is the
unique projected least-squares approximation. This distinction connects
the sampled serialized shift in this paper to established Hankel--DMD theory
\cite{ArbabiMezic2017,SusukiSakoHikihara2017,BruntonEtAl2017,
DasGiannakis2019,KambEtAl2020,KoltaiKunde2024,ColcheroEtAl2026}.

For general experimental data, $\widetilde{\bm{K}}_p$ may be non-normal.
Consequently, its eigenvectors need not be orthogonal and their individual
contributions need not possess an additive energy interpretation.

 The goal of the present approach is to obtain a stable orthogonal decomposition while retaining the action of the identified
Koopman approximation. Define the self-adjoint positive-semidefinite operator
\begin{equation}
    \bm{G}_{p}
    =
    \widetilde{\bm{K}}_{p}^{\mathrm H}\widetilde{\bm{K}}_{p},
    \qquad
    \bm{G}_{p}^{\mathrm H}=\bm{G}_{p},
    \qquad
    \bm{z}^{\mathrm H}\bm{G}_{p}\bm{z}
    =
    \|\widetilde{\bm{K}}_{p}\bm{z}\|_2^2\geq0 .
    \label{eq:koopman_energy_operator}
\end{equation}
By the spectral theory \cite{Davies1996}, there exist
$\mu_j\geq0$ and an orthonormal eigenvector matrix
$\bm{Q}=[\bm{q}_1,\ldots,\bm{q}_p]$ such that
\begin{equation}
\begin{gathered}
    \bm G_p\bm Q=\bm Q\bm\Lambda_G,
    \qquad
    \bm\Lambda_G=\operatorname{diag}(\mu_1,\ldots,\mu_p),
    \\[2mm]
    \mu_1\geq\mu_2\geq\cdots\geq\mu_p\geq0,
    \qquad
    \bm Q^{\mathrm H}\bm Q=\bm I_p.
\end{gathered}
\label{eq:koopman_energy_eigendecomposition}
\end{equation}
The values $\mu_j$ are the squared singular values of $\widetilde{\bm K}_p$,
measuring its one-step energetic action along $\bm q_j$.
Phase~I ranks these energy directions rather than the generally complex
Koopman eigenvalues.

The relation to the Koopman spectrum is explicit under normality: if
$\widetilde{\bm K}_p\bm v=\widetilde\lambda\bm v$ and
$\widetilde{\bm K}_p$ is normal, then
\begin{equation}
    \bm G_p\bm v
    =\widetilde{\bm K}_p^{\mathrm H}
      \widetilde{\bm K}_p\bm v
    =|\widetilde\lambda|^2\bm v.
    \label{eq:normal_koopman_energy_spectral_link}
\end{equation}

\begin{definition}[Hankel--Koopman energy modes and coefficient matrix]
\label{def:hkfed_energy_modes}
Let $\widetilde{\bm{K}}_p\in\mathbb{C}^{p\times p}$ be the reduced
finite-dimensional approximation of the Koopman operator, and let
$\bm{Q}\in\mathbb{C}^{p\times p}$ contain an orthonormal set of eigenvectors
of the positive semidefinite operator
$\bm{G}_{p}=\widetilde{\bm{K}}_p^{\mathrm H}\widetilde{\bm{K}}_p$. The
\emph{Hankel--Koopman energy modes} and their associated
finite-horizon coefficient matrix are defined by
\begin{equation}
    \bm{\Phi}
    =
    \bm{U}_{p}\bm{Q}
    =
    \begin{bmatrix}
        \bm{\phi}_1 & \cdots & \bm{\phi}_p
    \end{bmatrix}
    \in\mathbb F^{q\times p},
    \qquad
    \bm{C}
    =
    \bm{\Phi}^{\mathrm H}\bm{H}
    =
    \begin{bmatrix}
        \bm{c}_1^{\mathsf T}\\[-1mm]
        \vdots\\[-1mm]
        \bm{c}_p^{\mathsf T}
    \end{bmatrix}
    \in\mathbb F^{p\times K}.
    \label{eq:hkfed_modes_and_coefficients}
\end{equation}
For every $j=1,\ldots,p$, the associated \emph{Koopman-energy modal triplet}
is
\begin{equation}
    \mathfrak{t}_j
    =
    \left(
       \mu_j,\bm{\phi}_j,\bm{c}_j^{\mathsf T}
    \right).
    \label{eq:hkfed_modal_triplet}
\end{equation}
\end{definition}

\begin{remark}
The rows of $\bm C$ describe coefficient evolution along serialized-delay
positions; its columns are the reduced states. Each eigenvalue, mode,
and coefficient row forms one linked triplet and must be sorted or
retained together. For a non-normal Koopman approximation, these are
orthogonal energy directions rather than Koopman eigendirections.
\end{remark}

\begin{proposition}[Orthogonality of the Hankel--Koopman energy modes]
\label{prop:hkfed_mode_orthogonality}
The modes defined by \eqref{eq:hkfed_modes_and_coefficients} are orthonormal
with respect to the Hankel-space inner product and satisfy
\begin{equation}
\langle \bm{x},\bm{y}\rangle
    :=
    \bm{x}^{\mathrm H}\bm{y},
    \qquad
    \bm{\Phi}^{\mathrm H}\bm{\Phi}
    =
    \bm{I}_p.
    \label{eq:hkfed_orthogonality}
\end{equation}
Consequently, $\bm{C}$ is the unique coefficient matrix of the orthogonal
projection of $\bm{H}$ onto
$\operatorname{Range}(\bm{\Phi})$, and
\begin{equation}
    \widehat{\bm{H}}_{p}
    =\bm{\Phi}\bm{C}
    =\bm{\Phi}\bm{\Phi}^{\mathrm H}\bm{H}.
    \label{eq:full_hkfed_projection}
\end{equation}
\end{proposition}

\begin{proof}
Using \eqref{eq:hkfed_modes_and_coefficients}, the orthogonality of
$\bm{U}_p$ in \eqref{eq:hkfed_economy_svd}, and the orthogonality of $\bm{Q}$ in
\eqref{eq:koopman_energy_eigendecomposition}, one obtains
$\bm\Phi^{\mathrm H}\bm\Phi
=\bm Q^{\mathrm H}\bm U_p^{\mathrm H}\bm U_p\bm Q=\bm I_p$.
The coefficient and projection formulas then follow from the standard
orthogonal projection theorem in the delay space.
\end{proof}

\subsubsection{Finite-Horizon Modal Energy Criterion}

The finite-horizon modal energy criterion introduced in this work combines coefficient magnitude
with persistence under the identified energy operator.
The \emph{operator horizon} $L\in\mathbb N_0$ counts energy-operator
applications and differs from the modeling duration $T_{\mathrm{fin}}$
and an energy-retention percentage.

\begin{definition}[Finite-horizon persistence factor]
\label{def:finite_horizon_persistence}
Let $\mu_j\geq 0$ be the eigenvalue of the Koopman energy operator
$\bm{G}_p$ associated with the energy direction $\bm{q}_j$, and let
$L\in\mathbb{N}_0$ be a prescribed operator horizon. The
\emph{finite-horizon persistence factor} of the $j$th energy direction is
defined by
\begin{equation}
    P_j^{(L)}
    =
    \sum_{\ell=0}^{L}\mu_j^{\ell}
    =
    \begin{cases}
        \displaystyle
        \frac{1-\mu_j^{L+1}}{1-\mu_j},
        & \mu_j\neq1,\\[3mm]
        L+1,
        & \mu_j=1 .
    \end{cases}
    \label{eq:finite_horizon_persistence}
\end{equation}
It quantifies the accumulated energetic action of that direction over the
$L+1$ operator levels $\ell=0,\ldots,L$.
\end{definition}

\begin{definition}[Finite-horizon modal energy]
\label{def:finite_horizon_modal_energy}
Let $\bm{c}_j^{\mathsf T}$ be the temporal coefficient row associated with the
$j$th orthogonal Hankel--Koopman energy mode. The \emph{coefficient energy} of
this mode over the recorded experiment is
\begin{equation}
    \alpha_j
    =
    \|\bm{c}_j^{\mathsf T}\|_2^2,
    \label{eq:hkfed_coefficient_energy}
\end{equation}
and its \emph{finite-horizon modal energy} over operator horizon $L$ is defined
as
\begin{equation}
    E_j^{(L)}
    =
    \alpha_jP_j^{(L)}
    =
    \|\bm{c}_j^{\mathsf T}\|_2^2
    \sum_{\ell=0}^{L}\mu_j^{\ell}.
    \label{eq:hkfed_modal_energy}
\end{equation}
Thus, $\alpha_j$ measures the contribution of the mode to the recorded
experiment, whereas $P_j^{(L)}$ measures its persistence under the reduced
Koopman energy operator.
\end{definition}

\begin{theorem}[Finite-horizon Hankel--Koopman energy decomposition]
\label{thm:finite_horizon_hk_energy}
Let $\bm G_p$, $\bm Q$, $\bm\Lambda_G$, and $\bm C$ be defined by
\eqref{eq:koopman_energy_operator}--\eqref{eq:hkfed_modes_and_coefficients},
and let $L\in\mathbb N_0$. The orthonormal modes
$\bm\Phi=\bm U_p\bm Q$ admit the accumulated quadratic energy
\begin{equation}
    \mathscr J_L(\bm C)
    =\sum_{\ell=0}^{L}
      \operatorname{trace}(\bm C^{\mathrm H}\bm\Lambda_G^\ell\bm C),
    \label{eq:accumulated_hkfed_energy}
\end{equation}
with the additive decomposition
\begin{equation}
    \mathscr J_L(\bm C)
    =\sum_{j=1}^{p}\|\bm c_j^{\mathsf T}\|_2^2
      \sum_{\ell=0}^{L}\mu_j^\ell
    =\sum_{j=1}^{p}E_j^{(L)}.
    \label{eq:additive_hkfed_energy}
\end{equation}
For any prescribed $r$, retaining the $r$ largest $E_j^{(L)}$ maximizes
the retained accumulated energy among subsets of $r$ triplets from this
fixed modal basis.
\end{theorem}

\begin{proof}
Orthonormality follows from
$\bm\Phi^{\mathrm H}\bm\Phi
=\bm Q^{\mathrm H}\bm U_p^{\mathrm H}\bm U_p\bm Q=\bm I_p$.
Since $\bm\Lambda_G$ is diagonal,
$\operatorname{trace}(\bm C^{\mathrm H}\bm\Lambda_G^\ell\bm C)
=\sum_{j=1}^{p}\mu_j^\ell\|\bm c_j^{\mathsf T}\|_2^2$.
Summing over the finite operator horizon gives
\eqref{eq:additive_hkfed_energy}.
Retaining a subset of coefficient rows retains exactly the corresponding
nonnegative terms $E_j^{(L)}$. Their sum is maximized by the $r$ largest
terms, with any ordering of exact ties giving the same value.
\end{proof}

\begin{remark}[Finite-horizon persistence]
The energy interpretation is particularly transparent for eigenvalues near
unity. For $0\leq\mu_j\leq1$, $P_j^{(L)}\leq L+1$, with equality at
$\mu_j=1$. Hence a unit
eigenvalue identifies a direction whose norm is preserved by
$\widetilde{\bm{K}}_p$, and its accumulated energy increases linearly with the
horizon. Values near unity therefore identify nearly persistent
energy directions. Values below or above unity describe attenuating
or amplifying directions, whose relevance is assessed over the
prescribed operator horizon.
\end{remark}

\subsubsection{Koopman-Energy Triplet Selection and Reduced Reconstruction}

Let $\pi$ be a permutation such that
\begin{equation}
    E_{\pi_1}^{(L)}
    \geq E_{\pi_2}^{(L)}
    \geq\cdots\geq
    E_{\pi_p}^{(L)} .
    \label{eq:hkfed_energy_ordering}
\end{equation}
For a candidate dimension $r\in\{1,\ldots,p\}$, define
\begin{equation}
    \bm{\Lambda}_{G,r}
    =
    \operatorname{diag}
    \left(
        \mu_{\pi_1},\ldots,\mu_{\pi_r}
    \right),
    \;
    \bm{\Phi}_r
    =
    \begin{bmatrix}
        \bm{\phi}_{\pi_1}&\cdots&\bm{\phi}_{\pi_r}
    \end{bmatrix},
    \;
    \bm{C}_r
    =
    \begin{bmatrix}
        \bm{c}_{\pi_1}^{\mathsf T}\\
        \vdots\\
        \bm{c}_{\pi_r}^{\mathsf T}
    \end{bmatrix},
    \;
    \widehat{\bm{H}}_r=\bm{\Phi}_r\bm{C}_r .
    \label{eq:reduced_hankel_reconstruction}
\end{equation}
Equivalently, the rank-$r$ candidate is the ordered collection of the first
$r$ modal triplets,
$\mathfrak T_r=(\bm\Lambda_{G,r},\bm\Phi_r,\bm C_r)
=(\mathfrak t_{\pi_1},\ldots,\mathfrak t_{\pi_r})$.

Thus, $r$ is the cardinality of a complete Koopman-energy candidate.
The cumulative energy fraction
\begin{equation}
    \rho_r^{(L)}
    =
    \frac{\sum_{\ell=1}^{r}E_{\pi_\ell}^{(L)}}
         {\sum_{j=1}^{p}E_j^{(L)}}
    \label{eq:cumulative_hkfed_energy}
\end{equation}
quantifies the energy captured by the first $r$ ranked modes.

A principal role of Phase~I is to determine the
smallest number of Koopman-energy modes consistent with the required
reconstruction quality. When several $\mu_j$ lie close to unity, the
corresponding modes remain persistent over the finite horizon, and a high
energy-retention level may be associated with a large modal set. Phase~I
therefore integrates the finite-horizon energy ordering with a Pareto-based
selection that jointly evaluates reconstruction quality and reduced
dimension to identify the retained candidate.

For matrices of the same size, introduce the Frobenius inner product
$\langle\bm A,\bm B\rangle_F
=\operatorname{trace}(\bm A^{\mathrm H}\bm B)$.
Its induced norm is
$\|\bm{A}\|_F=\langle\bm{A},\bm{A}\rangle_F^{1/2}$.

For $\bm H\neq\bm 0$ and $\widehat{\bm H}_r\neq\bm 0$,
the relative reconstruction error and matrix cosine similarity
of candidate $r$ are, respectively,
\begin{align}
    \varepsilon_r
    &=
    \frac{
      \|\bm{H}-\widehat{\bm{H}}_r\|_F
    }{
      \|\bm{H}\|_F
    },
    \label{eq:hkfed_reconstruction_error}\\
    \gamma_r
    &=
    \frac{
      \operatorname{Re}
      \langle\bm{H},\widehat{\bm{H}}_r\rangle_F
    }{
      \|\bm{H}\|_F
      \|\widehat{\bm{H}}_r\|_F
    } .
    \label{eq:hkfed_cosine_similarity}
\end{align}

Let $\mathcal R=\{1,\ldots,p\}$
denote the finite set of candidate reduced dimensions, and let $\chi_r=r/p$ be
the corresponding dimension ratio. For a weak dimensionality penalty
$\beta_{\mathrm{dim}}\geq0$, define the vector-valued objective
\begin{equation}
    \bm{F}(r)
    =
    \begin{bmatrix}
        F_1(r)\\[1mm]
        F_2(r)
    \end{bmatrix}
    =
    \begin{bmatrix}
        \varepsilon_r+\beta_{\mathrm{dim}}\chi_r\\[1mm]
        1-\gamma_r+\beta_{\mathrm{dim}}\chi_r
    \end{bmatrix}.
    \label{eq:hkfed_pareto_objectives}
\end{equation}
The objectives favor reconstruction accuracy and structural alignment,
with $\beta_{\mathrm{dim}}\chi_r$ weakly penalizing larger dimensions.

\begin{definition}[Pareto-optimal reduced Koopman-energy candidate]
\label{def:pareto_optimal_dimension}
For $r_a,r_b\in\mathcal R$, candidate $\mathfrak T_{r_a}$
\emph{Pareto-dominates} $\mathfrak T_{r_b}$, written $r_a\prec r_b$,
if $F_i(r_a)\leq F_i(r_b)$ for $i=1,2$, with at least one strict inequality.
The Pareto-optimal candidates are those whose indices belong to
\begin{equation}
    \mathcal P
    =\{r\in\mathcal R:\nexists\,s\in\mathcal R
      \text{ such that }s\prec r\}.
    \label{eq:hkfed_pareto_set}
\end{equation}
\end{definition}

The Pareto search generally produces several nondominated candidates.
A second decision rule selects $\mathfrak T_{r^\star}$, whose retained
dimension is $r^\star$.

We specify an admissible interval
$r_{\min}\leq r\leq r_{\max}$ and a preferred dimension
$r_{\mathrm{tar}}$ within this interval, and define the admissible set
$\mathcal R_{\mathrm{adm}}
=\{r\in\mathcal R:r_{\min}\leq r\leq r_{\max}\}$,
and the final candidate set
\begin{equation}
    \mathcal{C}
    =
    \begin{cases}
        \mathcal{P}\cap\mathcal{R}_{\mathrm{adm}},
        & \mathcal{P}\cap\mathcal{R}_{\mathrm{adm}}\neq\varnothing,\\[1mm]
        \mathcal{R}_{\mathrm{adm}},
        & \mathcal{P}\cap\mathcal{R}_{\mathrm{adm}}=\varnothing.
    \end{cases}
    \label{eq:hkfed_pruned_pareto_set}
\end{equation}
The target $r_{\mathrm{tar}}$ expresses a soft dimension preference, yet
reconstruction quality may return a different order.
For $\alpha_{\mathrm{dim}}\in[0,1]$, define
\begin{equation}
\begin{aligned}
    J(r)
    &=
    \frac{1-\alpha_{\mathrm{dim}}}{2}
    \left(
       \frac{\varepsilon_r}
            {\max_{s\in\mathcal C}\varepsilon_s+\epsilon_{\mathrm{mach}}}
       +
       \frac{1-\gamma_r}
            {\max_{s\in\mathcal C}(1-\gamma_s)+\epsilon_{\mathrm{mach}}}
    \right)
    \\[1mm]
    &\quad+\alpha_{\mathrm{dim}}\frac{|r-r_{\mathrm{tar}}|}{r_{\mathrm{tar}}},
    \qquad
    r^\star=\min\!\left(\operatorname*{arg\,min}_{r\in\mathcal C}J(r)\right).
\end{aligned}
\label{eq:hkfed_final_selection_score}
\end{equation}
The reconstruction error and cosine loss are normalized by their maxima
over $\mathcal C$, with $\epsilon_{\mathrm{mach}}>0$ preventing zero
denominators. Among minimizers, the smallest retained dimension is selected.

The output of the modal stage is the selected Koopman-energy candidate
\begin{equation}
\begin{aligned}
    \mathfrak T_{r^\star}
    &=
    \left(
       \bm\Lambda_{G,r^\star},
       \bm\Phi_{r^\star},
       \bm C_{r^\star}
    \right).
\end{aligned}
    \label{eq:selected_hkfed_reconstruction}
\end{equation}
The selected coefficient matrix $\bm C_{r^\star}$ supplies the
coupled reduced states used in Phase~II.

\subsubsection{Anti-Diagonal Recovery of the Experimental Channels}

Modal truncation generally yields a matrix
$\widehat{\bm{H}}_{r^\star}$ that approximates the exact Hankel structure.
Consequently, the scalar positions representing a common entry of the serialized record
may contain distinct approximations. A unique serialized record is obtained
by averaging these estimates along the corresponding scalar anti-diagonals.
For each serialized index $\ell\in\{1,\ldots,N_H\}$, we define
\[\mathcal J_\ell=\{(i,k):1\leq i\leq q,\;1\leq k\leq K,\;i+k-1=\ell\}.\]

\begin{definition}[Uniform anti-diagonal averaging]
\label{def:uniform_antidiagonal_averaging}
Let $\widehat{\bm{H}}_{r^\star}
=\bigl[\widehat H_{i,k}\bigr]\in\mathbb F^{q\times K}$ be the reduced
Hankel reconstruction. The \emph{anti-diagonal multiplicity}
of serialized entry $\ell$ is
$\nu_\ell=|\mathcal J_\ell|=\min\{\ell,q,K,N_H-\ell+1\}$.
The \emph{uniform anti-diagonal average} of
$\widehat{\bm H}_{r^\star}$ is the recovered serialized sequence
$\{\widehat y_{H,\ell}\}_{\ell=1}^{N_H}$ defined by
\begin{equation}
    \widehat y_{H,\ell}
    =
    \frac{1}{\nu_\ell}
    \sum_{(i,k)\in\mathcal J_\ell}
    \widehat H_{i,k},
    \qquad \ell=1,\ldots,N_H.
    \label{eq:uniform_antidiagonal_recovery}
\end{equation}
Thus, all reconstructed occurrences of the same serialized record contribute
equally.
\end{definition}

Consequently, uniform anti-diagonal averaging is the least-squares projection
of the reduced matrix onto the set of Hankel serialized sequences.

Let
$\widehat{\bm y}_H=[\widehat y_{H,1},\ldots,
\widehat y_{H,N_H}]^{\mathsf T}$. The inverse of the serialization in
\eqref{eq:hkfed_number_hankel_columns} gives
\begin{equation}
    \widehat{\bm Y}
    =\operatorname{reshape}(\widehat{\bm y}_H,m,N)
    =
    \begin{bmatrix}
       \widehat{\bm y}_1&\cdots&\widehat{\bm y}_N
    \end{bmatrix}.
    \label{eq:recovered_experimental_channels}
\end{equation}
Thus, anti-diagonal recovery produces a vector of length $N_H=mN$,
and inverse reshaping restores the original channel ordering.

\subsubsection{Algorithmic Formulation and Computational Complexity}

Algorithm \ref{alg:hkfed} presents the Hankel–Koopman finite-horizon energy decomposition (HKFED) introduced in this paper.
The economy-size SVD supports the full finite-data construction;
modal truncation is applied after energy ranking. The computational cost is summarized in
Table~\ref{tab:framework_complexity}.

\begin{breakablealgorithm}
{Hankel--Koopman finite-horizon energy decomposition (HKFED)}
\label{alg:hkfed}
\begin{algorithmic}[1]
\Require Measurements $\bm{Y}$ on $\mathcal T_{\mathrm{fin}}$;
         imposed Hankel delay depth $q$;
         operator horizon $L$; Pareto and dimension-selection parameters.
\Ensure Selected modes $\bm{\Phi}_{r^\star}$, reduced coefficients
        $\bm{C}_{r^\star}$, associated energy eigenvalues
        $\bm\Lambda_{G,r^\star}$, reconstructed Hankel matrix
        $\widehat{\bm{H}}_{r^\star}$, and recovered channel matrix
        $\widehat{\bm Y}$.
\State Serialize $\bm Y$ directly as
       $\bm y_H=\operatorname{vec}(\bm Y)$ and construct
       $\bm H$, $\bm H_0$, and $\bm H_1$ according to
       \eqref{eq:hkfed_number_hankel_columns}--\eqref{eq:shifted_hankel_pair}.
\State Compute the rank-$p$ economy-size SVD of $\bm H_0$ and form
       $\widetilde{\bm K}_p$ from \eqref{eq:finite_koopman_approximation}.
\State Form $\bm G_p$, compute its eigendecomposition, and obtain
       $\bm\Phi$ and $\bm C$ from \eqref{eq:hkfed_modes_and_coefficients}.
\State Compute $E_j^{(L)}$ and sort the complete modal triplets
       by decreasing finite-horizon energy.
\State Evaluate the reconstruction objectives for the candidate dimensions,
       form $\mathcal C$ from \eqref{eq:hkfed_pruned_pareto_set}, and select
       $r^\star$ using \eqref{eq:hkfed_final_selection_score}.
\State Set $\widehat{\bm H}_{r^\star}=\bm\Phi_{r^\star}\bm C_{r^\star}$
       and recover $\widehat{\bm Y}$ using
       \eqref{eq:uniform_antidiagonal_recovery} and
       \eqref{eq:recovered_experimental_channels}.
\State \Return
       $\bm{\Phi}_{r^\star}$, $\bm{C}_{r^\star}$,
       $\bm{\Lambda}_{G,r^\star}$, $\{E_j^{(L)}\}_{j=1}^{p}$,
       $r^\star$, $\widehat{\bm{H}}_{r^\star}$, and $\widehat{\bm Y}$.
\end{algorithmic}
\end{breakablealgorithm}

\subsection{Phase II: Inverse-Calibrated Multi-Output NLARX Model}
\label{subsec:phase_nlarx}

NLARX models represent nonlinear input-output dynamics through delayed
outputs and inputs \cite{Liu2018,Gu2023,Zhao2024,Fatima2023}.
Phase~II uses this structure to identify the reduced evolution map
$\mathcal M_r$ in Definition~\ref{def:finite_horizon_experimental_rom}.
The selected coefficients are simulated recursively in Hankel space, reconstructed in
the physical channels, and assessed against the calibration data.
This assessment after reconstruction defines \emph{inverse calibration}.

The calibration window $\mathcal T_H$ is defined in
Definition~\ref{def:finite_observation_forecast_horizon}, with
$T_{\mathrm{start}}=t_{n_{\mathrm s}}$ and
$n_{\mathrm s}=N_Q+1$.
The
associated global sampling-index set is
\begin{equation}
\begin{gathered}
    \mathcal I_H
    =\{n_{\mathrm s},n_{\mathrm s}+1,\ldots,n_{\mathrm s}+H\},
    \qquad
    N_{\mathcal I}=H+1,
    \qquad
    N_{H,\mathcal I}=mN_{\mathcal I},
    \\[2mm]
    K_{\mathcal I}=N_{H,\mathcal I}-q+1,
    \qquad
    \ell_{\mathrm s}=m(n_{\mathrm s}-1)+1.
\end{gathered}
\label{eq:nlarx_finite_index_horizon}
\end{equation}
The window is admissible when $N_{H,\mathcal I}\geq q+1$, or
equivalently $K_{\mathcal I}\geq2$, while the subscript $\mathcal I$
denotes restriction to $\mathcal I_H$.

The corresponding measurement matrix, Hankel matrix, and selected coefficient matrix are, respectively
\begin{equation}
\begin{aligned}
    \bm{Y}_{\mathcal I}
    &=
    \begin{bmatrix}
      \bm y_{n_{\mathrm s}}&\cdots&\bm y_{n_{\mathrm s}+H}
    \end{bmatrix}
    \in\mathbb{F}^{m\times N_{\mathcal I}},
    \\
    \bm{H}_{\mathcal I}
    &=
    \begin{bmatrix}
      \bm h_{\ell_{\mathrm s}}&\cdots&
      \bm h_{\ell_{\mathrm s}+K_{\mathcal I}-1}
    \end{bmatrix}
    \in\mathbb{F}^{q\times K_{\mathcal I}},
    \\
    \bm{C}_{\mathcal I}
    &=
    \bm{C}_{r^\star}
    (:,\ell_{\mathrm s}:\ell_{\mathrm s}+K_{\mathcal I}-1)
    \in\mathbb{F}^{r^\star\times K_{\mathcal I}}.
\end{aligned}
\label{eq:nlarx_local_data_blocks}
\end{equation}
Local index $k=1,\ldots,K_{\mathcal I}$ corresponds to global Hankel column
$\ell_{\mathrm s}+k-1$; local evolution advances along these columns.

The associated coefficient, Hankel-matrix, and physical measurement
spaces are, respectively:
$\mathscr C_{\mathcal I}=\mathbb F^{r^\star\times K_{\mathcal I}}$,
$\mathscr H_{\mathcal I}=\mathbb F^{q\times K_{\mathcal I}}$, and
$\mathscr Y_{\mathcal I}=\mathbb F^{m\times N_{\mathcal I}}$.
Identification and simulation take place in $\mathscr C_{\mathcal I}$;
candidate assessment takes place in $\mathscr Y_{\mathcal I}$ after
modal reconstruction and anti-diagonal recovery.

\subsubsection{NLARX Dynamics in the Hankel Coefficient Space}

Let $n_a,n_b,n_k\in\mathbb N$, with
$1\leq\ell_{\bm\alpha}<K_{\mathcal I}$, and set
$d_{\bm\alpha}=r^\star(n_a+n_b)$.
The columns of $\bm C_{\mathcal I}$ are the coupled reduced states. For
$\bm\alpha=(n_a,n_b,n_k)$, define the required history length
$\ell_{\bm\alpha}=\max\{n_a,n_k+n_b-1\}$.
The $n_a$ block describes autoregressive memory and the delayed
$(n_b,n_k)$ block inter-coefficient coupling. Other coefficient components
act as exogenous regressors in each scalar equation, yielding a coupled
multi-output NLARX model. 

\begin{definition}[Coupled multi-output NLARX realization of
$\mathcal{M}_{r^\star}$]
\label{def:coupled_multioutput_nlarx}
Let
$\bm{\phi}_{\bm{\alpha}}:\mathbb{F}^{d_{\bm{\alpha}}}
\rightarrow\mathbb{F}^{p_{\bm{\alpha}}}$ be a fixed continuous nonlinear
feature map. For $k\geq\ell_{\bm{\alpha}}$, define
\begin{equation}
\bm{\gamma}_{\bm{\alpha},k}
=
\bm{\phi}_{\bm{\alpha}}\!\bigl(
\bm c_{\mathcal I,k},\ldots,\bm c_{\mathcal I,k-n_a+1};
\bm c_{\mathcal I,k-n_k+1},\ldots,\bm c_{\mathcal I,k-n_k-n_b+2}
\bigr).
\label{eq:nlarx_nonlinear_regressor}
\end{equation}
A feature-linear coupled NLARX model with parameter matrix
$\bm B_{\bm{\alpha}}\in
\mathbb{F}^{r^\star\times p_{\bm{\alpha}}}$ is
\begin{equation}
    \bm c_{\mathcal I,k+1}
    =
    \bm B_{\bm{\alpha}}
    \bm{\gamma}_{\bm{\alpha},k}
    +\bm\eta_k,
    \qquad
    k=\ell_{\bm{\alpha}},\ldots,K_{\mathcal I}-1.
    \label{eq:coupled_multioutput_nlarx}
\end{equation}
Here, $\bm\eta_k\in\mathbb F^{r^\star}$ denotes the
model-discrepancy vector.
The term $\bm B_{\bm\alpha}\bm\gamma_{\bm\alpha,k}$ realizes the
reduced evolution map $\mathcal M_{r^\star}$ of
Definition~\ref{def:finite_horizon_experimental_rom}.
\end{definition}

\begin{remark}
The model \eqref{eq:coupled_multioutput_nlarx} is linear in
$\bm B_{\bm\alpha}$ and nonlinear in its delayed coefficient
arguments. For each fixed
structure, $\lambda_{\mathrm{reg}}>0$ makes the coefficient-space
fitting functional strictly convex in $\bm B_{\bm\alpha}$.
Within the Pareto procedure, inverse calibration assesses each
freely simulated candidate after reconstruction in the physical
measurement space $\mathscr Y_{\mathcal I}$.
\end{remark}

Given the initial history
$\bm c_{\mathcal I,1},\ldots,
\bm c_{\mathcal I,\ell_{\bm{\alpha}}}$, let
\begin{equation}
    \widetilde{\bm C}_{\mathcal I}(\bm B_{\bm{\alpha}})
    =
    \mathscr S_{\mathcal I,\bm{\alpha}}
    (\bm B_{\bm{\alpha}})
    \label{eq:nlarx_recursive_simulator}
\end{equation}
denote the free-run sequence obtained from the observed Phase~I initial
history by recursively applying \eqref{eq:coupled_multioutput_nlarx}
with $\bm\eta_k=\bm0$ and regressors formed from simulated coefficients.
The recursion is identified and simulated in $\mathscr C_{\mathcal I}$, whose
columns correspond to the serialized Hankel representation associated with calibration window
$\mathcal T_H\subseteq\mathcal T_{\mathrm{fin}}$.

\subsubsection{Inverse Reconstruction and Physical-Space Calibration}

After coefficient-space identification and free-run simulation, each
candidate is mapped back to the data channels through the same modes and
recovery rule used in Phase~I.
Let $\mathscr A_{\mathcal I}$ denote local anti-diagonal averaging and
reshaping to $m\times N_{\mathcal I}$ physical channels, as in
\eqref{eq:uniform_antidiagonal_recovery} and
\eqref{eq:recovered_experimental_channels}.
The complete finite-horizon reconstruction operator is
\begin{equation}
\begin{aligned}
    \mathscr R_{\mathcal I,\bm{\alpha}}
    (\bm B_{\bm{\alpha}})
    &=
    \mathscr A_{\mathcal I}\!\left(
       \bm{\Phi}_{r^\star}
       \mathscr S_{\mathcal I,\bm{\alpha}}
       (\bm B_{\bm{\alpha}})
    \right)
.
\end{aligned}
\label{eq:nlarx_complete_reconstruction_operator}
\end{equation}

The reconstruction composes coefficient simulation, modal multiplication,
and anti-diagonal recovery; comparison with $\bm Y_{\mathcal I}$ is performed
in physical space.

\begin{definition}[Physical-space inverse-calibration score on a finite
calibration window]
\label{def:finite_horizon_inverse_calibrated_nlarx}
Let $\mathbb K_{\bm{\alpha}}\subset
\mathbb F^{r^\star\times p_{\bm{\alpha}}}$ be a nonempty admissible parameter
set and let $\lambda_{\mathrm{reg}}>0$. Define
\begin{equation}
\begin{aligned}
    \mathcal J_{\mathcal I,\lambda_{\mathrm{reg}}}
    (\bm B)
    &=
    \frac{
      \left\|
        \bm Y_{\mathcal I}
        -\mathscr R_{\mathcal I,\bm{\alpha}}(\bm B)
      \right\|_F^2
    }{
      \|\bm Y_{\mathcal I}\|_F^2
      +\epsilon_{\mathrm{mach}}
    }
    +
    \lambda_{\mathrm{reg}}\|\bm B\|_{\mathrm F}^2.
\end{aligned}
\label{eq:nlarx_inverse_calibration_functional}
\end{equation}
The misfit compares reconstructed and measured channels in
$\mathscr Y_{\mathcal I}$. The associated auxiliary continuous
parameter-level refinement problem is
\begin{equation}
    \bm B_{\mathcal I,\bm{\alpha}}^\star
    \in
    \operatorname*{arg\,min}_{\bm B\in\mathbb K_{\bm{\alpha}}}
    \mathcal J_{\mathcal I,\lambda_{\mathrm{reg}}}(\bm B)
    \label{eq:nlarx_inverse_calibrated_model}
\end{equation}
and any minimizer is called a \emph{regularized parameter-level
inverse-calibrated refinement} on the finite calibration window
$\mathcal T_H$ in \eqref{eq:finite_calibration_window}.
\end{definition}

Definition~\ref{def:finite_horizon_inverse_calibrated_nlarx} supplies the
auxiliary continuous inverse problem used in the following analysis.
Phase~II selects its NLARX structure by a finite Pareto search, with each
candidate assessed after physical reconstruction on $\mathcal T_H$.

\subsubsection{Existence, Uniqueness, and Inverse Stability of the
Auxiliary Continuous Refinement}

The following analysis concerns the continuous parameter-refinement
problem \eqref{eq:nlarx_inverse_calibrated_model} for a fixed NLARX
structure, modal basis, and initial coefficient history.
Continuity and coercivity give existence; additional curvature and
data-sensitivity assumptions give uniqueness and inverse stability after
physical reconstruction. The finite Pareto search used in the numerical
experiments is specified separately below.

\begin{assumption}[Finite-window admissibility and continuity]
\label{ass:nlarx_admissibility_continuity}
The parameter set $\mathbb K_{\bm{\alpha}}$ is nonempty and closed, the feature
map $\bm\phi_{\bm{\alpha}}$ is continuous, and, for every
$\bm B\in\mathbb K_{\bm{\alpha}}$, the recursive simulation remains finite at
all local coefficient indices $k=1,\ldots,K_{\mathcal I}$ associated with
$\mathcal T_H$.
\end{assumption}

\begin{proposition}[Continuity and coercivity]
\label{prop:nlarx_continuity_coercivity}
Under Assumption~\ref{ass:nlarx_admissibility_continuity},
$\mathscr S_{\mathcal I,\bm\alpha}:\mathbb K_{\bm\alpha}
\rightarrow\mathscr C_{\mathcal I}$ and
$\mathscr R_{\mathcal I,\bm{\alpha}}:\mathbb K_{\bm\alpha}
\rightarrow\mathscr Y_{\mathcal I}$ are continuous. Moreover, for
$\lambda_{\mathrm{reg}}>0$,
$\mathcal J_{\mathcal I,\lambda_{\mathrm{reg}}}$ is continuous and coercive:
$\|\bm B\|_{\mathrm F}\to\infty$ implies
$\mathcal J_{\mathcal I,\lambda_{\mathrm{reg}}}(\bm B)\to\infty$.
\end{proposition}

\begin{proof}
Induction over the finite recursion gives continuity of
$\mathscr S_{\mathcal I,\bm\alpha}$ in $\bm B$.
Modal multiplication and anti-diagonal recovery are continuous linear
maps, so $\mathscr R_{\mathcal I,\bm\alpha}$ and
$\mathcal J_{\mathcal I,\lambda_{\mathrm{reg}}}$ are continuous.
The bound $\mathcal J_{\mathcal I,\lambda_{\mathrm{reg}}}(\bm B)
\geq\lambda_{\mathrm{reg}}\|\bm B\|_{\mathrm F}^2$ gives coercivity.
\end{proof}

\begin{assumption}[Finite-window smoothness and curvature control]
\label{ass:nlarx_finite_horizon_identifiability}
Let $\mathbb K_{\bm\alpha}$ be nonempty, compact, and convex, and keep the
selected modes and prescribed initial coefficient history fixed. Assume that
$\mathscr R_{\mathcal I,\bm\alpha}$ is twice continuously differentiable
with respect to $\bm B$ on a neighborhood of $\mathbb K_{\bm\alpha}$. For
a calibration matrix $\bm Y$ near the measured matrix $\bm Y_{\mathcal I}$,
let $\mathcal J_{\mathcal I,0}(\bm B;\bm Y)$ denote the unregularized
reconstruction-error term in \eqref{eq:nlarx_inverse_calibration_functional},
with $\bm Y_{\mathcal I}$ replaced by $\bm Y$.

Assume that there exist finite constants $\rho_{\mathcal I}\geq0$ and
$L_{\mathcal I}\geq0$ such that, for every
$\bm B\in\mathbb K_{\bm\alpha}$, parameter increment $\bm\Delta$, and
calibration matrices $\bm Y$, $\bm Y_1$, and $\bm Y_2$ in a bounded
neighborhood of $\bm Y_{\mathcal I}$,
\begin{equation}
\begin{aligned}
    D_{\bm B}^2\mathcal J_{\mathcal I,0}
    (\bm B;\bm Y)[\bm\Delta,\bm\Delta]
    &\geq
    -2\rho_{\mathcal I}\|\bm\Delta\|_{\mathrm F}^2,
    \\
    \left\|
      \nabla_{\bm B}\mathcal J_{\mathcal I,0}(\bm B;\bm Y_1)
      -\nabla_{\bm B}\mathcal J_{\mathcal I,0}(\bm B;\bm Y_2)
    \right\|_{\mathrm F}
    &\leq
    2L_{\mathcal I}
    \|\bm Y_1-\bm Y_2\|_F.
\end{aligned}
\label{eq:nlarx_finite_window_sensitivity_bounds}
\end{equation}
Derivatives are taken with respect to $\bm B$, using real and imaginary
parts when $\mathbb F=\mathbb C$.
Here, $\rho_{\mathcal I}$ controls negative curvature and
$L_{\mathcal I}$ controls sensitivity of the parameter gradient to data.
Assume $\lambda_{\mathrm{reg}}>\rho_{\mathcal I}$.
\end{assumption}

\begin{theorem}[Existence, uniqueness, and inverse stability of the
auxiliary continuous inverse-calibration solution]
\label{thm:nlarx_existence_uniqueness_inverse_stability}
Fix an admissible calibration window and an NLARX structure.
Under Assumption~\ref{ass:nlarx_admissibility_continuity},
\eqref{eq:nlarx_inverse_calibrated_model} has a minimizer for every
$\lambda_{\mathrm{reg}}>0$.
If Assumption~\ref{ass:nlarx_finite_horizon_identifiability} also holds,
the minimizer $\bm B^\star(\bm Y)$ is unique for each admissible
calibration matrix $\bm Y$, and
\begin{equation}
    \|\bm B^\star(\bm Y_1)-\bm B^\star(\bm Y_2)\|_{\mathrm F}
    \leq
    \frac{L_{\mathcal I}}{\lambda_{\mathrm{reg}}-\rho_{\mathcal I}}
    \|\bm Y_1-\bm Y_2\|_F.
    \label{eq:nlarx_inverse_stability_bound}
\end{equation}
The modal basis and initial coefficient history are held fixed under
these data perturbations.
\end{theorem}

\begin{proof}
By Proposition~\ref{prop:nlarx_continuity_coercivity}, any minimizing
sequence is bounded. In finite dimensions it has a convergent
subsequence; closedness of $\mathbb K_{\bm\alpha}$ and continuity of the
functional make the limit a minimizer.

Under the second assumption,
\begin{equation}
    D_{\bm B}^2\mathcal J_{\mathcal I,\lambda_{\mathrm{reg}}}
    (\bm B;\bm Y)[\bm\Delta,\bm\Delta]
    \geq 2(\lambda_{\mathrm{reg}}-\rho_{\mathcal I})
    \|\bm\Delta\|_{\mathrm F}^2.
    \label{eq:nlarx_inverse_functional_strong_convexity}
\end{equation}
Convexity of the admissible set and
$\lambda_{\mathrm{reg}}>\rho_{\mathcal I}$ therefore give a unique minimizer.
Let $\bm B_i=\bm B^\star(\bm Y_i)$ and
$\bm\Delta=\bm B_1-\bm B_2$.
The constrained first-order optimality inequalities imply
$\langle\nabla_{\bm B}\mathcal J(\bm B_1;\bm Y_1)
-\nabla_{\bm B}\mathcal J(\bm B_2;\bm Y_2),\bm\Delta\rangle_F\leq0$,
where $\mathcal J=\mathcal J_{\mathcal I,\lambda_{\mathrm{reg}}}$
and the real Frobenius pairing is used for complex parameters.
Adding and subtracting $\nabla_{\bm B}\mathcal J(\bm B_2;\bm Y_1)$,
then using strong monotonicity and the data-sensitivity assumption, gives
\begin{equation}
    2(\lambda_{\mathrm{reg}}-\rho_{\mathcal I})
    \|\bm\Delta\|_{\mathrm F}^2
    \leq 2L_{\mathcal I}\|\bm Y_1-\bm Y_2\|_F
    \|\bm\Delta\|_{\mathrm F}.
    \label{eq:nlarx_inverse_variational_stability}
\end{equation}
For $\bm\Delta\neq\bm0$, division yields
\eqref{eq:nlarx_inverse_stability_bound}; the zero case is immediate.
\end{proof}

\subsubsection{Pareto-Based Inverse-Calibrated Model Selection}

For each feasible structure, coefficient-space fitting is followed by
free-run simulation and physical reconstruction. The resulting channels
are compared with the calibration measurements using the objectives below.

For a feasible structure $\bm\alpha$, let
$M_{\bm\alpha}=K_{\mathcal I}-\ell_{\bm\alpha}>0$.
Within the Pareto search, each fixed-structure candidate is fitted in
coefficient space using $\lambda_{\mathrm{reg}}>0$ and the functional
\begin{equation}
    \mathcal Q_{\mathcal I,\bm\alpha}(\bm B)
    =
    \sum_{k=\ell_{\bm\alpha}}^{K_{\mathcal I}-1}
    \|\bm c_{\mathcal I,k+1}
    -\bm B\bm\gamma_{\bm\alpha,k}\|_2^2
    +\lambda_{\mathrm{reg}}\|\bm B\|_{\mathrm F}^2.
    \label{eq:nlarx_regularized_identification_functional}
\end{equation}

Let $\mathfrak A_{\mathcal I}$ be a finite, nonempty family of feasible
structures $\bm\alpha=(n_a,n_b,n_k)$, with $M_{\bm\alpha}>0$ and one
candidate per structure. The functional
\eqref{eq:nlarx_regularized_identification_functional} has a unique minimizer
$\bm B_{\lambda_{\mathrm{reg}},\bm\alpha}$.
For brevity, write $\lambda=\lambda_{\mathrm{reg}}$. After recursive
simulation and physical reconstruction, let
\begin{equation}
    \widehat{\bm Y}_{\mathcal I,\bm\alpha}
    =
    \mathscr R_{\mathcal I,\bm\alpha}
    (\bm B_{\lambda,\bm\alpha}),
    \qquad
    \bm E_{\bm\alpha}
    =
    \bm Y_{\mathcal I}
    -
    \widehat{\bm Y}_{\mathcal I,\bm\alpha}.
    \label{eq:nlarx_candidate_reconstruction_residual}
\end{equation}
For any two vectors $\bm a,\bm b\in\mathbb R^n$, define the Pearson
correlation coefficient by
\begin{equation}
    \varrho_n(\bm a,\bm b)
    =\begin{cases}
      \displaystyle\frac{\langle\bm a^c,\bm b^c\rangle}
      {\|\bm a^c\|_2\|\bm b^c\|_2},
      &\|\bm a^c\|_2\|\bm b^c\|_2>0,\\[2mm]
      0,&\|\bm a^c\|_2\|\bm b^c\|_2=0.
    \end{cases}
\label{eq:finite_pearson_correlation}
\end{equation}
Here, $\bm a^c=\bm a-\bar a\bm1_n$ and
$\bm b^c=\bm b-\bar b\bm1_n$, where $\bar a$ and $\bar b$ are the
sample means and $\bm1_n=[1,\ldots,1]^{\mathsf T}\in\mathbb R^n$
denotes the vector of ones.

For Pareto assessment, assume $N_{\mathcal I}\geq2$
and that the measured and reconstructed physical channel
matrices are real-valued.
The four minimized objectives are
\begin{equation}
\begin{aligned}
    f_1(\bm\alpha)
    &=
    \frac{
      \|\bm E_{\bm\alpha}\|_F
    }{
      \|\bm Y_{\mathcal I}\|_F
      +\epsilon_{\mathrm{mach}}
    },
    \\
    f_2(\bm\alpha)
    &=
    1-\frac{1}{m}
    \sum_{i=1}^{m}
    \varrho_{N_{\mathcal I}}\!\left(
       \bm Y_{\mathcal I}(i,:)^{\mathsf T},
       \widehat{\bm Y}_{\mathcal I,\bm\alpha}(i,:)^{\mathsf T}
    \right),
    \\
    f_3(\bm\alpha)
    &=
    \frac{1}{m}
    \sum_{i=1}^{m}
    \left|
    \varrho_{N_{\mathcal I}-1}\!\left(
       \bm E_{\bm\alpha}(i,1:N_{\mathcal I}-1)^{\mathsf T},
       \bm E_{\bm\alpha}(i,2:N_{\mathcal I})^{\mathsf T}
    \right)
    \right|,
    \\
    f_4(\bm\alpha)
    &=
    \frac{\|\bm B_{\lambda,\bm\alpha}\|_{\mathrm F}}
         {1+\|\bm B_{\lambda,\bm\alpha}\|_{\mathrm F}},
    \qquad
    \bm f_{\mathcal I}(\bm\alpha)
    =
    \begin{bmatrix}f_1&f_2&f_3&f_4\end{bmatrix}^{\mathsf T},
\end{aligned}
\label{eq:nlarx_pareto_objectives}
\end{equation}
The objectives measure relative reconstruction error, mean correlation
loss, residual autocorrelation, and parameter magnitude, respectively.

\begin{definition}[Pareto-optimal inverse-calibrated NLARX structure]
\label{def:nlarx_pareto_optimality}
Using the dominance criterion in
Definition~\ref{def:pareto_optimal_dimension} for the four objectives
in \eqref{eq:nlarx_pareto_objectives}, let
$\mathfrak P_{\mathcal I}$ denote the nondominated subset of
$\mathfrak A_{\mathcal I}$.
\end{definition}

Let $\widehat f_j$ be objective $f_j$ normalized over
$\mathfrak P_{\mathcal I}$, with a zero normalized value when the objective
range is zero. For prescribed weights $w_j>0$, $\sum_{j=1}^{4}w_j=1$, define
\begin{equation}
    \Psi_{\mathcal I}(\bm\alpha)
    =
    \sum_{j=1}^{4}w_j\widehat f_j(\bm\alpha).
    \label{eq:nlarx_pareto_decision_score}
\end{equation}

The structure is selected by minimizing the decision score, with ties resolved successively
by $f_1,\ldots,f_4$ and then by $n_a,n_b,n_k$:
\begin{equation}
\bm\alpha_{\mathrm{Pareto}}^\star
=\operatorname*{arg\,min}_{\bm\alpha\in\mathfrak P_{\mathcal I}}
\bigl(\Psi_{\mathcal I}(\bm\alpha),
f_1(\bm\alpha),\ldots,f_4(\bm\alpha),n_a,n_b,n_k\bigr).
\label{eq:nlarx_pareto_structure_selection}
\end{equation}

Rule \eqref{eq:nlarx_pareto_structure_selection} uniquely selects the inverse-calibrated model
\begin{equation}
    \bm B_{\mathrm{Pareto}}
    =
    \bm B_{\lambda,\bm\alpha_{\mathrm{Pareto}}^\star},
    \qquad
    \widehat{\bm Y}_{\mathcal I,\mathrm{Pareto}}
    =
    \mathscr R_{
    \mathcal I,\bm\alpha_{\mathrm{Pareto}}^\star}
    (\bm B_{\mathrm{Pareto}}).
\label{eq:nlarx_unique_pareto_output}
\end{equation}

\begin{remark}[Scope of the mathematical guarantees]
Strictly convex candidate fitting and deterministic tie resolution give
a unique Pareto-selected model. Theorem~\ref{thm:nlarx_existence_uniqueness_inverse_stability}
concerns the auxiliary continuous refinement with a fixed structure and
basis. 
\end{remark}

\subsubsection{Algorithmic Formulation and Computational Complexity}

Algorithm~\ref{alg:nlarx_pareto_inverse_calibration} combines regularized
candidate identification, recursive coefficient simulation, physical-space
reconstruction, and Pareto model selection. 
The selected reconstructed channels supply the inputs for Phase~III.
Its computational and memory
requirements are summarized in Table~\ref{tab:framework_complexity}.

\begin{breakablealgorithm}
{Pareto-based inverse-calibrated NLARX model}
\label{alg:nlarx_pareto_inverse_calibration}
\begin{algorithmic}[1]
\Require $\bm Y$, $\bm\Phi_{r^\star}$,
         $\bm C_{r^\star}$, $q$, admissible calibration window
         $(\mathcal T_H,\mathcal I_H)$ with
         $\mathcal T_H\subseteq\mathcal T_{\mathrm{fin}}$, finite feasible family
         $\mathfrak A_{\mathcal I}$, $\lambda_{\mathrm{reg}}>0$, and positive
         decision weights $\{w_j\}_{j=1}^{4}$ with
         $\sum_{j=1}^{4}w_j=1$.
\Ensure Unique Pareto-selected structure
        $\bm\alpha^\star_{\mathrm{Pareto}}$, parameter matrix
        $\bm B_{\mathrm{Pareto}}$, simulated coefficients, and reconstructed
        channels $\widehat{\bm Y}_{\mathcal I,\mathrm{Pareto}}$.
\State Extract $\bm Y_{\mathcal I}$ and $\bm C_{\mathcal I}$ according to
       \eqref{eq:nlarx_local_data_blocks}.
\ForAll{$\bm\alpha\in\mathfrak A_{\mathcal I}$}
    \State Fit the coefficient-space candidate by minimizing
           \eqref{eq:nlarx_regularized_identification_functional}.
    \State Simulate the coefficients recursively and reconstruct the
           physical channels using \eqref{eq:nlarx_complete_reconstruction_operator}.
    \State After physical reconstruction, evaluate the four objectives in
           \eqref{eq:nlarx_pareto_objectives}.
\EndFor
\State Determine $\mathfrak P_{\mathcal I}$, normalize the objectives,
       and select the unique structure using
       \eqref{eq:nlarx_pareto_structure_selection}.
\State \Return $\bm\alpha^\star_{\mathrm{Pareto}}$,
       $\bm B_{\mathrm{Pareto}}$,
       $\widetilde{\bm C}_{\mathcal I,\mathrm{Pareto}}$, and
       $\widehat{\bm Y}_{\mathcal I,\mathrm{Pareto}}$.
\end{algorithmic}
\end{breakablealgorithm}

\subsection{Phase III: Finite-Horizon Recursive Forecasting}
\label{subsec:phase_forecasting}

NLARX models use delayed outputs and exogenous variables to represent
nonlinear response dynamics, with applications in environmental and
photovoltaic forecasting \cite{Balcha2025,Fentis2019,Moursi2021,
Chen2026,Ouyang2025}.
Here, the quantity-of-interest model is identified from paired experimental
quantity--channel samples on $\mathcal T_{\mathrm{obs}}$.
During forecasting, the physical channels reconstructed from the
Phase~II coefficient simulations supply its exogenous inputs on
$\mathcal T_{\mathrm{for}}\subseteq\mathcal T_H$.
The temporal relations are those in
\eqref{eq:finite_temporal_set_relations}.
This section specifies the scalar model and its recursive composition,
then establishes the finite-horizon stability of the resulting forecast.

\subsubsection{Reconstructed Physical Measurements and Quantity-of-Interest Model}

The selected Phase~II channel reconstruction is given by
\eqref{eq:nlarx_complete_reconstruction_operator}. Its columns indexed
by $\mathcal T_{\mathrm{for}}$ form the exogenous input matrix
$\widehat{\bm Y}_{\mathrm{for}}
=[\widehat{\bm y}_{N_Q+1},\ldots,\widehat{\bm y}_{N_Q+N_{\mathrm f}}]
\in\mathbb F^{m\times N_{\mathrm f}}$.
Here, $\widehat{\bm y}_{N_Q+h}$ is the channel vector reconstructed in
Phase~II from the simulated coefficient dynamics and supplied at forecast
step $h$. Channel and quantity samples use the same time index.

\begin{definition}[Channel-dependent quantity-of-interest model]
\label{def:phase3_quantity_of_interest}
Let $Q_{\mathrm I}:\mathbb F^m\rightarrow\mathbb R$ denote a scalar response
associated with the coupled experimental channels, with sampled values
\begin{equation}
    q_k=Q_{\mathrm I}(\bm y_k)+\xi_k,
    \label{eq:phase3_qoi_samples}
\end{equation}
where $\xi_k$ represents measurement error and effects not resolved by the
channel vector. The functional form of $Q_{\mathrm I}$ need not be known. A
\emph{quantity-of-interest scalar-output model} is a data-identified scalar map that
advances $q_k$ from its own past values and from current or delayed channel
values. \end{definition}

Let $q_1,\ldots,q_{N_Q}$ and $\bm y_1,\ldots,\bm y_{N_Q}$ be the paired
quantity and experimentally measured channel samples on
$\mathcal T_{\mathrm{obs}}$, with $1\leq N_Q<N$.

For the scalar NLARX order triple
$\bm\beta=(n_a^Q,n_b^Q,n_k^Q)$, with
$n_a^Q,n_b^Q\in\mathbb N$ and $n_k^Q\in\mathbb N_0$,
we define
\begin{equation}
    \ell_{\bm\beta}
    =\max\{n_a^Q,n_k^Q+n_b^Q-1\},
    \qquad
    d_{\bm\beta}=n_a^Q+mn_b^Q.
    \label{eq:phase3_qoi_history_dimension}
\end{equation}
A structure is feasible when $\ell_{\bm\beta}<N_Q$.
For $k=\ell_{\bm\beta}+1,\ldots,N_Q$, the regressor is
\begin{equation}
\begin{aligned}
    \bm z_{\bm\beta,k}
    &=\begin{bmatrix}
      q_{k-1}&\cdots&q_{k-n_a^Q}&
      \bm y_{k-n_k^Q}^{\mathsf T}&\cdots&
      \bm y_{k-n_k^Q-n_b^Q+1}^{\mathsf T}
    \end{bmatrix}^{\mathsf T},\\
    \bm\varphi_{\bm\beta,k}^Q
    &=\bm\varphi_{\bm\beta}^Q(\bm z_{\bm\beta,k}).
\end{aligned}
\label{eq:phase3_qoi_regressor_and_features}
\end{equation}
Here, $\bm\varphi_{\bm\beta}^Q$ is a fixed, continuously differentiable feature map
$\bm\varphi_{\bm\beta}^Q:\mathbb R^{d_{\bm\beta}}\to\mathbb R^{p_{\bm\beta}}$,
with $p_{\bm\beta}=4d_{\bm\beta}+2$ in the implementation.
The scalar NLARX realization is
\begin{equation}
    q_k
    =
    \bm\theta_{\bm\beta}^{\mathsf T}
    \bm\varphi_{\bm\beta,k}^{Q}+\varepsilon_k^Q.
    \label{eq:phase3_qoi_nlarx_model}
\end{equation}
Thus, $n_a^Q$ controls the memory of the quantity-of-interest, $n_b^Q$ the
channel-memory depth, and $n_k^Q$ the channel delay. The value $n_k^Q=0$
allows the current channel vector to drive the current response: an
experimental vector during identification and a Phase~II-reconstructed
physical-channel vector during forecasting.

\subsubsection{Correlation-Based Pareto Identification of Quantity-of-Interest Scalar-Output Model}

Let $M_{\bm\beta}=N_Q-\ell_{\bm\beta}>0$.
For $\lambda_Q>0$, the unique ridge candidate associated with
$\bm\beta$ minimizes
\begin{equation}
    \sum_{k=\ell_{\bm\beta}+1}^{N_Q}
    \left|q_k-\bm\theta^{\mathsf T}
    \bm\varphi_{\bm\beta,k}^Q\right|^2
    +\lambda_Q\|\bm\theta\|_2^2,
    \qquad \bm\theta\in\mathbb R^{p_{\bm\beta}}.
    \label{eq:phase3_qoi_ridge_solution}
\end{equation}
Entries of the ridge candidate below the magnitude threshold $\tau_Q$ are
set to zero; $\bm\theta_{\lambda_Q,\bm\beta}$ subsequently denotes this
thresholded candidate.

Each candidate is simulated freely on $\mathcal T_{\mathrm{obs}}$ from
the first $\ell_{\bm\beta}$ observed quantity values, using the measured
identification channels throughout. Denote its simulated trajectory by
$\widehat{\bm q}_{Q,\bm\beta}=[\widehat q_{1,\bm\beta},\ldots,
\widehat q_{N_Q,\bm\beta}]$.

Let $\bm q_Q=[q_1,\ldots,q_{N_Q}]$ satisfy
$\|\bm q_Q\|_2>0$, and define
$\bm e_{Q,\bm\beta}
=\bm q_Q-\widehat{\bm q}_{Q,\bm\beta}$.
The four minimized objectives are
\begin{equation}
\begin{aligned}
    f_1^Q(\bm\beta)
    &=1-\varrho_{N_Q}\!\left(
       \bm q_Q^{\mathsf T},
       \widehat{\bm q}_{Q,\bm\beta}^{\mathsf T}
      \right),
    \\
    f_2^Q(\bm\beta)
    &=\frac{\|\bm e_{Q,\bm\beta}\|_2}{\|\bm q_Q\|_2},
    \\
    f_3^Q(\bm\beta)
    &=\left|
      \varrho_{N_Q-1}\!\left(
       \bm e_{Q,\bm\beta}(1:N_Q-1)^{\mathsf T},
       \bm e_{Q,\bm\beta}(2:N_Q)^{\mathsf T}
      \right)
      \right|,
    \\
    f_4^Q(\bm\beta)
    &=\frac{\|\bm\theta_{\lambda_Q,\bm\beta}\|_2}
            {1+\|\bm\theta_{\lambda_Q,\bm\beta}\|_2},
\end{aligned}
\label{eq:phase3_qoi_pareto_objectives}
\end{equation}
representing temporal-shape correlation loss, relative amplitude error,  residual
persistence, and parameter magnitude, respectively.
The Pearson coefficients are defined in \eqref{eq:finite_pearson_correlation}.

Let $\mathfrak B_Q$ be a finite, nonempty family of feasible
structures, and let $\mathfrak P_Q$ be its nondominated
subset under the objectives in
\eqref{eq:phase3_qoi_pareto_objectives}. 
The final correlation-primary score
is
\begin{equation}
    S_Q(\bm\beta)
    =
    \mathcal{C}_1 f_1^Q(\bm\beta)+\mathcal{C}_2 f_2^Q(\bm\beta)
    +\mathcal{C}_3 f_3^Q(\bm\beta)+\mathcal{C}_4 f_4^Q(\bm\beta),
    \quad \bm\beta\in\mathfrak P_Q,\; \mathcal{C}_i\in \mathbb{R}_{+}.
    \label{eq:phase3_qoi_final_score}
\end{equation}

\begin{remark}[Correlation-primary decision]
The coefficient hierarchy
$\mathcal C_1\gg\mathcal C_2>\mathcal C_3>\mathcal C_4$
in \eqref{eq:phase3_qoi_final_score} assigns a distinctly dominant role to
temporal-shape correlation in the final decision, while relative amplitude
error, residual persistence, and parameter magnitude remain active
secondary criteria. 
Because Pareto filtering precedes the scalar score, a candidate that is strictly
inferior in all four objectives cannot be selected by the weighting rule.
\end{remark}

\subsubsection{Composed Finite-Horizon Forecasting of Quantity-of-Interest}

Let $N_{\mathrm f}\in\mathbb N$ denote the prescribed number of forecasting
steps on $\mathcal T_{\mathrm{for}}$. Each step $h=1,\ldots,N_{\mathrm f}$
corresponds to one sampling instant, one recursively generated value of the
quantity-of-interest, and one application of the selected one-step NLARX
time-advance operator.

To define the delayed channel inputs during the first forecast steps,
the future driver sequence is extended to earlier indices by repeating its
first available column:
\begin{equation}
    \widehat{\bm y}^{\,\mathrm{ext}}_k
    =
    \begin{cases}
        \widehat{\bm y}_{N_Q+1}, & k\leq N_Q,\\
        \widehat{\bm y}_k, & k=N_Q+1,\ldots,N_Q+N_{\mathrm f}.
    \end{cases}
    \label{eq:phase3_extended_future_drivers}
\end{equation}
At forecast step $h$, the quantity-history state
$\widehat{\bm s}_{h-1}$ contains the last $n_a^{Q,\star}$
quantity values, ordered from most recent to oldest.
The driver block $\widehat{\bm w}_h$ stacks
$\widehat{\bm y}^{\,\mathrm{ext}}_{N_Q+h-n_k^{Q,\star}-j}$,
$j=0,\ldots,n_b^{Q,\star}-1$, in increasing order of $j$.

The selected one-step NLARX time-advance operator is
\begin{equation}
    \mathcal F_{Q,h}^\star(\bm s;\bm w)
    =
    \begin{bmatrix}
      (\bm\theta_Q^\star)^{\mathsf T}
      \bm\varphi_Q\!\left(
        [\bm s^{\mathsf T},\bm w^{\mathsf T}]^{\mathsf T}
      \right)
      \\
      \bm s_{1:n_a^{Q,\star}-1}
    \end{bmatrix},
    \qquad h=1,\ldots,N_{\mathrm f},
    \label{eq:phase3_qoi_time_advance_operator}
\end{equation}
where $\bm s_{1:n_a^{Q,\star}-1}$ denotes the first
$n_a^{Q,\star}-1$ state entries, the shift block is empty when
$n_a^{Q,\star}=1$, and
$\bm\varphi_Q$ has the feature structure in
\eqref{eq:phase3_qoi_regressor_and_features} selected by
$\bm\beta_Q^\star$. The recursive state and the quantity forecast are then
\begin{equation}
\begin{aligned}
    \widehat{\bm s}_h
    &=\mathcal F_{Q,h}^\star
      (\widehat{\bm s}_{h-1};\widehat{\bm w}_h),\\
    \widehat q_{N_Q+h}
    &=(\widehat{\bm s}_h)_1,
    \qquad h=1,\ldots,N_{\mathrm f}.
\end{aligned}
\label{eq:phase3_qoi_recursive_forecast}
\end{equation}
The initial state is
$\widehat{\bm s}_0=[q_{N_Q},\ldots,q_{N_Q-n_a^{Q,\star}+1}]^{\mathsf T}$.

\begin{definition}[Finite-horizon numerical forecasting by operator composition]
\label{def:phase3_numerical_forecasting}
For the fixed reconstructed physical driver blocks supplied by Phase~II,
define the partial
forecast composition by
\begin{equation}
    \mathcal C_{Q,h}^\star
    =
    \mathcal F_{Q,h}^\star(\,\cdot\,;\widehat{\bm w}_h)
    \circ\cdots\circ
    \mathcal F_{Q,1}^\star(\,\cdot\,;\widehat{\bm w}_1),
    \qquad h=1,\ldots,N_{\mathrm f},
    \label{eq:phase3_numerical_operator_composition}
\end{equation}
The finite-horizon numerical forecast on
$\mathcal T_{\mathrm{for}}$ is the trajectory obtained by applying
$\mathcal C_{Q,h}^\star$ to the terminal observed quantity-history state for
$h=1,\ldots,N_{\mathrm f}$ and extracting its first component as in
\eqref{eq:phase3_qoi_recursive_forecast}. Hence the forecast is
a finite ordered composition of scalar NLARX time-advance operators driven by
the physical measurements reconstructed in Phase~II from the simulated
coefficient dynamics.
\end{definition}

\begin{theorem}[Finite-horizon stability of the composed numerical forecast]
\label{thm:phase3_composed_forecast_stability}
Fix the selected NLARX structure and parameters.
For $i=1,2$, let $\bm s_h^{(i)}$ be the forecast state obtained from
$\bm s_0^{(i)}$ and the driver sequence $\{\bm w_h^{(i)}\}$ by the
operator composition in Definition~\ref{def:phase3_numerical_forecasting}.
Assume that the visited state--driver pairs lie in a compact convex set
on a neighborhood of which $\bm\varphi_Q$ is continuously differentiable.
Let $L_s^Q,L_w^Q\geq0$ be uniform Lipschitz constants for the one-step
operator with respect to its state and driver arguments on this set.
Then, for $h=1,\ldots,N_{\mathrm f}$,
\begin{equation}
\begin{aligned}
    \|\bm s_h^{(1)}-\bm s_h^{(2)}\|_2
    \leq{}&(L_s^Q)^h\|\bm s_0^{(1)}-\bm s_0^{(2)}\|_2\\
    &+L_w^Q\sum_{j=1}^{h}(L_s^Q)^{h-j}
      \|\bm w_j^{(1)}-\bm w_j^{(2)}\|_2.
\end{aligned}
\label{eq:phase3_composed_forecast_stability}
\end{equation}
The same right-hand side bounds
$|\widehat q_{N_Q+h}^{(1)}-\widehat q_{N_Q+h}^{(2)}|$.
For prescribed initial history and drivers, the forecast is uniquely
defined and Lipschitz stable over the finite horizon.
\end{theorem}

\begin{proof}
Continuous differentiability on a neighborhood of the compact convex
set gives bounded partial derivatives. The mean-value inequality
therefore supplies the finite constants $L_s^Q$ and $L_w^Q$.
Writing $e_h=\|\bm s_h^{(1)}-\bm s_h^{(2)}\|_2$, the one-step estimate gives
\begin{equation}
    e_h\leq L_s^Q e_{h-1}
    +L_w^Q\|\bm w_h^{(1)}-\bm w_h^{(2)}\|_2.
    \label{eq:phase3_forecast_error_recursion}
\end{equation}
Induction yields \eqref{eq:phase3_composed_forecast_stability}, with
$(L_s^Q)^0=1$, including when $L_s^Q=0$.
Since the quantity forecast is the first state component,
$|\widehat q_{N_Q+h}^{(1)}-\widehat q_{N_Q+h}^{(2)}|\leq e_h$.
Each one-step map is single-valued, so its finite composition gives a
unique trajectory. All terms in the estimate are finite.
\end{proof}

The estimate describes propagation of perturbations in the initial
quantity history and reconstructed channel inputs through a fixed
selected model. Forecast accuracy against the quantity of interest is
assessed separately in the numerical experiments.

The ratio between the number of paired samples used to identify the scalar
NLARX model and the number of recursively forecasted samples is
\begin{equation}
    \mathcal R_{Q/\mathrm{for}}
    =\frac{N_Q}{N_{\mathrm f}}.
    \label{eq:phase3_calibration_forecast_ratio}
\end{equation}

This ratio records the number of identification samples per forecast
step; $N_{\mathrm f}$ gives the number of successive operator applications.
It is an experimental-design indicator, while sensitivity of the fixed
forecast model is described by
Theorem~\ref{thm:phase3_composed_forecast_stability}.

\subsubsection{Algorithmic Formulation and Computational Complexity}

Algorithm~\ref{alg:phase3_qoi_identification} identifies and uniquely selects
the quantity-of-interest scalar-output model from the paired experimental quantity and
channel measurements on $\mathcal T_{\mathrm{obs}}$.
Algorithm~\ref{alg:phase3_qoi_forecast} then advances the selected scalar model
recursively using the physical channel inputs reconstructed in Phase~II on
$\mathcal T_{\mathrm{for}}$.

The identification search and recursive forecasting costs are summarized
in Table~\ref{tab:framework_complexity}. Once the selected model and
reconstructed drivers are available, the forecast cost grows linearly
with the prescribed number of steps.

\begin{breakablealgorithm}
{Correlation-primary Pareto identification of the
quantity-of-interest NLARX model}
\label{alg:phase3_qoi_identification}
\begin{algorithmic}[1]     
\Require Experimental quantity samples $q_1,\ldots,q_{N_Q}$
         and corresponding experimentally measured channel
         samples $\bm y_1,\ldots,\bm y_{N_Q}$ on
         $\mathcal T_{\mathrm{obs}}
         \subseteq\mathcal T_{\mathrm{fin}}$;
         finite, nonempty feasible order family $\mathfrak B_Q$;
         $\lambda_Q>0$; threshold $\tau_Q\geq0$;
         prescribed decision weights
         $\mathcal C_i>0$, $i=1,\ldots,4$.
\Ensure Unique selected structure $\bm\beta_Q^\star$, parameter vector
        $\bm\theta_Q^\star$, selected free-run
        simulation, complete candidate table, and Pareto table.
\ForAll{$\bm\beta\in\mathfrak B_Q$}
    \State Fit the ridge candidate using \eqref{eq:phase3_qoi_ridge_solution}
           and apply the threshold $\tau_Q$.
    \State Simulate freely from the prescribed initial history using
           observed identification channels and evaluate
           \eqref{eq:phase3_qoi_pareto_objectives}.
\EndFor
\State Determine the nondominated set $\mathfrak P_Q$.
\State Minimize \eqref{eq:phase3_qoi_final_score} on $\mathfrak P_Q$ and
       resolve exact ties by the predefined candidate order.
\State \Return $\bm\beta_Q^\star$, $\bm\theta_Q^\star$, selected simulation, and result tables.
\end{algorithmic}
\end{breakablealgorithm}

\begin{breakablealgorithm}
{Composed finite-horizon numerical forecast of the quantity-of-interest}
\label{alg:phase3_qoi_forecast}
\begin{algorithmic}[1]
\Require Selected quantity model
         $(\bm\beta_Q^\star,\bm\theta_Q^\star)$;
         terminal quantity history;
         physical channel inputs reconstructed in Phase~II,
         $\widehat{\bm Y}_{\mathrm{for}}$ on
         $\mathcal T_{\mathrm{for}}\subseteq\mathcal T_H$;
         $N_{\mathrm f}\in\mathbb N$ forecast steps.
\Ensure Quantity forecast
        $[\widehat q_{N_Q+1},\ldots,
          \widehat q_{N_Q+N_{\mathrm f}}]$ on
        $\mathcal T_{\mathrm{for}}$.
\State Initialize $\widehat{\bm s}_0$ with the terminal observed quantity history.
\For{$h=1,\ldots,N_{\mathrm f}$}
    \State Assemble the delayed channel block $\widehat{\bm w}_h$ using
       \eqref{eq:phase3_extended_future_drivers}.
    \State Advance the forecast state using
           \eqref{eq:phase3_qoi_recursive_forecast}.
\EndFor
\State \Return the quantity-of-interest forecast values.
\end{algorithmic}
\end{breakablealgorithm}

\subsection{End-to-End Algorithm for Data-Driven Twin Modeling and Forecasting}\label{End-to-End}

Algorithm~\ref{alg:end_to_end_data_driven_twin} assembles the three phases
under the temporal relations \eqref{eq:finite_temporal_set_relations}, met by the finite windows used in this phase-to-phase coupling. The output of each phase supplies the input required by the next one.
Theorem~\ref{thm:finite_horizon_hk_energy} supports energy ranking,
and Theorem~\ref{thm:nlarx_existence_uniqueness_inverse_stability}
analyzes the auxiliary continuous refinement.
Once the NLARX structures and parameters have been selected,
Theorem~\ref{thm:phase3_composed_forecast_stability} bounds the sensitivity
of the composed quantity forecast to its initial history and channel inputs.

\begin{breakablealgorithm}
{End-to-end data-driven twin model and forecasting workflow}
\label{alg:end_to_end_data_driven_twin}
\begin{algorithmic}[1]
\Require Multichannel measurements $\bm Y$ on $\mathcal T_{\mathrm{fin}}$;
quantity samples on $\mathcal T_{\mathrm{obs}}$; temporal sets satisfying
\eqref{eq:finite_temporal_set_relations}; and phase-specific parameters.
\Ensure Selected modal representation, coefficient simulator, reconstructed
channels, scalar quantity model, and finite-horizon forecast.
\State Apply Algorithm~\ref{alg:hkfed} to obtain
$\mathfrak T_{r^\star}=(\bm\Lambda_{G,r^\star},
\bm\Phi_{r^\star},\bm C_{r^\star})$.
\State On the admissible calibration window $\mathcal T_H$, apply
Algorithm~\ref{alg:nlarx_pareto_inverse_calibration} to obtain
$\bm\alpha^\star$, $\bm B^\star$, simulated coefficients
$\widetilde{\bm C}_{\mathcal I}^\star$, and reconstructed channels
$\widehat{\bm Y}_{\mathcal I}^\star$.
\State Extract $\widehat{\bm Y}_{\mathrm{for}}$ from the reconstructed
columns indexed by $\mathcal T_{\mathrm{for}}$.
\State Apply Algorithm~\ref{alg:phase3_qoi_identification} to the paired
experimental quantity--channel samples on $\mathcal T_{\mathrm{obs}}$,
obtaining $(\bm\beta_Q^\star,\bm\theta_Q^\star)$.
\State Apply Algorithm~\ref{alg:phase3_qoi_forecast} to the terminal
observed quantity history and $\widehat{\bm Y}_{\mathrm{for}}$.
\State \Return the selected models, reconstructed channels,
and finite-horizon quantity forecast.
\end{algorithmic}
\end{breakablealgorithm}

\subsection{Computational Complexity}
\label{subsec:framework_complexity}

Table~\ref{tab:framework_complexity} collects the principal computational
costs of the three-phase workflow. Let $s=\min\{q,K-1\}$, $A_{\mathcal I}=|\mathfrak A_{\mathcal I}|$,
and $A_Q=|\mathfrak B_Q|$. As above,
$M_{\bm\alpha}=K_{\mathcal I}-\ell_{\bm\alpha}$ and
$M_{\bm\beta}=N_Q-\ell_{\bm\beta}$ are the numbers of usable transitions.
Each scalar feature is assumed to have constant evaluation cost.
The Phase~III identification cost excludes solver-dependent optional
refinement, and its forecast cost assumes that Phase~II has already
provided the reconstructed channel drivers.

\begin{table}[!htbp]
\centering
\caption{Principal computational costs of the three-phase framework.}
\label{tab:framework_complexity}
\small
\renewcommand{\arraystretch}{1.25}
\begin{tabular}{@{}
    >{\raggedright\arraybackslash}p{0.3\linewidth}
    >{\raggedright\arraybackslash}p{\dimexpr0.7\linewidth-2\tabcolsep\relax}
    @{}}
\toprule
\textbf{Stage} & \textbf{Computational complexity}\\
\midrule
Phase~I HKFED &
$\mathcal O\!\left(q(K-1)s+p^3+pqK+pL\right)$\\[4pt]

Phase~II search &
$\mathcal O\!\left(
\sum_{\bm\alpha\in\mathfrak A_{\mathcal I}}\mathcal C_{\bm\alpha}
+A_{\mathcal I}mN_{\mathcal I}+A_{\mathcal I}^2
\right)$\\[4pt]

Phase~III identification &
$\mathcal O\!\left(
\sum_{\bm\beta\in\mathfrak B_Q}
\left[M_{\bm\beta}p_{\bm\beta}^2+p_{\bm\beta}^3+N_Qp_{\bm\beta}\right]
+A_Q^2
\right)$\\[4pt]

Phase~III forecast &
$\mathcal O\!\left(N_{\mathrm f}(m+p_{\bm\beta_Q^\star})\right)$\\
\bottomrule
\end{tabular}
\end{table}

\section{Numerical Experiments}\label{Num_exp}

The numerical experiments assess the three-phase framework on a finite
multichannel record. Phase~I constructs the modal basis from the complete
record. Phase~II identifies and assesses a nonlinear channel model on
$\mathcal T_H$, then recursively simulates the retained Hankel
coefficients. Physical-space reconstruction converts these simulations
into channel trajectories that drive Phase~III. The quantity-of-interest
model is identified separately on $\mathcal T_{\mathrm{obs}}$ and advanced
recursively over $\mathcal T_{\mathrm{for}}\subseteq\mathcal T_H$.

This temporal arrangement permits assessment of the channel model and
the quantity-of-interest model together over a common finite interval.
The reported accuracy characterizes their composition within the
prescribed calibration setting: the inputs supplied to Phase~III are
model-generated trajectories whose underlying channel model was
calibrated on $\mathcal T_H$. The experiments therefore evaluate
finite-horizon prediction conditional on that calibration arrangement;
they do not constitute an independent test of the complete framework
on an unseen future interval.

\subsection{Experimental Data and Computational Setting}
\label{subsec:numerical_experiment_setting}

The case study uses weather-dependent channels from Deva--Hunedoara
region of Romania to evaluate the complete framework.
The data source, location, acquisition period, and channel units are
reported in Table~\ref{tab:numerical_experiment_setting}.
After removing invalid observations and duplicate time stamps, samples
are sorted chronologically and filtered to retain shortwave radiation
above $5\ \mathrm{W/m^2}$.

The snapshot matrix $\bm Y\in\mathbb R^{4\times461}$ has the form in
\eqref{eq:experimental_snapshot_matrix}. Its rows contain cloud cover,
ambient temperature, wind speed, and relative humidity, respectively.
This solar-power setting supports the assessment of weather-dependent
modeling and forecasting \cite{LafuenteCacho2025,Cargan2024,
Salman2024,CevikBektas2024}.

\begin{table}[!htbp]
\caption{Data source, variables, and temporal partition used in the numerical
experiment.}
\label{tab:numerical_experiment_setting}%
\begin{tabular*}{\textwidth}{@{\extracolsep\fill}lll}
\toprule
Item & Symbol or channel & Setting \\
\midrule
Geographical location & --
    & $(45.88^{\circ}\mathrm{N},22.88^{\circ}\mathrm{E})$ \\
Data source & -- & Historical Weather API \cite{OpenMeteoHistoricalWeather}\\
Downloaded interval & -- & 1 January--31 March 2026 \\
Temporal resolution & -- & $1$ h \\
Number of measured channels & $m$ & $4$ \\
Modeling samples & $N$ & $461$ \\
Finite modeling horizon & $\mathcal T_{\mathrm{fin}}$ & $[1,\;461]$ \\
Row 1 of $\bm Y$ & -- & cloud cover [\%] \\
Row 2 of $\bm Y$ & -- & ambient temperature at $2$ m [$^{\circ}$C] \\
Row 3 of $\bm Y$ & -- & wind speed at $10$ m [m/s] \\
Row 4 of $\bm Y$ & -- & relative humidity at $2$ m [\%] \\
Daylight/response variable & -- & shortwave radiation [W/m$^2$] \\
Nominal PV parameters & --
    & $(10\ \mathrm{kW},-0.0035\ ^{\circ}\mathrm{C}^{-1},46^{\circ}\mathrm{C})$ \\
\bottomrule
\end{tabular*}
\end{table}

\subsection{Phase I Results: Hankel--Koopman Finite-Horizon Energy Decomposition}
\label{subsec:numerical_phase1}

Algorithm~\ref{alg:hkfed} is applied directly to the four-channel measurement record
using delay depth $q=200$ and operator horizon $L=100$.
Table~\ref{tab:phase1_parameters} gives the resulting dimensions and
selection parameters.

\begin{table}[!htbp]
\caption{Phase~I HKFED parameters and data-dependent matrix dimensions.}
\label{tab:phase1_parameters}%
\begin{tabular*}{\textwidth}{@{\extracolsep\fill}lll}
\toprule
Quantity & Symbol & Value \\
\midrule
Serialized measurement-record length & $N_H=mN$ & 1844\\
Imposed Hankel delay depth (rows) & $q$ & 200 \\
Hankel columns & $K=N_H-q+1$ & 1645 \\
Hankel-matrix size & $q\times K$ & $200\times 1645$ \\
Operator horizon & $L$ & $100$ \\
Dimension-preference weight & $\alpha_{\mathrm{dim}}$ & 0.03 \\
Weak dimension penalty & $\beta_{\mathrm{dim}}$ & 0.02 \\
Admissible dimension interval & $[r_{\min},r_{\max}]$ & $[100,180]$ \\
Preferred dimension & $r_{\mathrm{tar}}$ & 150 \\
\bottomrule
\end{tabular*}
\end{table}

\subsubsection{Energy Spectrum and Finite-Horizon Triplet Ranking}
\label{subsubsec:numerical_phase1_energy}

Figure~\ref{fig:phase1_selected_eigenvalues} shows the energy-operator
spectrum and the eigenvalues belonging to the retained triplets.
These are energy eigenvalues; their relation to the Koopman eigenvalues
under normality is given by \eqref{eq:normal_koopman_energy_spectral_link}.

\begin{figure}[!h]
\centering
 \includegraphics[width=0.90\textwidth]{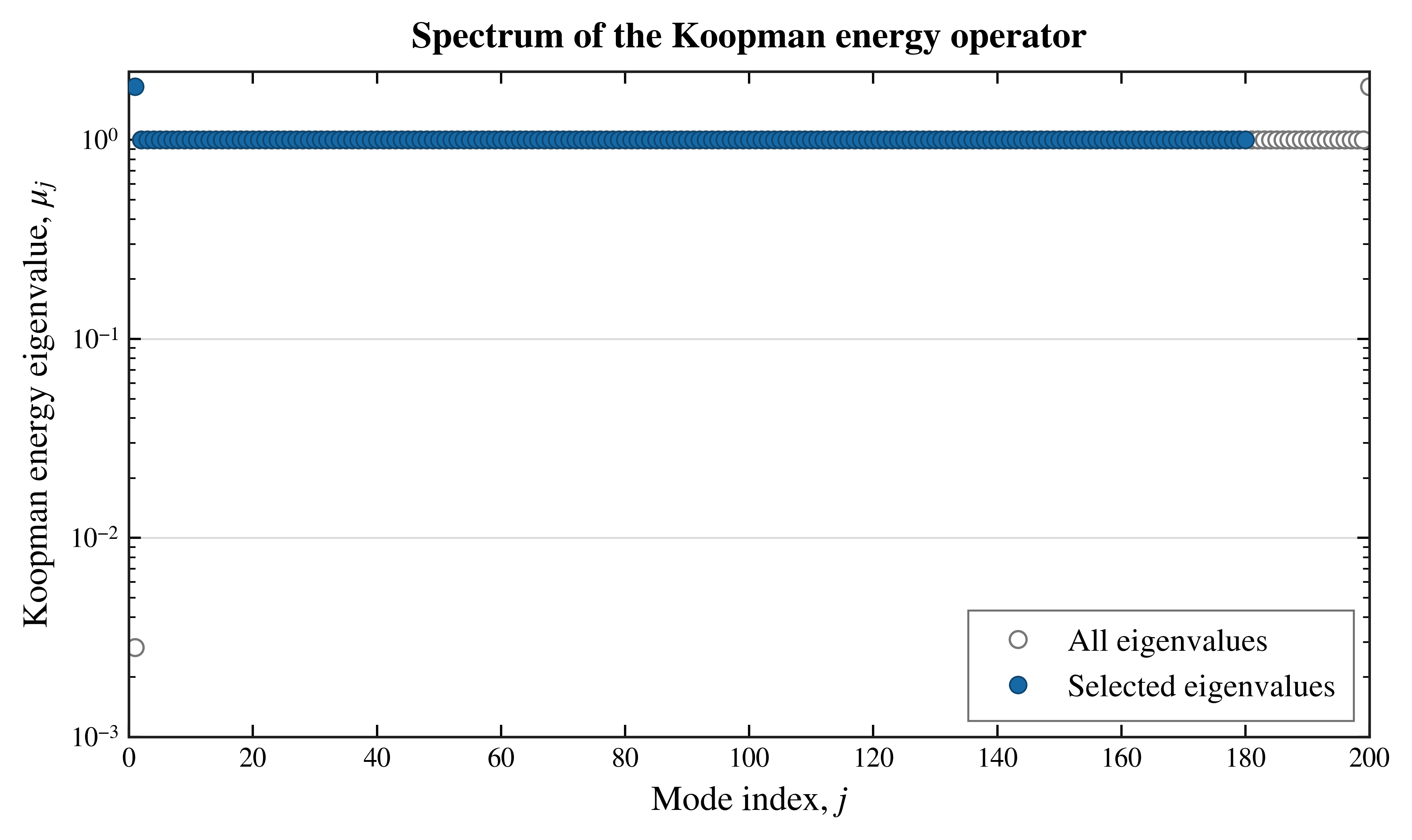}
\caption{Spectrum of the Koopman energy operator $\bm G_p$, shown in decreasing
eigenvalue order. Directions belonging to the selected triplet family
$\mathfrak T_{r^\star}$ are highlighted; unretained directions are shown
for comparison.}
\label{fig:phase1_selected_eigenvalues}
\end{figure}

Figure~\ref{fig:phase1_modal_energy} presents the ranked modal energies
and persistence factors for $L=100$. Persistence near $101$ corresponds
to near-unit energy eigenvalues, while departures indicate attenuation
or amplification over the prescribed operator horizon.

\begin{figure}[!h]
\centering
 \includegraphics[width=0.90\textwidth]{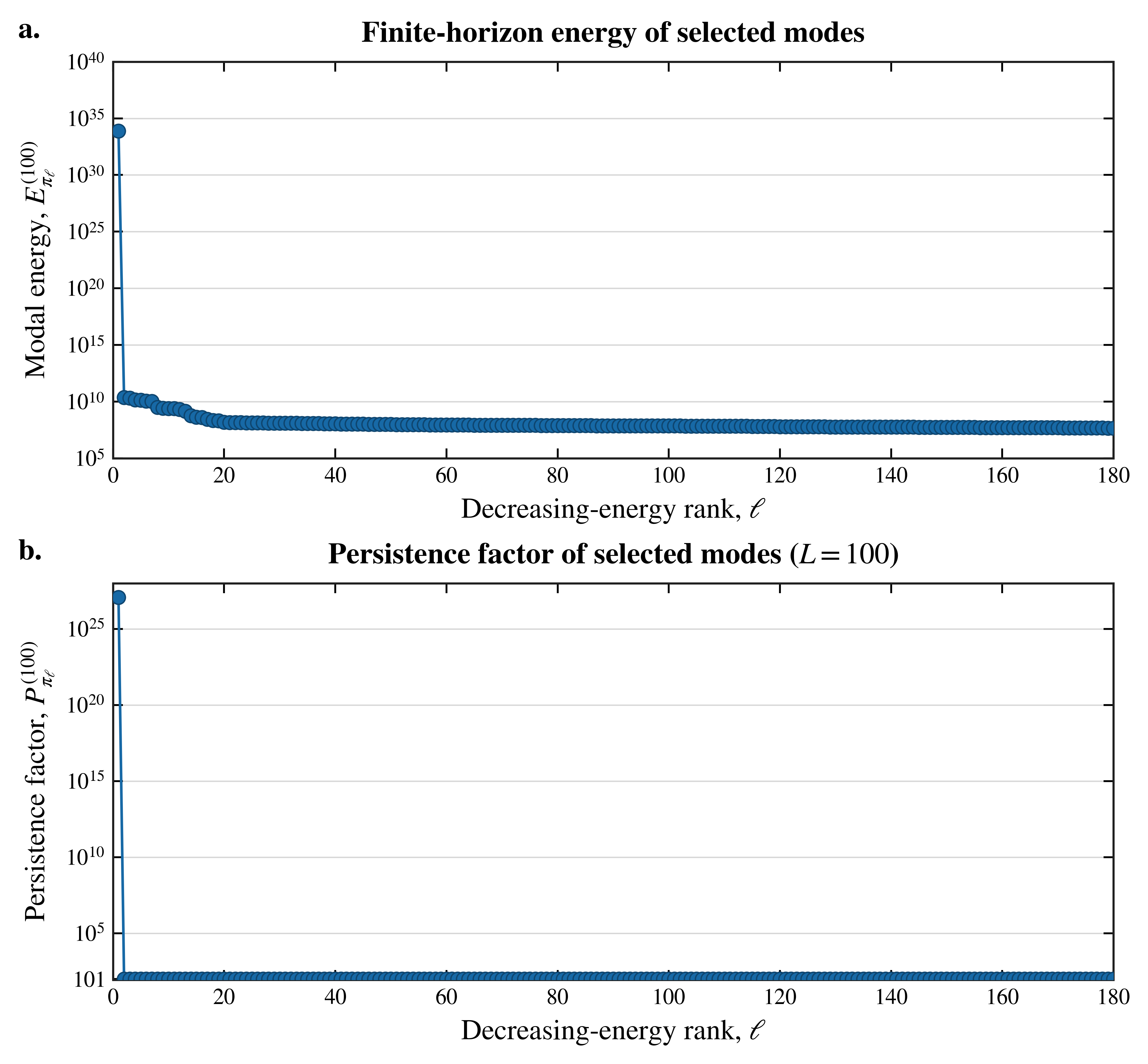}
\caption{a. Finite-horizon modal energies $E_{\pi_\ell}^{(100)}$ in decreasing-energy
rank $\ell$.\\
b. Persistence factors $P_{\pi_\ell}^{(100)}$ for the selected modes.
Values near $101$ indicate near-unit energy eigenvalues; smaller or larger
values indicate attenuation or amplification over the operator horizon.}
\label{fig:phase1_modal_energy}
\end{figure}

\subsubsection{Pareto Selection of the Retained Koopman-Energy Candidate}
\label{subsubsec:numerical_phase1_pareto}

The $p=200$ candidates were evaluated using the error and similarity
criteria in \eqref{eq:hkfed_reconstruction_error} and
\eqref{eq:hkfed_cosine_similarity}.
Figure~\ref{fig:phase1_pareto_selection} shows their accuracy trade-off
and its dependence on retained dimension.

The nondominated set is $\mathcal P=\{200\}$, outside the admissible
interval $\mathcal R_{\mathrm{adm}}=\{100,\ldots,180\}$.
Consequently, \eqref{eq:hkfed_pruned_pareto_set} uses the fallback
$\mathcal C=\mathcal R_{\mathrm{adm}}$.
Minimizing \eqref{eq:hkfed_final_selection_score}, with the soft target
$r_{\mathrm{tar}}=150$, selects $r^\star=180$.
The resulting triplet family retains $90\%$ of the data-supported rank
and supplies the modes and coefficients for the subsequent phases.

\begin{figure}[!htbp]
\centering
\includegraphics[width=0.90\textwidth]{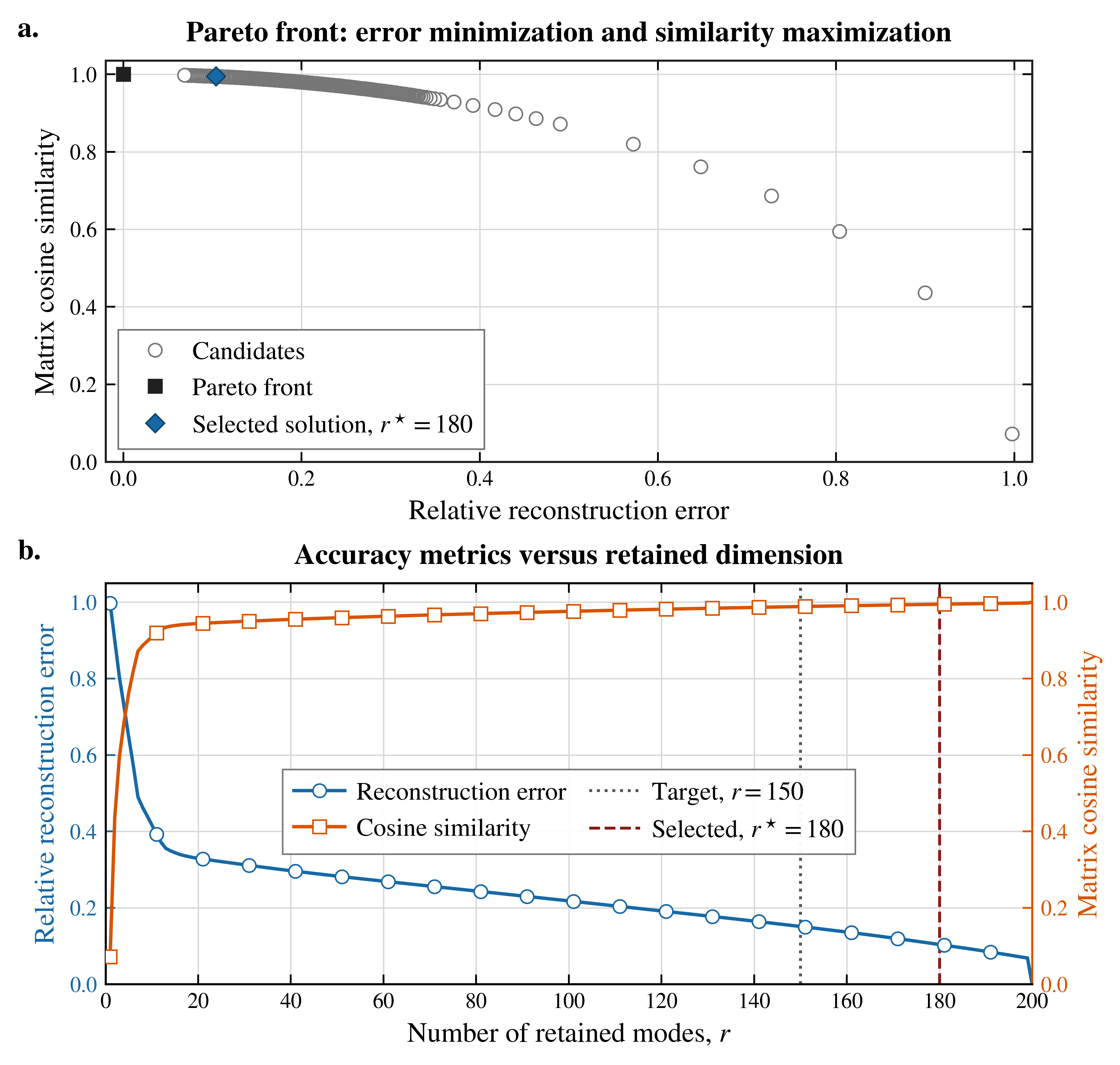}
\caption{Phase~I candidate selection.\\
a. Accuracy plane $(\varepsilon_r,\gamma_r)$: open circles denote candidates,
the black square the Pareto optimum $r=200$, and the blue diamond the selected
$r^\star=180$. Accuracy improves toward the upper left.\\
b. Reconstruction error and cosine similarity versus retained dimension.
The dotted line marks $r_{\mathrm{tar}}=150$; the dashed line marks
$r^\star=180$, selected by minimizing $J$ on
$\mathcal C=\{100,\ldots,180\}$.}
\label{fig:phase1_pareto_selection}
\end{figure}

Table~\ref{tab:phase1_pareto_selection} summarizes the selected triplet
family $\mathfrak T_{r^\star}$ and its diagnostics.

\begin{table}[!htbp]
\caption{Summary of the Phase~I finite-horizon Pareto selection.}
\label{tab:phase1_pareto_selection}%
\begin{tabular*}{\textwidth}{@{\extracolsep\fill}lll}
\toprule
Reported quantity & Symbol & Numerical value \\
\midrule
Number of Pareto-optimal candidates & $|\mathcal P|$ & 1 \\
Selected reduced dimension & $r^\star$ & 180\\
Retained-rank fraction & $r^\star/p$ & 0.90 \\
Relative Hankel reconstruction error & $\varepsilon_{r^\star}$ & 0.0696 \\
Matrix cosine similarity & $\gamma_{r^\star}$ & 0.9976 \\
Orthogonality residual (Frobenius norm) & --
& $3.1037\times10^{-14}$ \\
Maximum off-diagonal inner product & -- & $9.2981\times10^{-16}$ \\
Phase~I execution time & -- & 1.3761 s \\
\bottomrule
\end{tabular*}
\end{table}

For the Hankel factorization,
$\mathrm{CR}_r=qK/[r(q+K)]$ compares storage of $\bm H$ with that of
$\bm\Phi_r$ and $\bm C_r$.
The reported ratio $\mathrm{CR}_{r^\star}=0.9907$ is below one, so
these factors provide no storage saving even relative to the lifted matrix.
The reduced-order character is the spectral truncation
$r^\star/p=0.90$, which limits the lifted coefficient state evolved in
Phase~II. Its role here is to organize the coupled dynamics in selected
energy coordinates.

\subsubsection{Orthogonality and Reconstruction of the Experimental Channels}
\label{subsubsec:numerical_phase1_reconstruction}

The selected modal matrix
$\bm\Phi_{r^\star}=\bm\Phi_{180}$ was first assessed through its Gram matrix
\begin{equation}
    \bm G_{\Phi,r^\star}
    =
    \bm\Phi_{r^\star}^{\mathrm H}
    \bm\Phi_{r^\star}.
    \label{eq:numerical_phase1_gram_matrix}
\end{equation}
Proposition~\ref{prop:hkfed_mode_orthogonality} gives
$\bm G_{\Phi,r^\star}=\bm I_{r^\star}$ for exact orthonormality.
Figure~\ref{fig:phase1_mode_orthogonality} displays its absolute entries;
the $180$ unit-height diagonal blocks are the modal self-inner products.

\begin{figure}[!htbp]
\centering
\includegraphics[width=0.6\textwidth]{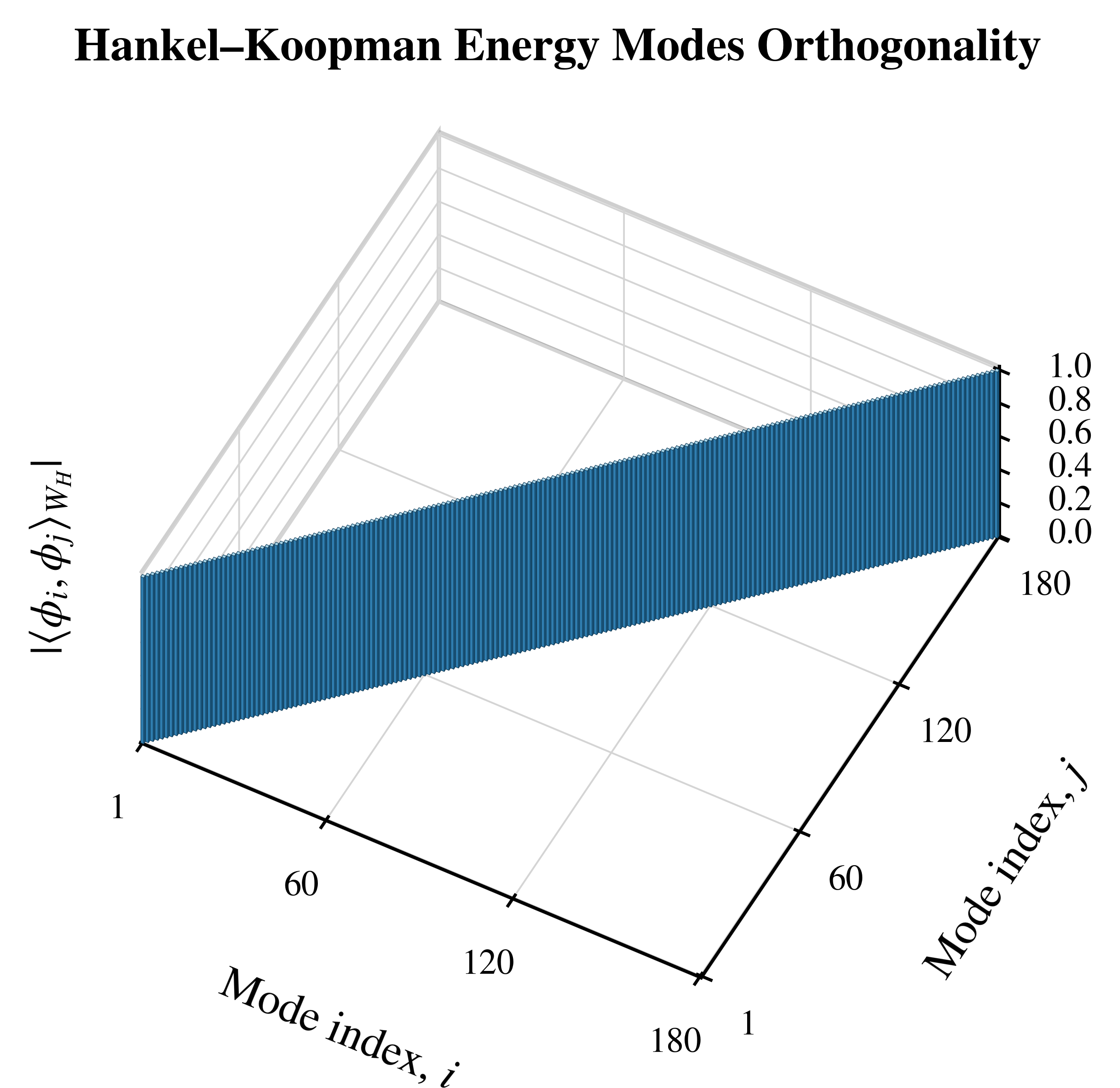}
\caption{Absolute Gram matrix
$|\bm\Phi_{r^\star}^{\mathrm H}\bm\Phi_{r^\star}|$ for the $180$ retained
modes. Unit diagonal entries and off-diagonal values at machine precision
numerically verify Proposition~\ref{prop:hkfed_mode_orthogonality}.}
\label{fig:phase1_mode_orthogonality}
\end{figure}

Table~\ref{tab:phase1_pareto_selection} confirms orthonormality to machine
precision, supporting the projection used to recover the coefficients.

The selected Hankel reconstruction was returned to the four channels by
anti-diagonal recovery \eqref{eq:uniform_antidiagonal_recovery} and inverse
reshaping \eqref{eq:recovered_experimental_channels}.
Channel $i$ was assessed by relative error, root-mean-square error,
mean absolute error, and Pearson correlation
\eqref{eq:finite_pearson_correlation}:

\begin{equation}
\begin{alignedat}{2}
e_{i,\mathrm{rel}}
&=
\frac{
    \left\|\bm Y(i,:)-\widehat{\bm Y}(i,:)\right\|_2
}{
    \left\|\bm Y(i,:)\right\|_2
},
&\qquad
\operatorname{RMSE}_i
&=
\frac{1}{\sqrt{N}}
\left\|\bm Y(i,:)-\widehat{\bm Y}(i,:)\right\|_2,
\\[1.5ex]
\operatorname{MAE}_i
&=
\frac{1}{N}
\left\|\bm Y(i,:)-\widehat{\bm Y}(i,:)\right\|_1,
&\qquad
R_i
&=
\varrho_N\!\left(
    \bm Y(i,:)^{\mathsf T},
    \widehat{\bm Y}(i,:)^{\mathsf T}
\right).
\end{alignedat}
\label{eq:numerical_phase1_channel_metrics}
\end{equation}
Table~\ref{tab:phase1_channel_accuracy} reports the metrics.
Relative error and correlation are dimensionless; RMSE and MAE retain
the channel units.

\begin{table}[!htbp]
\caption{Channelwise Phase~I HKFED reconstruction metrics over the
$N=461$ observation samples after uniform anti-diagonal recovery.}
\label{tab:phase1_channel_accuracy}%
\small
\begin{tabular*}{\textwidth}{@{\extracolsep\fill}lccccc}
\toprule
Channel & Unit & $e_{i,\mathrm{rel}}$ & $\operatorname{RMSE}_i$
    & $\operatorname{MAE}_i$ & $R_i$ \\
\midrule
Cloud cover & \% & 0.0648 & 5.7629
    & 2.6034 & 0.9821 \\
Ambient temperature at $2$ m & $^{\circ}$C & 0.1281
    & 0.7416 & 0.4350 & 0.9918 \\
Wind speed at $10$ m & m/s & 0.5975 & 1.2885
    & 0.4679 & 0.5777 \\
Relative humidity at $2$ m & \% & 0.0729 & 6.0580
    & 4.0726 & 0.9034 \\
\bottomrule
\end{tabular*}
\end{table}

\begin{figure}[!htbp]
\centering
\includegraphics[width=0.85\textwidth]{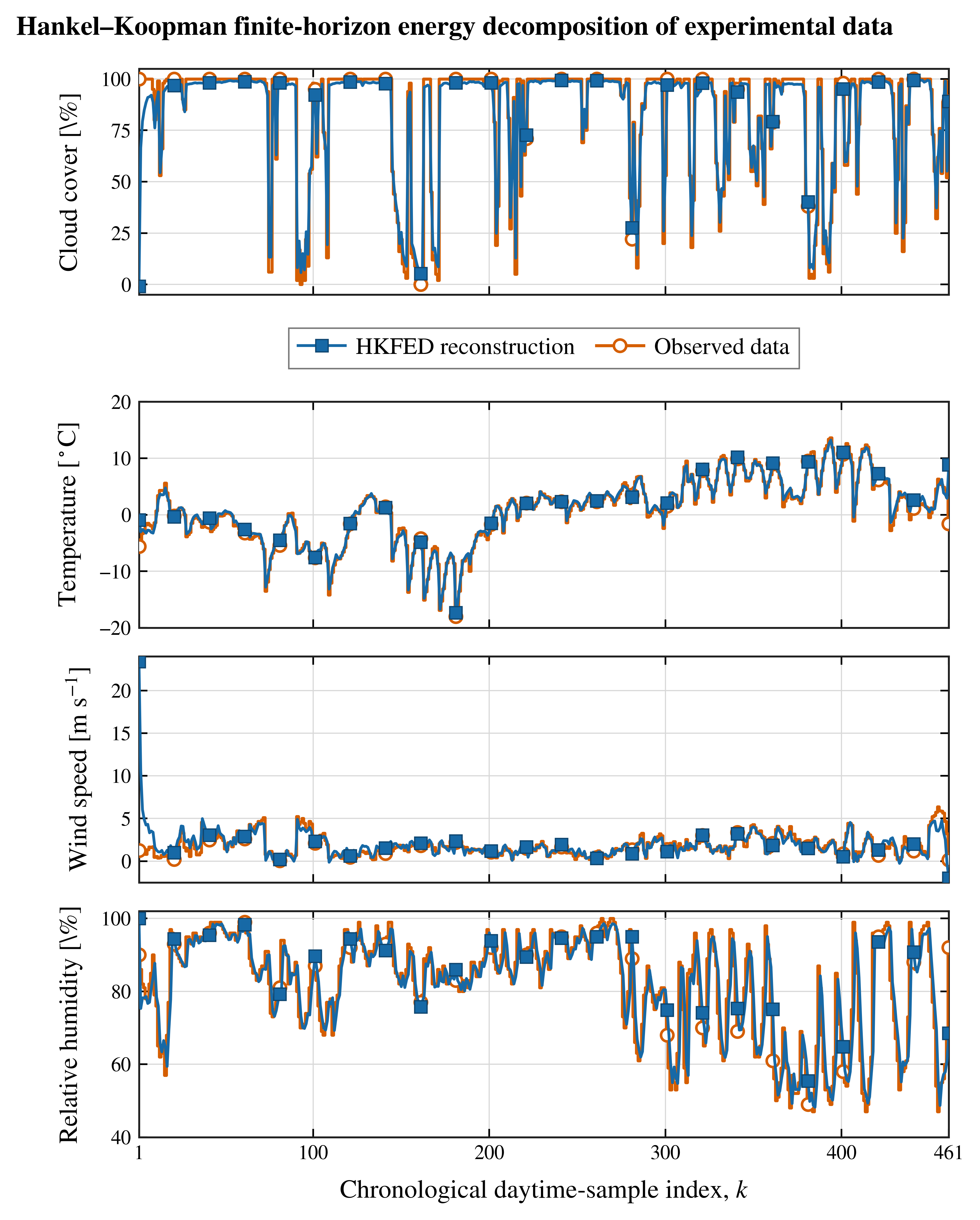}
\caption{Observed channels (open circles) and HKFED reconstructions from
$\mathfrak T_{180}$ after anti-diagonal recovery (blue squares).
All axes use the chronological daytime-sample index $k=1,\ldots,461$.}
\label{fig:phase1_channel_reconstruction}
\end{figure}

Cloud cover and temperature retain strong temporal agreement, while
wind speed has the largest relative error and lowest correlation
(Figure~\ref{fig:phase1_channel_reconstruction} and
Table~\ref{tab:phase1_channel_accuracy}).
Relative Hankel error aggregates the lifted entries, whereas each
channelwise relative error is evaluated after recovery using its own
reference norm. Small aggregate error therefore does not imply comparable
accuracy in every channel, making physical-space assessment essential
for subsequent quantity forecasting.

\subsection{Phase II: Inverse-Calibrated Multi-Output NLARX Results}
\label{subsec:numerical_phase2}

Algorithm~\ref{alg:nlarx_pareto_inverse_calibration} was applied to the
coefficients of the $r^\star=180$ retained modes on four calibration windows
within $\mathcal T_{\mathrm{fin}}=[1,461]$. Candidate trajectories were
reconstructed using $\bm\Phi_{180}$ and anti-diagonal recovery, then assessed
by the Pareto rule in \eqref{eq:nlarx_pareto_structure_selection}.
Table~\ref{tab:phase2_parameters} lists the numerical settings.

\begin{table}[!htbp]
\caption{Phase~II inverse-calibrated multi-output NLARX parameters and
finite calibration windows.}
\label{tab:phase2_parameters}%
\begin{tabular*}{\textwidth}{@{\extracolsep\fill}lll}
\toprule
Quantity & Symbol & Value \\
\midrule
Selected Phase~I dimension & $r^\star$ & 180 \\
Number of experimental channels & $m$ & 4 \\
First calibration window & $\mathcal T_H^1$ & $[288,\;388]$\\
Second calibration window & $\mathcal T_H^2$ & $[360,\;460]$\\
Third calibration window & $\mathcal T_H^3$ & $[310,\;460]$\\
Fourth calibration window & $\mathcal T_H^4$ & $[300,\;450]$\\
Tested NLARX order triple & $\bm{\alpha}=(n_a,n_b,n_k)$ 
    & $(1:8,1:4,1:3)$ \\
Candidate-fitting regularization parameter & $\lambda_{\mathrm{reg}}$
    & $ 10^{-4}$ \\
Pareto normalized decision weights & $(w_1,w_2,w_3,w_4)$
    & $(0.3,0.5,0.1,0.1)$ \\
\bottomrule
\end{tabular*}
\end{table}

\subsubsection{Pareto-Selected NLARX Structure}
\label{subsubsec:numerical_phase2_selection}

For all four finite calibration windows $\mathcal T_H^i$, $i=1,\ldots,4$,
the Pareto-based procedure selected the admissible order triple
\begin{equation}
    \bm\alpha_{\mathrm{Pareto}}^\star
    =\bm\alpha^\star=(8,1,1).
    \label{eq:numerical_phase2_selected_structure}
\end{equation}
Table~\ref{tab:phase2_selection_results} reports the selected-model
diagnostics for $\mathcal T_H^1$.

\begin{table}[!htbp]
\caption{Phase~II Pareto-based NLARX model-selection results on the first
calibration window $\mathcal T_H^1$.}
\label{tab:phase2_selection_results}
\small
\begin{tabular*}{\textwidth}{@{\extracolsep\fill}lcc}
\toprule
Reported quantity & Symbol & Value \\
\midrule
Selected order triple & $\bm\alpha^\star$ & $(8,1,1)$ \\
Required initial history & $\ell_{\bm\alpha^\star}$ & 8 \\
Number of nonlinear features & $p_{\bm\alpha^\star}$ & 2881 \\
Parameter-matrix norm & $\|\bm B^\star\|_{\mathrm F}$ & 5.1481 \\
Relative reconstruction error & $f_1$ & 0.2708 \\
Mean correlation loss & $f_2$ & 0.0588 \\
Residual autocorrelation & $f_3$ & 0.5702 \\
Parameter penalty & $f_4$ & 0.8373 \\
Pareto decision score & $\Psi_{\mathcal I}$ & 0.1395 \\
Number of Pareto-optimal candidates & $|\mathfrak P_{\mathcal I}|$ & 83 \\
Phase~II execution time & -- & 18.4773 s \\
\bottomrule
\end{tabular*}
\end{table}

\subsubsection{Pareto Front in the Objective Space}
\label{subsubsec:numerical_phase2_pareto}

For visualization, the four objectives in
\eqref{eq:nlarx_pareto_objectives} are paired as
\begin{equation}
P_{12}(\bm\alpha)
=
\left(
f_1(\bm\alpha)^2+f_2(\bm\alpha)^2
\right)^{1/2},
\qquad
P_{34}(\bm\alpha)
=
\left(
f_3(\bm\alpha)^2+f_4(\bm\alpha)^2
\right)^{1/2}.
\label{eq:numerical_phase2_paired_objectives}
\end{equation}
These coordinates are used only for display; dominance and model
selection remain defined in the original four-objective space.

\begin{figure}[!htbp]
\centering
\includegraphics[width=0.9\textwidth]
{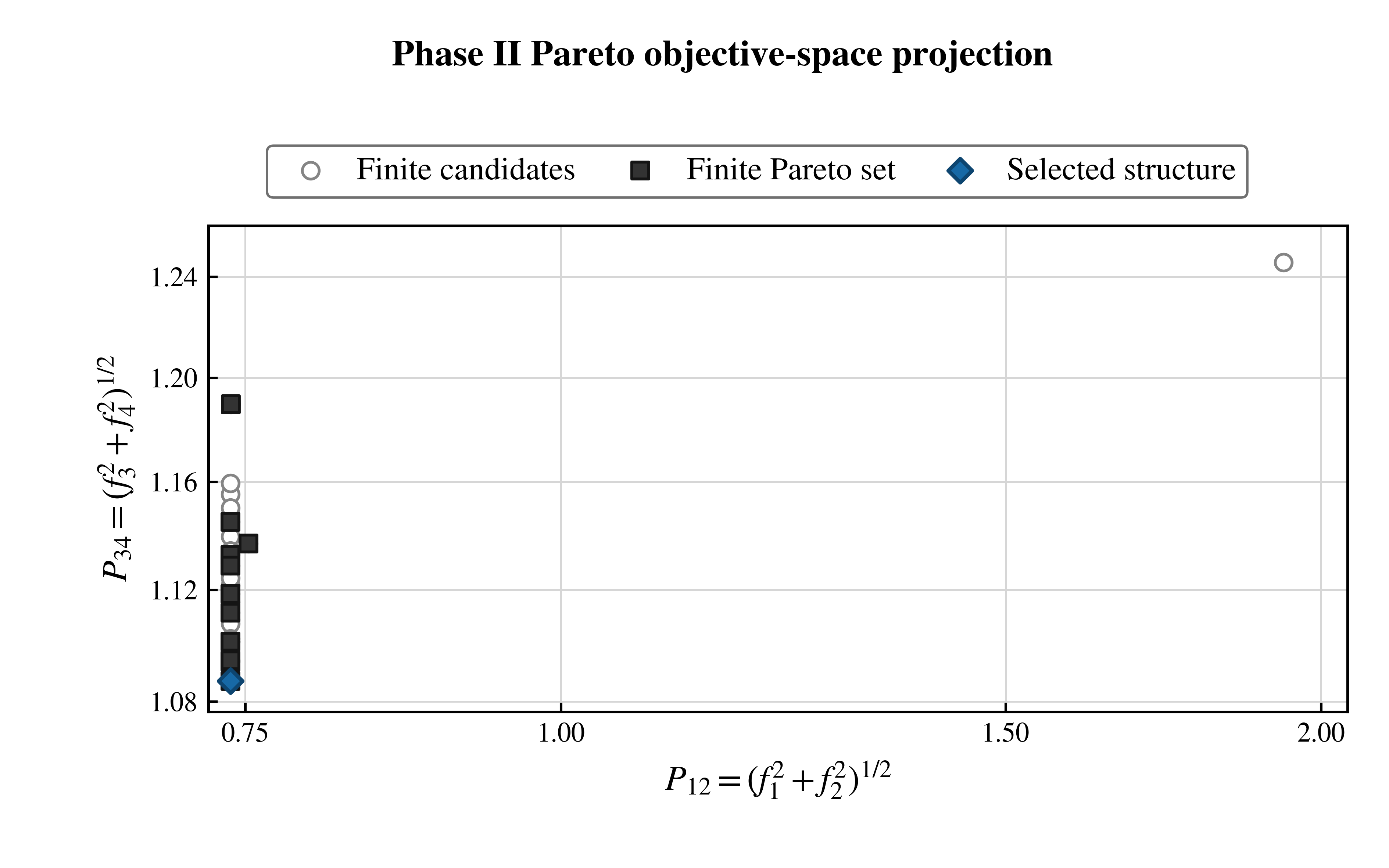}
\caption{Phase~II candidates in the display coordinates $P_{12}$ and $P_{34}$,
defined in \eqref{eq:numerical_phase2_paired_objectives}.
Open circles denote candidates with finite objectives, black squares the
four-objective Pareto subset, and the blue diamond the model selected
by $\Psi_{\mathcal I}$ and \eqref{eq:nlarx_pareto_structure_selection}.}
\label{fig:phase2_pareto_objective_space}
\end{figure}

Figure~\ref{fig:phase2_pareto_objective_space} shows the selected
structure within a broad nondominated family. The decision weights
resolve this remaining trade-off after Pareto filtering.

\subsubsection{Measurement-Space Reconstruction of Channels}
\label{subsubsec:numerical_phase2_reco}

Channelwise relative error and Pearson correlation are evaluated as in
\eqref{eq:numerical_phase1_channel_metrics}, restricted to each complete
calibration window and including the prescribed initial-history samples.
Table~\ref{tab:phase2_channel_accuracy} reports all four experiments.

\begin{table}[!htbp]
\caption{Channelwise Phase~II inverse-calibrated reconstruction
metrics over the four calibration windows $\mathcal{T}_H^i$, $i=1,\ldots,4$.}
\label{tab:phase2_channel_accuracy}%
\small
\begin{tabular*}{\textwidth}{@{\extracolsep\fill}llcccc}
\toprule
& & \multicolumn{4}{c}{Pareto-based} \\
\cmidrule(lr){3-6}
Channel & Metric
    & $\mathcal{T}_H^1$ & $\mathcal{T}_H^2$
    & $\mathcal{T}_H^3$ & $\mathcal{T}_H^4$ \\
\midrule
\multirow{2}{*}{Cloud cover}
    & $e_{i,\mathrm{rel}}$
    & $0.0285$
    & $0.0293$
    & $0.0286$
    & $0.0280$ \\
    & $R_i$
    & 0.9998 & 0.9995 & 0.9996 & 0.9996 \\
\addlinespace

\multirow{2}{*}{\shortstack[l]{Ambient temperature\\at $2$ m}}
    & $e_{i,\mathrm{rel}}$
    & $0.0682$
    & $0.1016$
    & $0.0872$
    & $0.0818$ \\
    & $R_i$
    & 0.9944 & 0.9896 & 0.9901 & 0.9935 \\
\addlinespace

\multirow{2}{*}{\shortstack[l]{Wind speed\\at $10$ m}}
    & $e_{i,\mathrm{rel}}$
    & $0.1489$
    & $0.2862$
    & $0.25102$
    & $0.1799$ \\
    & $R_i$
    & 0.9353 & 0.8738 & 0.8826 & 0.9234 \\
\addlinespace

\multirow{2}{*}{\shortstack[l]{Relative humidity\\at $2$ m}}
    & $e_{i,\mathrm{rel}}$
    & $0.1091$
    & $0.1158$
    & $0.1127$
    & $0.1147$ \\
    & $R_i$
    & 0.8350 & 0.8798 & 0.8633 & 0.8566 \\
\bottomrule
\end{tabular*}
\end{table}

Cloud cover and temperature show the strongest agreement, while
wind speed and humidity exhibit larger discrepancies.
Figures~\ref{fig:phase2_nlarx_simulation_residuals_TH1}--\ref{fig:phase2_nlarx_simulation_residuals_TH4}
show their temporal distribution on the four windows.
These are reconstruction metrics on the calibration windows, including
the prescribed initial history. They characterize the channel inputs
supplied to Phase~III within the stated retrospective setting.

\begin{figure}[!htbp]
    \centering
    \includegraphics[
        width=\textwidth,
        height=0.90\textheight,
        keepaspectratio
    ]{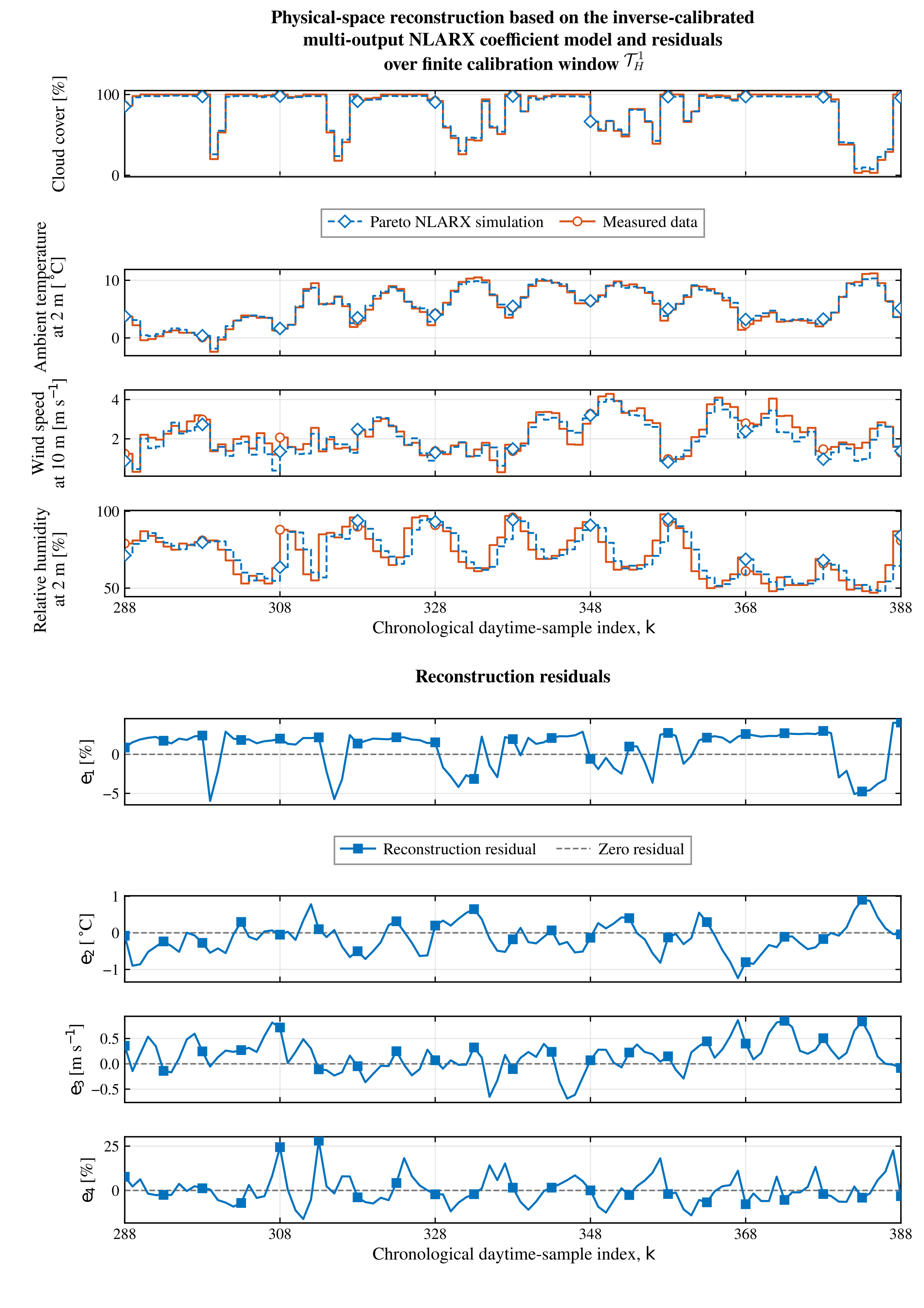}
    \caption{Physical-channel reconstruction and residuals from the inverse-calibrated
multi-output NLARX coefficient model on
$\mathcal T_H^1=[288,388]$.}
    \label{fig:phase2_nlarx_simulation_residuals_TH1}
\end{figure}
\begin{figure}[!htbp]
    \centering
    \includegraphics[
        width=\textwidth,
        height=0.90\textheight,
        keepaspectratio
    ]{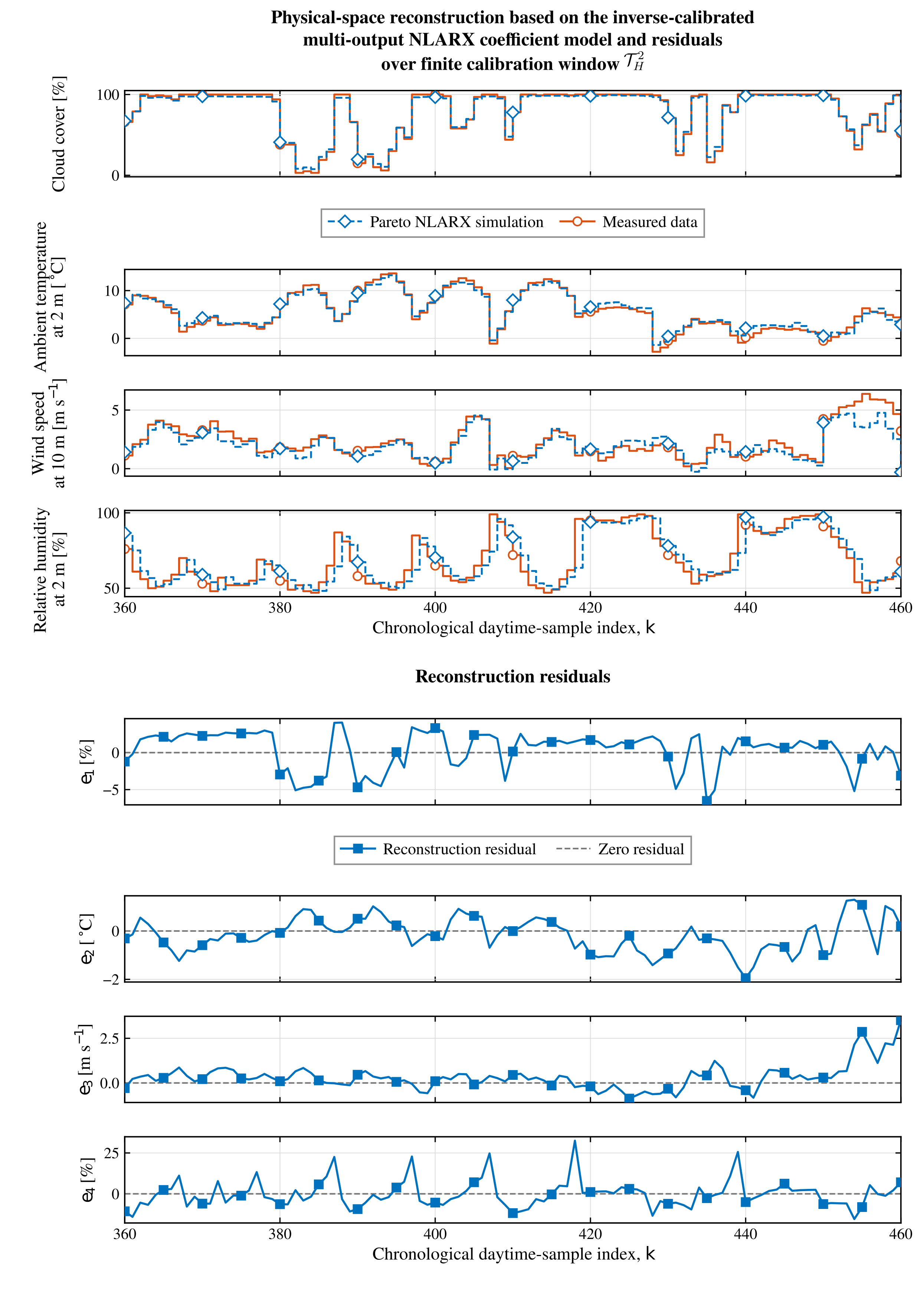}
    \caption{Physical-channel reconstruction and residuals from the inverse-calibrated
multi-output NLARX coefficient model on
$\mathcal T_H^2=[360,460]$.}
    \label{fig:phase2_nlarx_simulation_residuals_TH2}
\end{figure}
\begin{figure}[!htbp]
    \centering
    \includegraphics[
        width=\textwidth,
        height=0.90\textheight,
        keepaspectratio
    ]{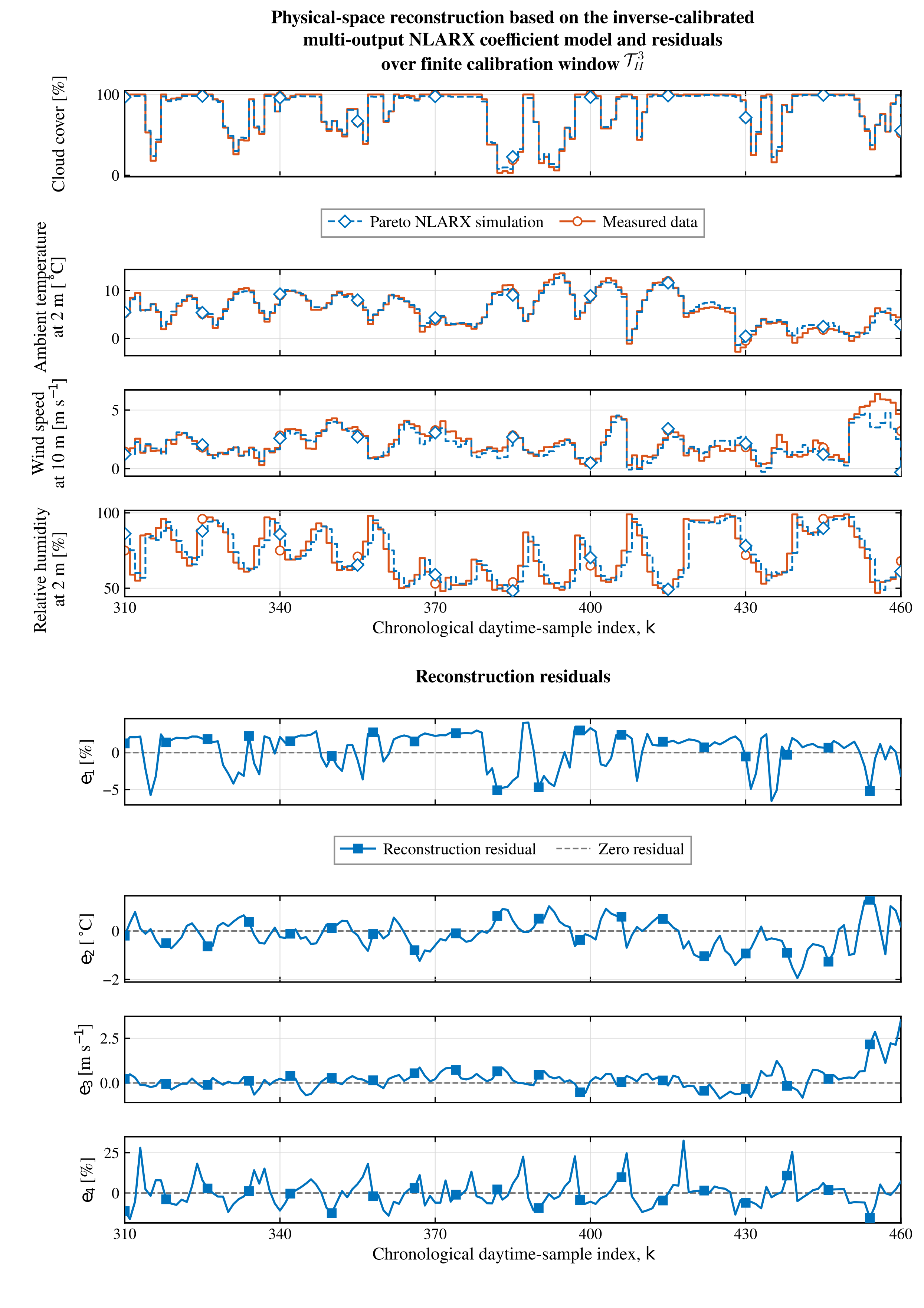}
    \caption{Physical-channel reconstruction and residuals from the inverse-calibrated
multi-output NLARX coefficient model on
$\mathcal T_H^3=[310,460]$.}
    \label{fig:phase2_nlarx_simulation_residuals_TH3}
\end{figure}
\begin{figure}[!htbp]
    \centering
    \includegraphics[
        width=\textwidth,
        height=0.90\textheight,
        keepaspectratio
    ]{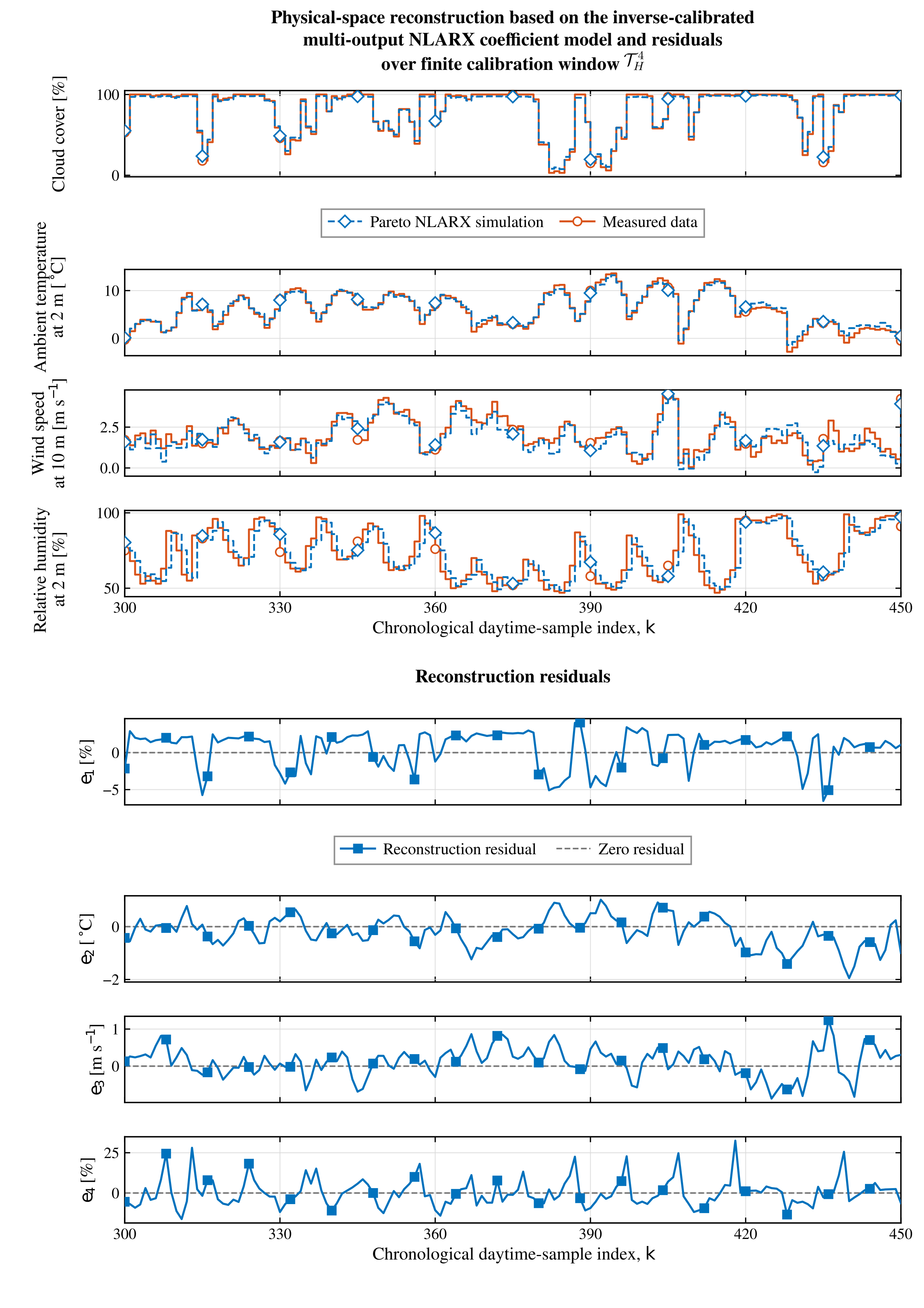}
    \caption{Physical-channel reconstruction and residuals from the inverse-calibrated
multi-output NLARX coefficient model on
$\mathcal T_H^4=[300,450]$.}
    \label{fig:phase2_nlarx_simulation_residuals_TH4}
\end{figure}

\subsection{Phase III: Finite-Horizon Recursive Forecasting Results}
\label{subsec:numerical_phase3}

Four experiments evaluate the recursive quantity-of-interest forecast using
the reconstructed Phase~II channels as exogenous drivers. Accuracy is reported
through Pearson correlation, relative error, and the identification-to-forecast
ratio in \eqref{eq:phase3_calibration_forecast_ratio}.

\subsubsection{Solar-Plant Power as the Quantity-of-Interest}
\label{subsubsec:numerical_phase3_qoi}

Here, $Q_{\mathrm I}$ in Definition~\ref{def:phase3_quantity_of_interest}
is normalized solar-plant output power. Using the channel order in
\eqref{eq:experimental_snapshot_matrix}, its value at the $k$th hourly sample is
\begin{equation}
\begin{aligned}
Q_{\mathrm I}(\bm y_k)
={}&P_{\max}
\left[1-\alpha_c(\bm y_k)_1\right]
\left[1-\gamma_T\bigl((\bm y_k)_2-T_{\mathrm{ref}}\bigr)\right]
\\
&\times
\left[1+\gamma_w(\bm y_k)_3\right]
\left[1-\gamma_h(\bm y_k)_4\right].
\end{aligned}
\label{eq:numerical_phase3_solar_power}
\end{equation}
The parameters are $P_{\max}=1$, $\alpha_c=0.85$,
$T_{\mathrm{ref}}=25\,^{\circ}\mathrm C$,
$\gamma_T=0.004\,^{\circ}\mathrm C^{-1}$,
$\gamma_w=0.015\,\mathrm{s\,m^{-1}}$, and $\gamma_h=0.20$.
The power and attenuation coefficients are dimensionless.

The factors represent cloud attenuation, temperature-dependent conversion,
wind cooling, and humidity attenuation. The quantity is constructed as
$q_k=Q_{\mathrm I}(\bm y_k)$, with $\xi_k=0$ in
\eqref{eq:phase3_qoi_samples}. The experiment therefore evaluates a
prescribed weather-dependent response. Its past samples provide the
autoregressive history, while measured channels drive identification
and reconstructed Phase~II channels drive forecasting.

\subsubsection{Correlation-Primary Identification of the $Q_{\mathrm I}$ NLARX Model}
\label{subsubsec:numerical_phase3_identification}

A scalar NLARX model was identified on each observation horizon
$\mathcal T_{\mathrm{obs}}^i$ using the admissible structure family
\begin{equation}
    \mathfrak B_Q
    =
    \left\{
      (n_a^Q,n_b^Q,n_k^Q):
      n_a^Q\in\{1,\ldots,12\},\;
      n_b^Q\in\{1,\ldots,6\},\;
      n_k^Q\in\{0,\ldots,3\}
    \right\}.
    \label{eq:numerical_phase3_candidate_family}
\end{equation}
The ridge parameter was $\lambda_Q=10^{-6}$, and coefficients with magnitude below
$\tau_Q=10^{-8}$ were assigned a zero value. Selection follows the correlation-primary Pareto procedure in
Section~\ref{subsec:phase_forecasting}.
Table~\ref{tab:phase3_selected_power_models} reports the selected models
and experiment-specific windows. The very small identification errors
concern reproduction of the constructed response on these windows;
predictive performance is evaluated separately using reconstructed
channel inputs over the forecast intervals.

\begin{table}[!htbp]
\caption{Experiment-specific Phase~III  settings and selected
correlation-primary NLARX results for the solar-plant power quantity of
interest $Q_{\mathrm I}$.}
\label{tab:phase3_selected_power_models}%
\begingroup
\scriptsize
\setlength{\tabcolsep}{0.5pt}
\begin{tabular*}{\textwidth}{@{\extracolsep\fill}lccccc}
\toprule
Reported quantity & Symbol & $\mathcal T_{\mathrm{obs}}^1$
& $\mathcal T_{\mathrm{obs}}^2$ & $\mathcal T_{\mathrm{obs}}^3$
& $\mathcal T_{\mathrm{obs}}^4$ \\
\midrule
Finite observation horizon
& $\mathcal T_{\mathrm{obs}}^i$
& $[1,287]$ & $[1,359]$ & $[1,309]$ & $[1,299]$ \\
Paired identification samples
& $N_Q$ & 287 & 359 & 309 & 299 \\
Finite forecast horizon
& $\mathcal T_{\mathrm{for}}^i$
& $[288,335]$ & $[360,431]$ & $[310,409]$ & $[300,449]$ \\
Phase~II calibration window
& $\mathcal T_H^i$
& $[288,388]$ & $[360,460]$ & $[310,460]$ & $[300,450]$ \\
Forecast samples
& $N_{\mathrm f}$ & 48 & 72 & 100 & 150 \\
Selected order triple
& $\bm\beta_Q^\star$
& $(7,2,0)$ & $(12,2,0)$ & $(9,2,0)$ & $(12,4,0)$ \\
Correlation loss
& $f_1^Q(\bm\beta_Q^\star)$
& $3.1641\times 10^{-14}$ & $1.3212\times 10^{-14}$ & $2.498\times 10^{-14}$ & $4.5408\times 10^{-14}$ \\
Relative amplitude error
& $f_2^Q(\bm\beta_Q^\star)$
& $1.7339\times 10^{-7}$ & $1.0956\times 10^{-7}$ & $1.5385\times 10^{-7}$ & $2.0868\times 10^{-7}$ \\
Residual persistence
& $f_3^Q(\bm\beta_Q^\star)$
& $0.0022$ & $4.0479\times 10^{-3}$ & $3.8305\times 10^{-3}$ & $4.5837\times 10^{-3}$ \\
Parameter magnitude
& $f_4^Q(\bm\beta_Q^\star)$
& $0.4892$ & $0.4910$ & $0.4900$ & $0.4896$ \\
Final decision score
& $S_Q(\bm\beta_Q^\star)$
& $5.003\times 10^{-3}$ & $5.113\times 10^{-3}$ & $5.092\times 10^{-3}$ & $5.1255\times 10^{-3}$ \\
\bottomrule
\end{tabular*}
\endgroup
\end{table}

\subsubsection{Finite-Horizon Forecasts of Quantity-of-Interest }
\label{subsubsec:numerical_phase3_forecasts}

The four experiments use $N_{\mathrm f}=48$, $72$, $100$, and $150$
forecast steps. Their observation, calibration, and forecast windows
are listed in Table~\ref{tab:phase3_selected_power_models} and satisfy
\eqref{eq:finite_temporal_set_relations}.
For each experiment, the selected quantity model is driven by the
Phase~II reconstructed channels on
$\mathcal T_{\mathrm{for}}^i\subseteq\mathcal T_H^i$.

Let $\bm q_{\mathrm{for}}$ and $\widehat{\bm q}_{\mathrm{for}}$
collect the validation and forecast values, respectively.
Temporal agreement and relative error are evaluated by
\begin{equation}
\begin{aligned}
R_Q^{(\mathrm{for})}
&=\varrho_{N_{\mathrm f}}(\bm q_{\mathrm{for}},
                         \widehat{\bm q}_{\mathrm{for}}),\\
e_{Q,\mathrm{rel}}^{(\mathrm{for})}
&=\frac{\|\bm q_{\mathrm{for}}-\widehat{\bm q}_{\mathrm{for}}\|_2}
        {\|\bm q_{\mathrm{for}}\|_2}.
\label{eq:numerical_phase3_forecast_metrics}
\end{aligned}
\end{equation}
The Pearson coefficient is defined in \eqref{eq:finite_pearson_correlation}.
Table~\ref{tab:phase3_forecast_accuracy} gives these metrics and the
identification-to-forecast ratio.

\begin{table}[!htbp]
\caption{Phase~III finite-horizon solar-plant power forecasting results for
the four experiment-specific forecast horizons.}
\label{tab:phase3_forecast_accuracy}%
\begingroup
\scriptsize
\setlength{\tabcolsep}{2pt}
\begin{tabular*}{\textwidth}{@{\extracolsep\fill}lccccc}
\toprule
Reported quantity & Symbol & $\mathcal T_{\mathrm{for}}^1$
& $\mathcal T_{\mathrm{for}}^2$ & $\mathcal T_{\mathrm{for}}^3$
& $\mathcal T_{\mathrm{for}}^4$ \\
\midrule
Forecasted samples
& $N_{\mathrm f}$ & 48 & 72 & 100 & 150 \\
Identification-to-forecast ratio
& $\mathcal R_{Q/\mathrm{for}}$
& $5.9792$ & $4.9861$ & $3.0900$ & $1.9933$ \\
Forecast correlation
& $R_Q^{(\mathrm{for})}$
& $0.9975$ & $0.9949$ & $0.9933$ & $0.9923$ \\
Forecast relative error
& $e_{Q,\mathrm{rel}}^{(\mathrm{for})}$
& $0.0863$ & $0.0922$ & $0.0987$ & $0.1073$ \\
Phase~III execution time
& -- & 3.7901 s & 2.2284 s & 2.1479 s & 2.1292 s \\
\bottomrule
\end{tabular*}
\endgroup
\end{table}

Finite-horizon forecasts of normalized solar-plant power over $48$, $72$,
$100$, and $150$ hourly steps are presented in
Figs.~\ref{fig:phase3_power_forecast_48h},
\ref{fig:phase3_power_forecast_72h},
\ref{fig:phase3_power_forecast_100h}, and
\ref{fig:phase3_power_forecast_150h}, respectively.

\begin{figure}[!htbp]
    \centering
    \includegraphics[width=0.98\textwidth]
    {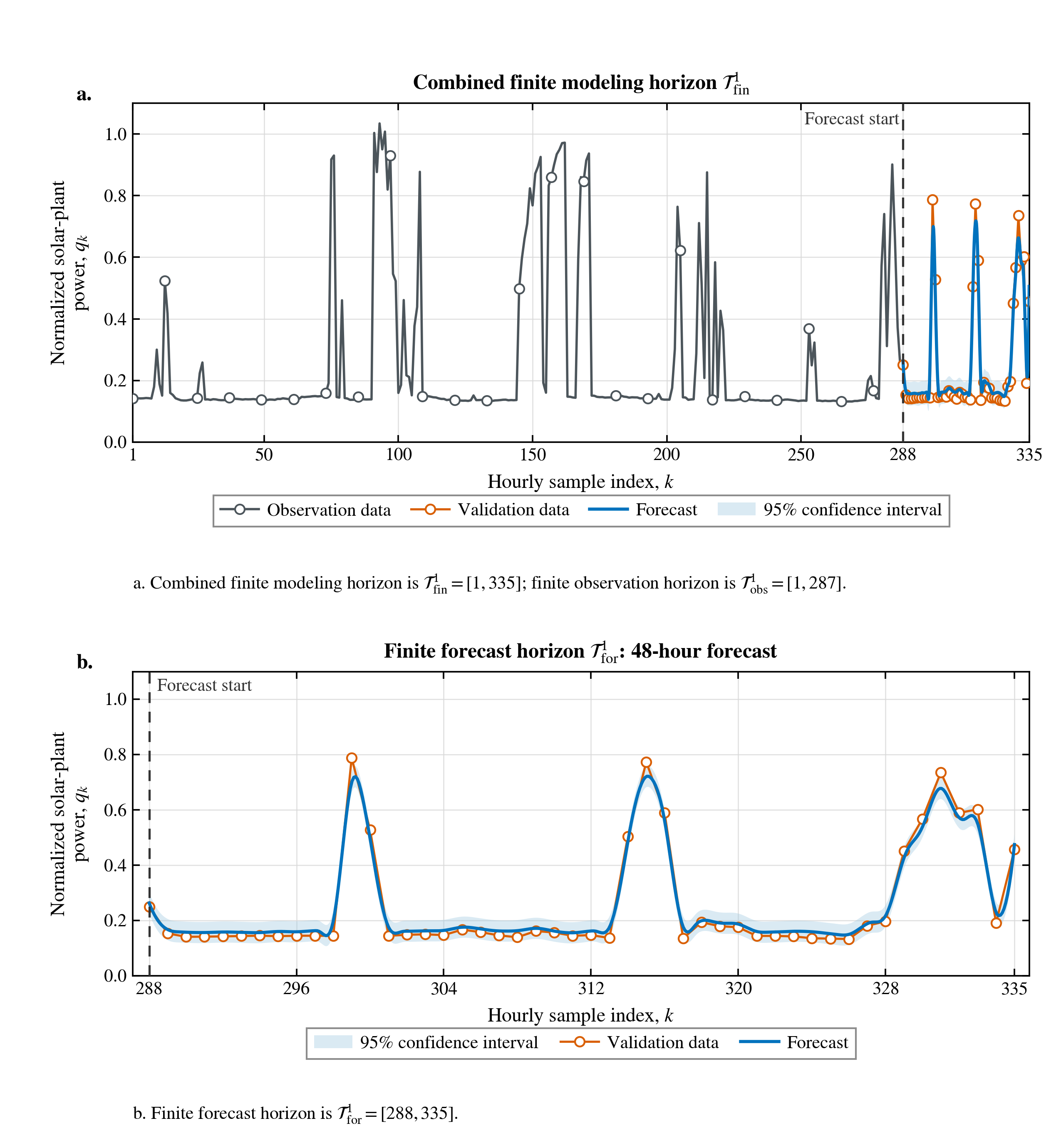}
    \caption{Phase~III forecast of normalized solar-plant power over 48 hourly steps.\\
a. Observation and forecast trajectories, separated by the forecast-start
line at $k=288$.\\
b. Validation data, recursive forecast, and $95\%$ confidence interval on
$\mathcal T_{\mathrm{for}}^1=[288,335]$.
The observation and calibration windows are listed in
Table~\ref{tab:phase3_selected_power_models}.}
    \label{fig:phase3_power_forecast_48h}
\end{figure}
\begin{figure}[!htbp]
    \centering
    \includegraphics[width=0.98\textwidth]
    {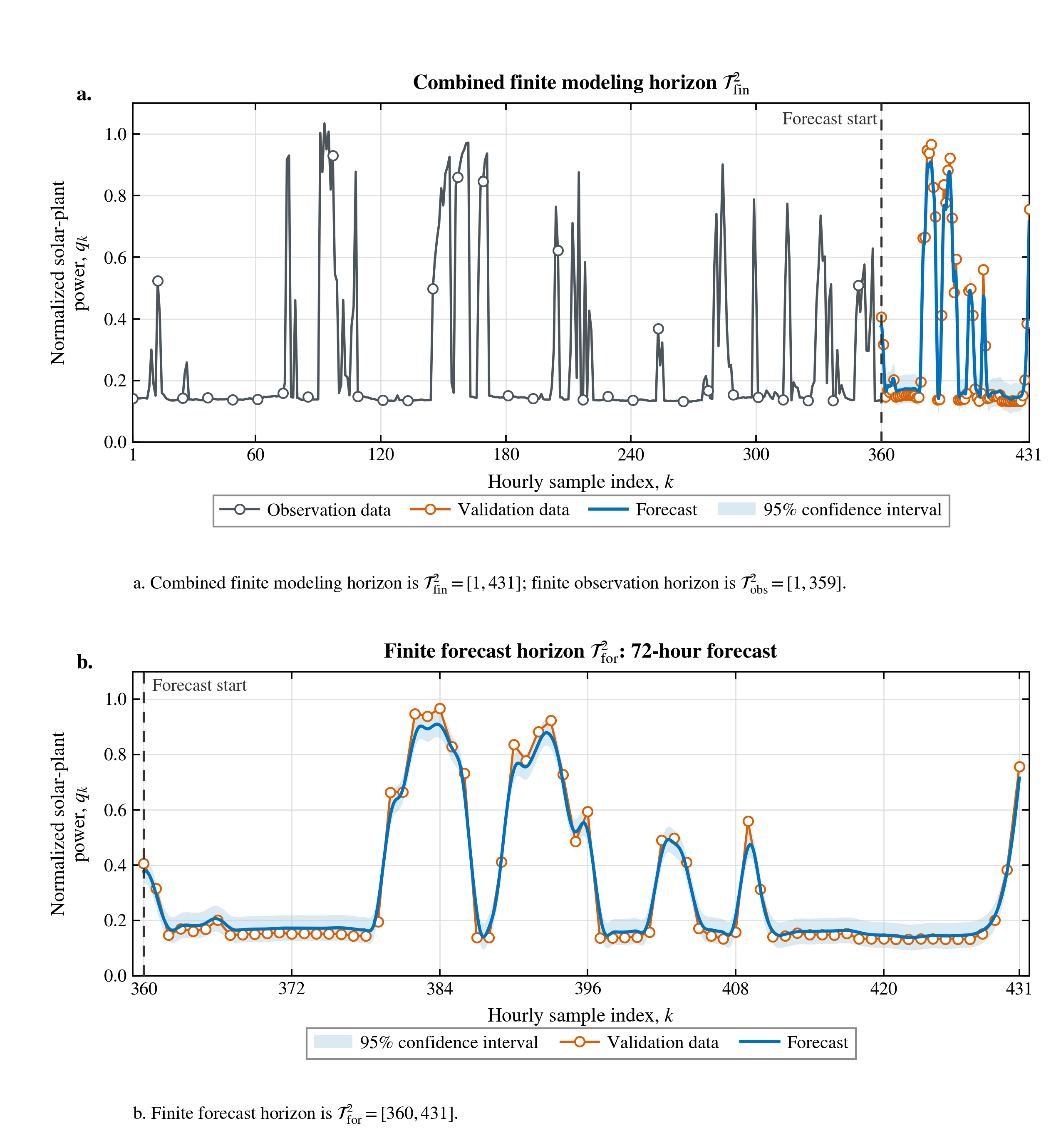}
    \caption{Phase~III forecast of normalized solar-plant power over 72 hourly steps.\\
a. Observation and forecast trajectories, separated by the forecast-start
line at $k=360$.\\
b. Validation data, recursive forecast, and $95\%$ confidence interval on
$\mathcal T_{\mathrm{for}}^2=[360,431]$.
The observation and calibration windows are listed in
Table~\ref{tab:phase3_selected_power_models}.}
    \label{fig:phase3_power_forecast_72h}
\end{figure}
\begin{figure}[!htbp]
\centering
\includegraphics[width=0.98\textwidth]
{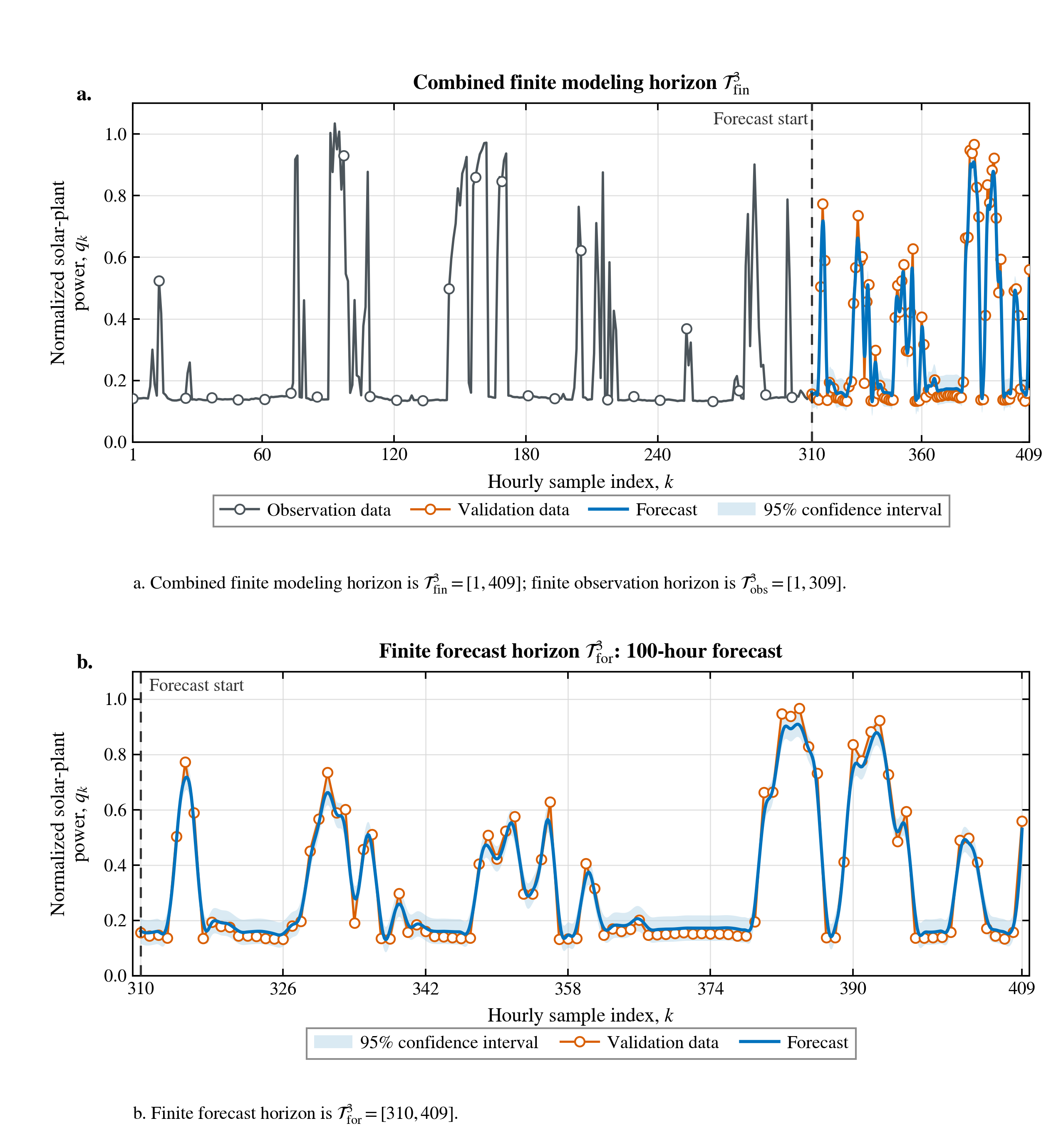}
\caption{Phase~III forecast of normalized solar-plant power over 100 hourly steps.\\
a. Observation and forecast trajectories, separated by the forecast-start
line at $k=310$.\\
b. Validation data, recursive forecast, and $95\%$ confidence interval on
$\mathcal T_{\mathrm{for}}^3=[310,409]$.
The observation and calibration windows are listed in
Table~\ref{tab:phase3_selected_power_models}.}
\label{fig:phase3_power_forecast_100h}
\end{figure}
\begin{figure}[!htbp]
\centering
\includegraphics[width=0.98\textwidth]
{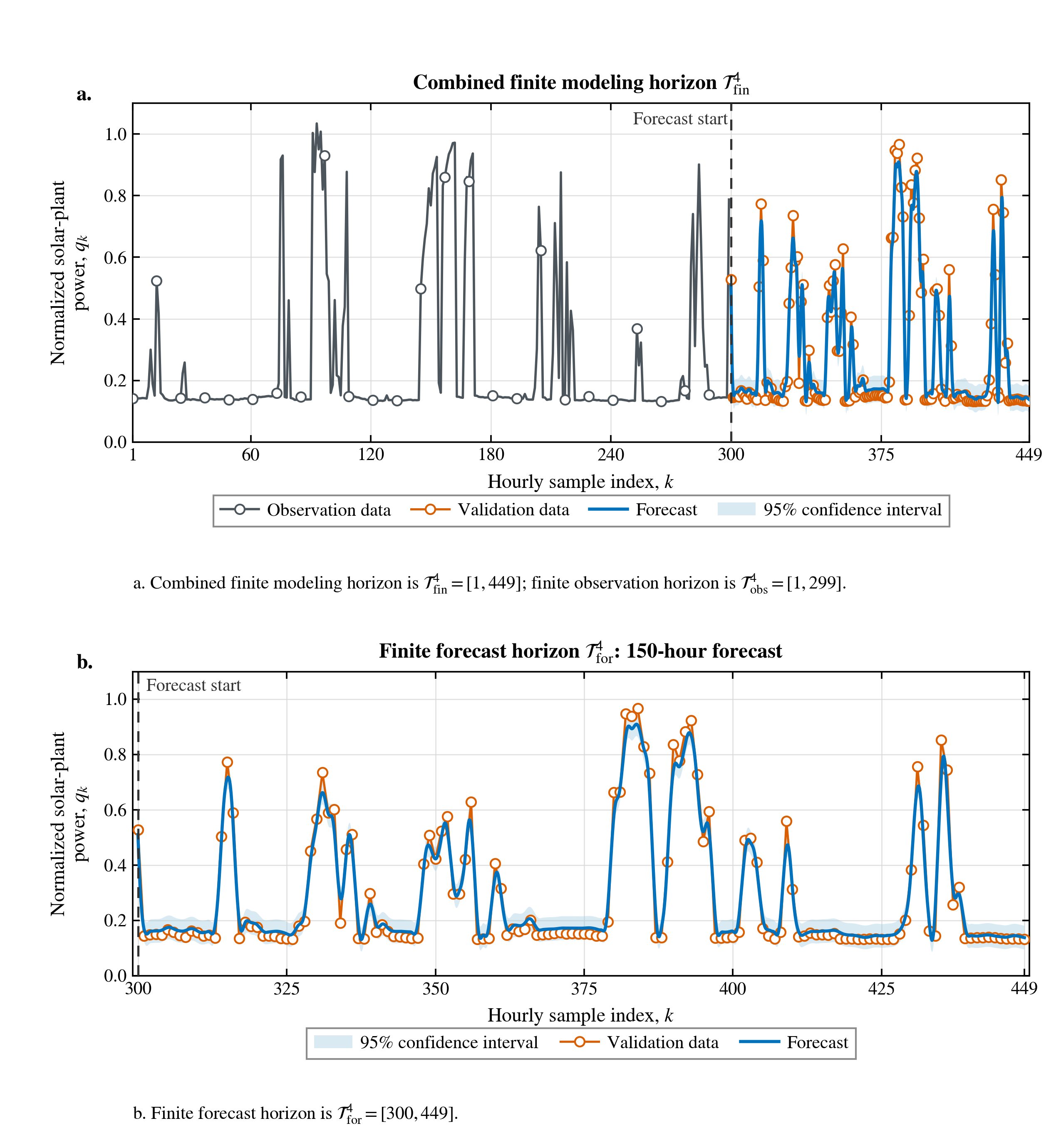}
\caption{Phase~III forecast of normalized solar-plant power over 150 hourly steps.\\
a. Observation and forecast trajectories, separated by the forecast-start
line at $k=300$.\\
b. Validation data, recursive forecast, and $95\%$ confidence interval on
$\mathcal T_{\mathrm{for}}^4=[300,449]$.
The observation and calibration windows are listed in
Table~\ref{tab:phase3_selected_power_models}.}
\label{fig:phase3_power_forecast_150h}
\end{figure}

All four forecasts retain correlations above $0.99$
($0.9923$--$0.9975$), preserving temporal shape through the $150$-hour
recursion. Relative error rises from $0.0863$ at $48$ hours to $0.1073$
at $150$ hours, while $\mathcal R_{Q/\mathrm{for}}$ falls from $5.9792$
to $1.9933$. The experiments also change the observation and calibration
windows, so this trend alone cannot isolate the effects of recursion
length or identification-to-forecast ratio.

The prescribed response combines the channels with different sensitivities,
so strong quantity correlation alone cannot establish uniform channel
reconstruction accuracy. Attributing the quantity error to particular
drivers would require a separate sensitivity analysis.
The four experiments demonstrate the composed procedure within one
finite-record setting. Comparisons with direct NLARX prediction and
alternative modal representations would be needed to isolate the benefit
of each phase; independent plant-output measurements would extend the
validation beyond the constructed response.

\section{Conclusion}\label{conclusion}

This paper developed a unified finite-horizon mathematical framework
linking measurement decomposition, nonlinear identification, physical
reconstruction, and channel-driven forecasting of a quantity of interest.
Hankel--Koopman Finite-Horizon Energy Decomposition (HKFED) constructs
orthogonal energy modes and ranks their linked modal triplets by
finite-horizon energy and persistence. In Phase~II, an inverse-calibrated
multi-output NLARX model identifies the retained coefficient dynamics,
with candidate models assessed after recursive simulation and
physical-space reconstruction. The resulting model-generated channel
trajectories supply the inputs to the recursive quantity-of-interest
model in Phase~III.

The numerical study examines this composition within a prescribed
calibration arrangement. The channel model is identified and assessed
on a finite calibration window. The quantity-of-interest model is
identified separately on an earlier observation window and advanced
recursively over a forecast interval contained within the channel-calibration
window. This experimental-design choice permits evaluation of the
coupled models over a common finite interval. The reported accuracy
therefore characterizes the composed mathematical model conditional
on the channel dynamics identified within that calibration window.

Each phase is supported by mathematical analysis under stated assumptions.
Theorem~\ref{thm:finite_horizon_hk_energy} establishes the additive energy
decomposition and the energy-ranking criterion within the fixed modal basis.
Theorem~\ref{thm:nlarx_existence_uniqueness_inverse_stability} establishes
existence, uniqueness, and inverse stability of the auxiliary continuous
inverse-calibration solution.
Theorem~\ref{thm:phase3_composed_forecast_stability} quantifies
finite-horizon sensitivity to perturbations in the initial quantity
history and channel inputs for a fixed selected forecasting model.

For the constructed solar-power quantity of interest, the four
experiments yield forecast correlations of $0.9923$--$0.9975$ and
relative errors of $0.0863$--$0.1073$. These results demonstrate strong
temporal agreement while retaining measurable amplitude discrepancies.
They illustrate how the identified channel dynamics and the
quantity-of-interest model operate together within the stated
calibration arrangement.

The choices of delay depth, model families, temporal windows, and
quantity of interest allow adaptation to other coupled processes.
Future work will examine comparative assessment of the three phases
and broader nonlinear applications, particularly those with
two-dimensional spatial domains.

\section*{Declarations}

\subsection*{Author Contributions}

D.A.B. conceived the study, developed the methodology and mathematical
analysis, wrote the original manuscript, reviewed and edited the manuscript,
developed the software, and performed the validation. M.T. conducted the
investigation, curated the data, contributed engineering expertise in solar
photovoltaic systems, contributed to the software development, and performed
the validation. Both authors reviewed and approved the final manuscript.

\subsection*{Funding}

This research did not receive any specific grant from funding agencies in the public, commercial, or not-for-profit sectors.

\bibliographystyle{elsarticle-num}

\end{document}